\documentclass[a4paper,11pt]{article}

\author{Cecilia Holmgren\footnote{Department of Mathematics, Stockholm University, 114 18 Stockholm, Sweden. Supported in part by the Swedish Research Council.}
	\and
	Svante Janson\footnote{Department of Mathematics, Uppsala University, SE-75310 Uppsala, Sweden. Supported in part by the Knut and Alice Wallenberg Foundation.}}

\title{Asymptotic distribution of two-protected nodes in ternary search trees}

\date{March 21, 2014}

\usepackage{graphicx}

\usepackage{bbold}
\usepackage{mathrsfs}
\usepackage{mathrsfs}
\usepackage{amsmath}
\usepackage{amsfonts}
\usepackage{amsthm}
\usepackage{amssymb, graphicx, epsfig}
\usepackage{times}
\usepackage{hyperref}

\usepackage{latexsym}

\numberwithin{equation}{section}

\hypersetup{
    bookmarks=true,         % show bookmarks bar?
    unicode=false,          % non-Latin characters in Acrobat’s bookmarks
    pdftoolbar=true,        % show Acrobat’s toolbar?
    pdfmenubar=true,        % show Acrobat’s menu?
    pdffitwindow=true,      % page fit to window when opened
    pdftitle={My title},    % title
    pdfauthor={Author},     % author
    pdfsubject={Subject},   % subject of the document
    pdfnewwindow=true,      % links in new window
    pdfkeywords={keywords}, % list of keywords
    colorlinks=true,       % false: boxed links; true: colored links
    linkcolor=blue,          % color of internal links
    citecolor=blue,        % color of links to bibliography
    filecolor=blue,      % color of file links
    urlcolor=blue           % color of external links
}

\newenvironment{changemargin}[2]{%
\begin{list}{}{%
\setlength{\topsep}{0pt}%
\setlength{\topmargin}{#1}%
\setlength{\listparindent}{\parindent}%
\setlength{\itemindent}{\parindent}%
\setlength{\parsep}{\parskip}%
}%
\item[]}{\end{list}}

\newtheorem{thm}{Theorem}[section] 
\newtheorem{Lemma}[thm]{Lemma}%[section]
\theoremstyle{definition}
\newtheorem{rem}[thm]{Remark}%[section]
\newtheorem{example}[thm]{Example}

\newcommand\REM[1]{{\raggedright\texttt{[#1]}\par\marginal{XXX}}}
\newcommand\ga{\alpha}

\newcommand\gd{\delta}

\newcommand\gl{\lambda}

\newcommand\gs{\sigma}
\newcommand\gss{\sigma^2}

\newcommand{\refT}[1]{Theorem~\ref{#1}}

\newcommand{\refL}[1]{Lemma~\ref{#1}}

\newcommand{\refS}[1]{Section~\ref{#1}}

\newcommand\set[1]{\ensuremath{\{#1\}}}

\newcommand\bigpar[1]{\bigl(#1\bigr)}
\newcommand\Bigpar[1]{\Bigl(#1\Bigr)}

\newcommand\lrpar[1]{\left(#1\right)}

\newcommand\E{\operatorname{\mathbb E{}}}

\newcommand\xfrac[2]{#1/#2}

\newcommand\N{\mathcal N}

\newcommand\Polya{P{\'o}lya}
\newcommand\sumimi{\sum_{i=0}^{m-1}}
\newcommand\iii{^{(i)}}
\renewcommand\Re{\operatorname{Re}}

\begin{document}
\maketitle

%!! Ta bort sidnumreringen fr frsta sidan ocks
%\thispagestyle{empty}

\begin{abstract}
We study protected  nodes in $m$-ary search trees, by putting them in
context of generalised P\'olya urns.  
We show that the number of two-protected nodes (the nodes that are neither
leaves nor parents of leaves) in a random ternary search tree is
asymptotically normal.
The methods apply in principle to $m $-ary search trees with larger $m$
as well, 
%since we can describe the P\'olya urn  for all $ m $, 
although the size of the matrices used in the calculations grow rapidly with
$ m $; we conjecture that the method yields an
asymptotically normal distribution for all $m\le26$.

The one-protected nodes, and their complement, i.e., the leaves, 
are easier to analyze.  
By using a simpler P\'olya urn (that is similar to the one that has earlier
been used to study the total number of nodes in $ m $-ary search trees), we
prove normal limit laws for the number of one-protected nodes and the number
of leaves for all $ m\leq 26 $.  
\end{abstract}



\section{Introduction}\label{intro}
%\subsection {Preliminaries}\label{prel}
There are many recent studies of so-called protected nodes in various
classes of random trees, see
e.g.\ 
\cite{Bona,CheonShapiro,DevroyeJanson,DuProdinger,HolmgrenJanson,MahmoudWard,Mansour}. 
A node is \emph{protected} (more precisely, two-protected) if it is not a
leaf and  
none of its children is a leaf.  

In this paper we consider the number of protected nodes in $ m $-ary
search trees (see \refS{mary} for definitions), by putting them in context
of generalised \Polya{} urns. 
The following result is our main theorem.
We let $\stackrel{d}\rightarrow  $ denote convergence in distribution and 
denote a normal distribution by $ \N(\mu,\sigma^2)$. 

\begin{thm}\label{main}Let $ Z_n $ be the number of protected nodes in a
  ternary search tree with $n$ keys.
Then 
$$
\dfrac{Z_n-\frac{57}{700}n}{\sqrt{n}}\stackrel{d}\longrightarrow 
\N\lrpar{0,\frac{1692302314867}{43692253605000}}.  $$
\end{thm}

For a binary search tree, 
we obtain by the same method a new proof of
the following result, which earlier has been
obtained by different methods, first 
by Mahmoud and Ward \cite{MahmoudWard} (using generating functions), 
and later in \cite{HolmgrenJanson} (using fringe trees).

\begin{thm}\label{binary}Let $Y_n $ be the number of protected nodes in
  a binary search tree.
Then
$$\dfrac{Y_n-\frac{11}{30}n}{\sqrt{n}}
\stackrel{d}\longrightarrow 
\mathcal{N}\lrpar{0,\frac{29}{225}}.  $$
\end{thm}

\begin{rem} \label{expectedvalues} 
Theorems \ref{main} and  \ref{binary} imply that
$\frac{E(Z_n)}{n} \rightarrow \frac{57}{700}$ and $\frac{E(Y_n)}{n}
\rightarrow \frac{11}{30} $.  
Recall that 
necessary and sufficient conditions for $ L^{1} $ convergence of a sequence
$ (X_n) $ of random variables, are that $X_n\stackrel{p}{\rightarrow} X $
(where $ \stackrel{p}\rightarrow  $ denotes convergence in probability) and
that the sequence $ (X_n) $ is uniformly integrable. 
Note that Theorems \ref{main} and  \ref{binary} imply 
$ \frac{Z_n}{n} \stackrel{p}{\rightarrow} \frac{57}{700}$ and
$ \frac{Y_n}{n} \stackrel{p}{\rightarrow} \frac{11}{30}$, respectively,
since if a sequence of random variables converges in distribution to a constant it also converges in probability to that constant.
 Since  $0\le \frac{Z_n}{n} \leq 1$ and  $0\le \frac{Y_n}{n} \leq 1$,
 uniformly integrability for the sequences $ (\frac{Z_n}{n}) $ and $
 (\frac{Y_n}{n}) $ obviously holds. Hence, $ (\frac{Z_n}{n}) $ and $
 (\frac{Y_n}{n}) $ converge in $ L^{1} $ to $ \frac{57}{700} $ and
 $\frac{11}{30} $, respectively; in particular,
 $\frac{E(Z_n)}{n} \rightarrow \frac{57}{700}$ and $\frac{E(Y_n)}{n}
 \rightarrow \frac{11}{30} $.

We conjecture that also the variances (and higher moments) converge in Theorems
\ref{main} and  \ref{binary}.
\end{rem}

The methods apply to larger $m$ too, at least in principle,
see Sections \ref{Polya} and \ref{Sm}.

Similarly, we may consider the one-protected nodes, i.e.\ the non-leaves. 
These are easier to analyze than the two-protected nodes and
using a minor variation of a \Polya{} urn earlier used to
study the total number of nodes
\cite{Mahmoud2,Janson,Mahmoud:Polya}, 
we prove 
in Sections  \ref{Sleaves3} and \ref{oneprotected}
normal limit laws for 
the number of one-protected nodes and the number of leaves
in an $ m $-ary search tree for all $ m\leq 26 $.

\subsection{Protected nodes in $m$-ary search trees described as generalised \Polya{} urns}

\subsubsection{A generalised \Polya{} urn}
A (generalised) \Polya{} urn process is defined as follows,
see e.g.\ \cite{Janson} or \cite{Mahmoud:Polya}.
There are balls of $ q $ types (or colours) $ 1,\dots, q $, and
for each $ n $ a random vector $ X_n=(X_{n,1},\dots,X_{n,q}) $ , where $
X_{n,i} $ is the number of balls of type  
$ i $ in the urn at time $ n $. The urn starts with a given vector $ X_0
$. For each type $ i $, there is an activity (or weight) $ a_i\geq 0 $  and
a random vector $ \xi_i=(\xi_{i1},\dots, \xi_{iq}) $. The urn evolves
according to a Markov process. At each time $ n\geq 1 $, 
one ball is drawn at random from the urn, with the probability of any ball
proportional to its activity. Thus, the drawn ball has  type $ i$  
with probability $ \frac{a_iX_{n-1,i}}{\sum_{j}a_jX_{n-1,j}} $. 
If the drawn ball has type $i$, it is replaced
together with $\Delta X_{n,j}\iii$ balls of type $j$, $j=1,\dots,n$,
where the vector 
$ \Delta X_{n}\iii=(\Delta X_{n,1}\iii,\dots, \Delta X_{n,q}\iii) $
has the same distribution as $ \xi_i$ and is
independent of everything else that has happened so far.  
(We allow $\Delta X_{n,i}\iii=-1$, which means that the drawn ball is
  \emph{not} replaced.)
We let $ A $ denote the $ q\times q $ matrix
\begin{equation}\label{A}
A=(a_j\E\xi_{ji})_{i,j=1}^{q}.   
\end{equation}
The matrix $ A $ with its eigenvalues and
eigenvectors is central for proving limit theorems. 

The basic assumptions in \cite{Janson} are the following.
We say that a type $i$ is \emph{dominating} if every other type $j$ may
appear at some time in an urn started with a single ball of type $i$.
\begin{itemize}
\item[(A1)] $ \xi_{ij}\geq 0 $ for $ j\neq i $ (i.e., balls of other types
  than the drawn ball are   never removed);
$ \xi_{ii} \geq -1$.
\item[(A2)]
$ \E(\xi_{ij}^{2})<\infty $ for all $ i,j\in\{1,\dots,q\} $.
\item [(A3)] 
The largest eigenvalue $ \lambda_1 $ of $ A $ is positive.
\item [(A4)]
The largest eigenvalue $ \lambda_1 $ is simple.
\item [(A5)] 
There exists a dominating type $ i $ with $ X_{0i}>0 $, i.e., we start with
at least one ball of a dominating type.
\item [(A6)]
$ \lambda_1 $ is an eigenvalue of the submatrix of $A$ given by the
  dominating types.
\end{itemize}
Furthermore, \cite{Janson} says that the process becomes \emph{essentially
  extinct} if at some  time there are no balls of any dominating type left.
We will also use the following simlifying assumption.
\begin{itemize}
  \item [(A7)]
With probability 1, the urn never becomes essentially extinct.
\end{itemize}

In the \Polya{} urns used in this paper, it is easily seen (from the
definitions using trees) that every type with non-zero activity is
dominating.  If we remove rows and 
columns corresponding to the types with activity 0 from $A$, then 
 the removed columns are identically 0, so the set of non-zero eigenvalues
 of $A$ is not changed. The remaining matrix is irreducible, and using
the Perron--Frobenius theorem, it is easy to verify 
all conditions (A1)--(A6), see \cite[Lemma 2.1]{Janson}. Furthermore,
in our urns there will always be a ball of positive activity,
so essential extinction is impossible.

Before stating the results that we use, we need some notation. 
With a vector $ v $ we
mean a column vector, and we write $ v' $ for a row vector.  
We denote the transpose of a matrix $ A $ as $ A' $. By an eigenvector of $
A $ we mean a right eigenvector, a left eigenvector is the same as an
eigenvector of the matrix $ A' $. If $ u $ and $ v $ are vectors then $ u'v
$ is a scalar while $ uv' $ is a $ q\times q $ matrix. We also use the
notation $ u\cdot v $ for $ u'v $. 
We let $ \lambda_1 $ denote the largest eigenvalue. 
Let $ a=(a_1,\dots,a_q) $ denote the
(column) vector of activities, and let $ u_1 $ and $ v_1 $ denote left and
right eigenvectors of $A$ corresponding to the 
largest eigenvector $ \lambda_1 $, i.e., vectors satisfying   
\begin{align*}
u_1'A=\lambda_1 u_1',&&& Av_1=\lambda_1 v_1.  
\end{align*}
We assume that $ v_1 $ and $ u_1 $ are normalized such
that 
\begin{align}\label{normalized} a\cdot v_1=a'v_1=v_1'a=1,
&&&
u_1\cdot  v_1=u_1'v_1=v_1'u_1=1,  
\end{align} see \cite[equations (2.2)--(2.3)]{Janson}. 
We write $ v_1=(v_{11},\dots,v_{1q}) $.

We define   
$$P_{\lambda_1} =v_1u_1', $$ and $ P_{I}=I_{q}-P_{\lambda_1} $, where $
I_{q} $ is the $ q\times q $ identity matrix. 
(Thus $ P_{\lambda_1} $ is a one-dimensional projection onto the eigenspace
corresponding to $\gl_1$, 
such that $P_{\gl_1}$ commutes with the matrix $ A $, see \cite[equation
  (2.2)]{Janson}). We 
define the matrices 
\begin{align}\label{Bi} B_i &:=\E(\xi_i\xi_i')\\\label{B}B&:=\sum_{i=1}^{q}v_{1i}a_iB_i\\
\label{Sigma}
\Sigma_I&:=\int_{0}^{\infty}P_{I}e^{sA}Be^{sA^{'}}P_{I}'e^{-\lambda_1s}ds,
\end{align}
where we recall that $ e^{tA}=\sum_{j=0} ^{\infty}\xfrac{t^{j}A^{j}}{j!}$.

It is proved in \cite{Janson} that, under assumptions (A1)--(A7), $X_n$ is
asymptotically normal if
$\Re\lambda\le\lambda_1/2 $ for each eigenvalue $ \lambda\neq \lambda_1 $;
more precisely, if 
$\Re\lambda<\lambda_1/2 $ for each such $ \lambda$, 
then 
$n^{-1/2} (X_n-n\mu)\stackrel{d}\rightarrow \N(0,\Sigma)$ for some $\mu$ and
$\Sigma$. The asymptotic covariance matrix $\Sigma$ may be calculated in
different ways; 
we use the following results from \cite{Janson}, which apply under different
additional assumptions. 

\begin{thm}[{\cite[Theorem 3.22 and Lemma 5.4]{Janson}}]\label{polyathm}
Assume \textup{(A1)--(A7)} and that  we have normalized as in
\eqref{normalized}. 
Also assume that\/ $\Re\lambda<\lambda_1/2 $
for each eigenvalue $ \lambda\neq \lambda_1 $.
Suppose that $
a\cdot\E(\xi_i)=m $ for some $ m>0 $ and every $ i $. 
%Conditioned on essential non-extinction we have 
Then, as $ n\rightarrow \infty $, 
$$n^{-1/2} (X_n-n\mu)\stackrel{d}\rightarrow \N(0,\Sigma),$$
with 
$\mu=\gl_1v_1$ and 
covariance matrix $ \Sigma $ equal to $ m \Sigma_I$, with $ \Sigma_I $
as in \eqref{Sigma}. 
\qed
\end{thm}

\begin{thm}[{\cite[Theorem 3.22 and Lemma 5.3]{Janson}}]\label{simplepolyathm}
Assume \textup{(A1)--(A7)}, and that  we have normalized as in
\eqref{normalized}. 
Also assume that\/ $\Re\lambda<\lambda_1/2 $
for each eigenvalue $ \lambda\neq \lambda_1 $.
If the matrix $ A $ is
diagonalisable, and  
$ \{u_i\}_{i=1}^{q} $ and $\{v_i\}_{i=1}^{q} $ are dual bases of left
and right eigenvectors, respectively, i.e., $ u_iA=\lambda_iu_i $,
 $Av_i=\lambda_iv_i $ and $ u_i\cdot v_j= \delta_{ij} $ (where $\delta_{ij}$
is the Kronecker delta, and the $ \lambda_i $, $ i=1,\dots,q $, do
not have to be distinct). 
%Conditioned on essential non-extinction we have 
Then, as
$ n\rightarrow \infty $, 
$$n^{-1/2} (X_n-n\mu)\stackrel{d}\rightarrow \N(0,\Sigma),$$
with 
$\mu=\lambda_1v_1$ and
covariance matrix $ \Sigma $ equal to 
\begin{align}\label{simpleSigma}
\Sigma=\sum_{j,k=2}^{q} 
\frac{u_j'Bu_k}{\lambda_1-\lambda_j-\lambda_k}v_jv_k',
\end{align}  with the matrix $ B $ as in \eqref{B}.
\qed
\end{thm}

\subsubsection{$M$-ary search trees}\label{mary} 
We recall the definition of $ m $-ary search trees, see 
e.g.\ \cite{Mahmoud:Evolution} or \cite{Drmota}.
An  $ m $-ary search tree, where $ m\geq 2 $, is constructed recursively
from a sequence of $n$ \emph{keys} (numbers). We assume that the keys 
are i.i.d.\ uniform random numbers in $[0,1]$.
(Only the order of the keys matter, so alternatively, we may assume that the
keys form  a uniformly random permutation of $\set{1,\dots,n}$.)
Each node may contain up to $ m-1 $ keys.  
We start with a tree containing just an empty root. 
The first $ m-1 $ keys are put in the root, and are placed in increasing
order from left to right; they divide the set of real numbers into $m$
intervals $J_1,\dots,J_m$.
When the root is full (after the first $ m-1 $ keys are added), it gets $ m$
children that are initially empty, 
and each further key is passed to one of the children
depending on which interval it belongs to; a key in $J_i$ is passed to the 
$i$:th child.
(The binary search tree is the
simplest case where keys  are passed to the left or right child
depending on whether it is larger or smaller than the key in the root.) 
The procedure repeats recursively in the subtrees until all keys are added
to the tree.  
%(Thus, each permutation of $\{1,\dots, n\} $ gives an $ m $-ary search tree
%but it is not unique since some permutations correspond to the same tree.) 

Nodes that contain at least one key are called \emph{internal}, while empty
nodes  are called \emph{external}. We regard the $m$-ary search tree as
consisting only of the internal nodes; the external nodes are places for
potential additions, and are useful when discussing the tree (e.g.\ below),
but are not 
really part of the tree. Thus, a \emph{leaf} is an internal node that has no
internal children, but it may have external children. (It will have external
children if it is full, but not otherwise.)
Similarly, a protected node is an internal node that is not a leaf,
and has no child that is a leaf. (It may have external nodes as children.)

We say that a node with $i\le m-2$ keys has $i+1$ \emph{gaps}, while a full
node has no gaps. It is easily seen that a $m$-ary search tree with $n$ keys
has $n+1$ gaps; the gaps correspond to the intervals of real numbers between
the keys (and $\pm\infty$), 
and a new key  
has the same probability $1/(n+1)$ of belonging to any of the gaps.
Thus the evolution of the $m$-ary search tree may be described by choosing a
gap uniformly at random at each step. Equivalently, the probability that the
next key is added to a node is proportional to the number of gaps at that node.

\Polya{} urns have been used 
in some earlier studies,
e.g.\ \cite{Mahmoud2,Janson}, 
to
describe the number of nodes in $ m $-ary search trees containing $ i $ keys
where $ 0\leq i\leq m-1 $; then a node containing $ i $ keys is called a
node of type $ i $ and thus the generalised \Polya{} urn has $ m $ different
types. It has been shown that for this process, when $ m\leq 26 $ the number
of different types  has an asymptotic multivariate normal distribution, but
this does not hold for larger $ m $. (Since the condition $\Re \gl <
\gl_1/2$ for $\gl\neq\gl_1$ on the eigenvalues of the matrix $A$ in
\eqref{A} holds only if $m\le26$.)
Since the number of nodes in the whole tree is a linear combination of these
numbers, this implies in particular that the  distribution of the
random number of nodes in an $ m $-ary search tree containing $ n $ keys is
asymptotically normal for $ m\leq 26 $. 
In this \Polya{} urn, with one ball representing each node, the activity of
a ball is the number of gaps, i.e., $i+1$ for a ball of type $i\le m-2$, and
$0$ for a ball of type $m-1$.

Alternatively, see \cite{Janson}, we can use a \Polya{} urn where each ball
represents a gap; thus a node with $ i $ keys corresponds to $ i+1 $ balls
for $ 0\leq i\leq m-2 $, and these balls are all given type $i$.
(Full nodes are ignored.)
This is thus an  urn with $ m-1 $ types, all with activities 1.

\subsubsection{Protected nodes and generalised \Polya{} urns}\label{Polya}

We will see that it is possible to use  a generalised \Polya{} urn
also to study protected nodes in an $ m $-ary search tree, 
although the urn consists of quite a few different types.

 \textbf{Description of the Types in the \Polya{} urn.}
Given an $ m $-ary search tree $T$ with $ n $ keys 
together with its external
nodes, 
erase all edges that connect two internal non-leaves.
This yields  a forest of small trees, where (assuming $n\ge m$)
each tree has a root that is a non-leaf in $T$ while all other nodes
are leaves or external nodes in $T$.
We regard these small trees as the balls in our
generalised \Polya{} urn. The type of a ball (tree) is 
the type of the tree as an unordered tree, i.e., up to permutations of the
children. 
The type of a tree in the urn is thus described by the numbers $k_i$,
$i=0,\dots,{m-1}$, of  
children of the root with $i$ keys; each of these children is an external node
($i=0$) or a leaf ($i\ge1$), and it has itself children only when $i=m-1$
when it has $m$ external children; thus the type is uniquely determined
by $k_0,\dots,k_{m-1}$, and we can label the type by $(k_0,\dots,k_{m-1})$.
Since the root of any of the small trees has $m$ children 
(including external ones)
in the original tree $T$, 
we have $\sum_{i=0}^{m-1} k_i \le m$,  
(with the remainder  $m-\sum_{i=0}^{m-1} k_i$
equal to the number of erased edges to children in the original tree $T$
that are non-leaves). 
Furthermore, the case $k_0=m$ is excluded, since the root of the small tree is a
non-leaf in $T$.
The total number of types is thus one less than the number of compositions
of $m$ into $m+1$ non-negative parts, i.e.,
$\binom{2m}{m}-1$.

The activity in the \Polya{} urn of one of these types is the number of gaps
that it contains. The root has no gaps, so a tree with 
type $(k_0,\dots,k_{m-1})$ 
has activity $\sumimi (i+1)k_i$.
Moreover, if we add a new key to a leaf, it is still a leaf, so in the
\Polya{} urn, this corresponds to replacing a tree by another tree where we
have increased by 1 the number of keys of one of the children of the root.
The same holds if we add a key to an external node that is a child of the
root. However, if we add a key to an external node that is a child of a
leaf, then that leaf becomes a non-leaf, so the edge from it to the root is
erased and the tree is split into two (one of which always has the type
$(m-1,1,0,\dots,0)$). See  \refS{S2} for examples.
Note that in general, a small tree may be transformed in several different
ways when we add a new key, depending on which gap it goes into.
Hence, the additions $\xi_i$ in the \Polya{} urn will be random.

A protected node in $T$ is a non-leaf, and is therefore a root in one of the
small trees. Moreover, it must not have any child that is a leaf, so all its
children are external nodes.  Thus, the number of protected nodes in $T$
equals the number of balls in the urn that have types
$(k_0,0,\dots,0)$ with $0\le k_0\le m-1$.

\section{Protected nodes in binary search trees and \Polya{} urns}
\label{S2}
In this section we demonstrate the technique of using the \Polya{} urn
defined above to study the number of protected nodes, 
by applying it to the simplest case $m=2$, the binary search tree. 
This gives us a new proof of \refT{binary};
for earlier proofs, see 
\cite{MahmoudWard} and \cite{HolmgrenJanson}.

For a binary tree, the number of types in the \Polya{} urn defined above is
$\binom42-1=5$. We show the different types in Figure \ref{fig:typesbinary},
with a numbering that will be used below.
(For convenience we omit the external nodes in the figures. We use dotted
lines for edges attached to external nodes.) With our characterization of
the types in Section \ref{Polya}, the types $ i\in\{1,\dots,5\} $ correspond
to $ (0,2) $, $ (1,1) $,  
$ (0,1) $, $ (1,0) $ and $ (0,0) $, respectively.

\begin{figure}[t]
\begin{minipage}[t]{0.17\linewidth}
\centering
\includegraphics[width=\textwidth]{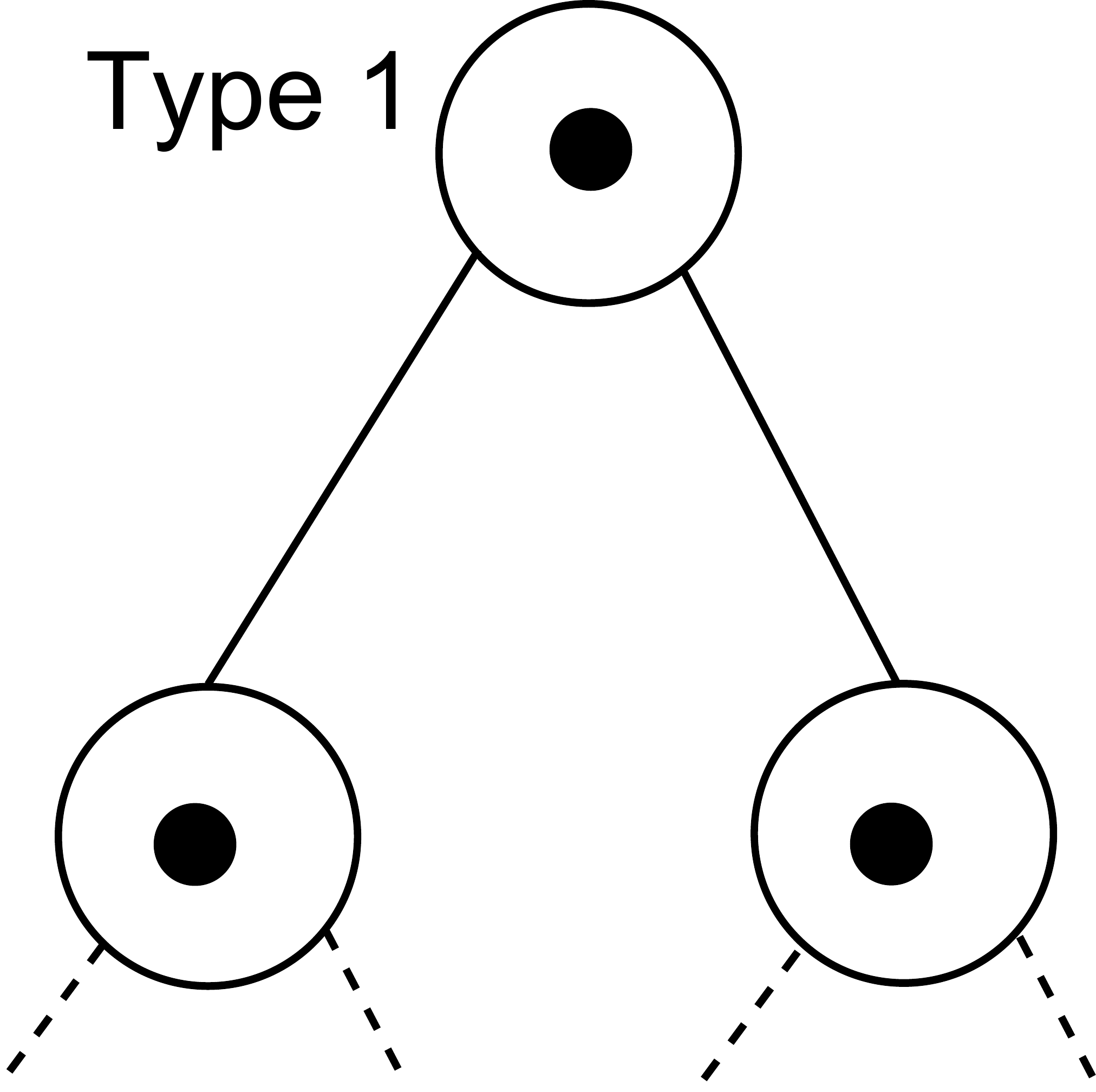}
%\caption{Type 1}
\label{type1b}
\end{minipage}
\hspace{0.2cm}
\begin{minipage}[t]{0.14\linewidth}
\centering
\includegraphics[width=\textwidth]{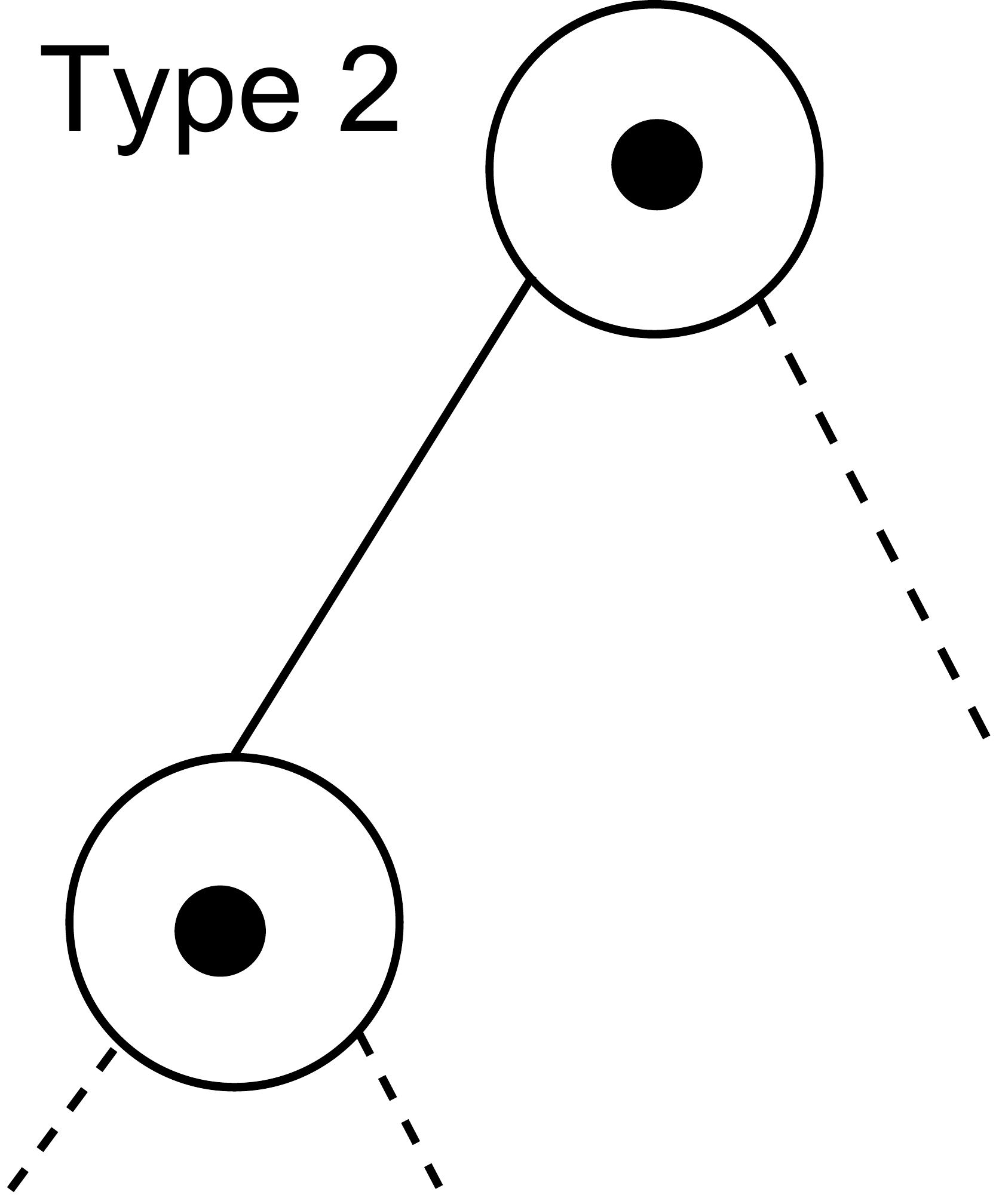}
%\caption{default}
\label{type2b}
\end{minipage}
\hspace{0.2cm}
\begin{minipage}[t]{0.115\linewidth}
\centering
\includegraphics[width=\textwidth]{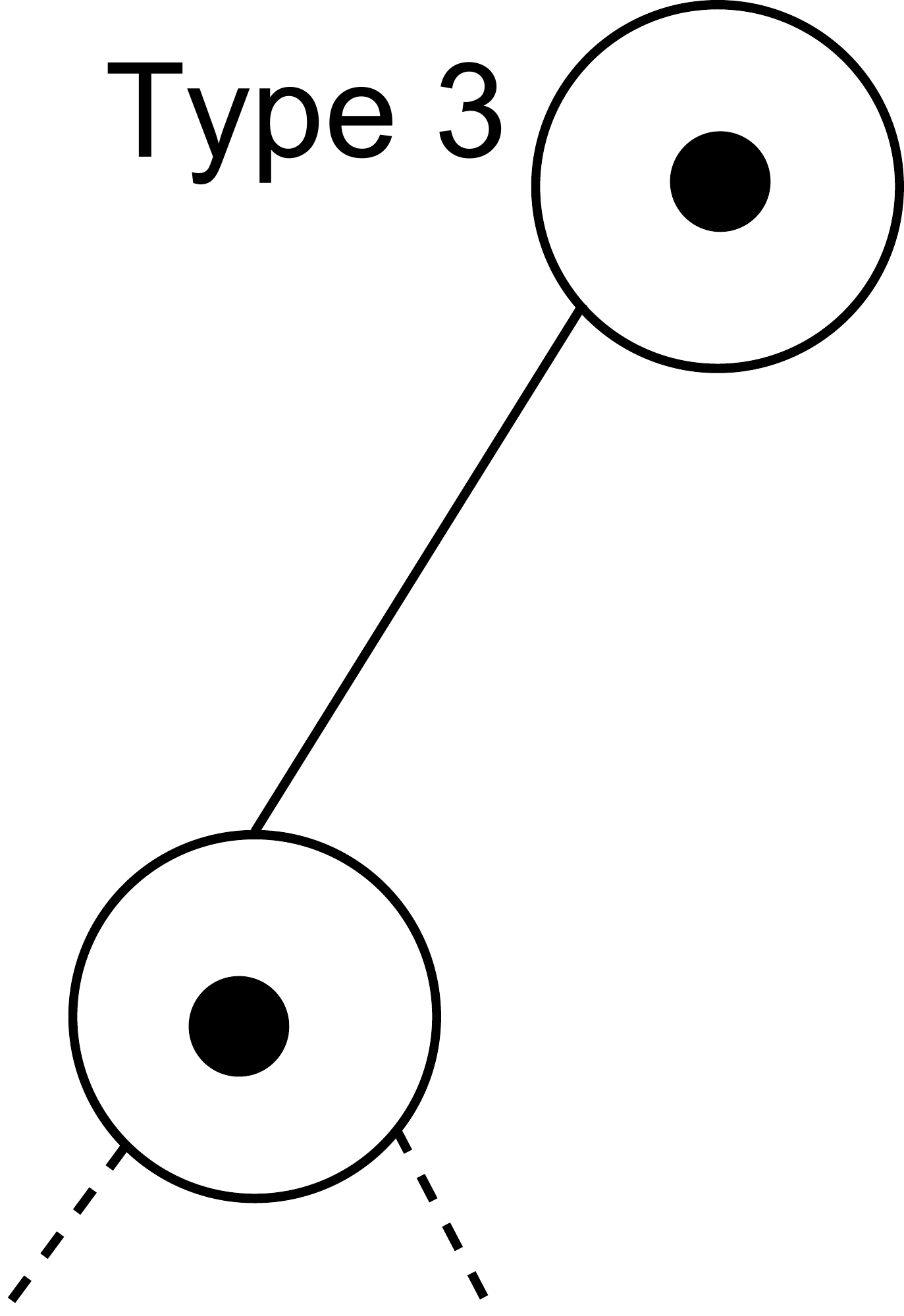}
%\caption{default}
\label{type3b}
\end{minipage}\hspace{0.2cm}
\begin{minipage}[t]{0.17\linewidth}
\centering
\raisebox{5pt}{
\includegraphics[width=\textwidth]{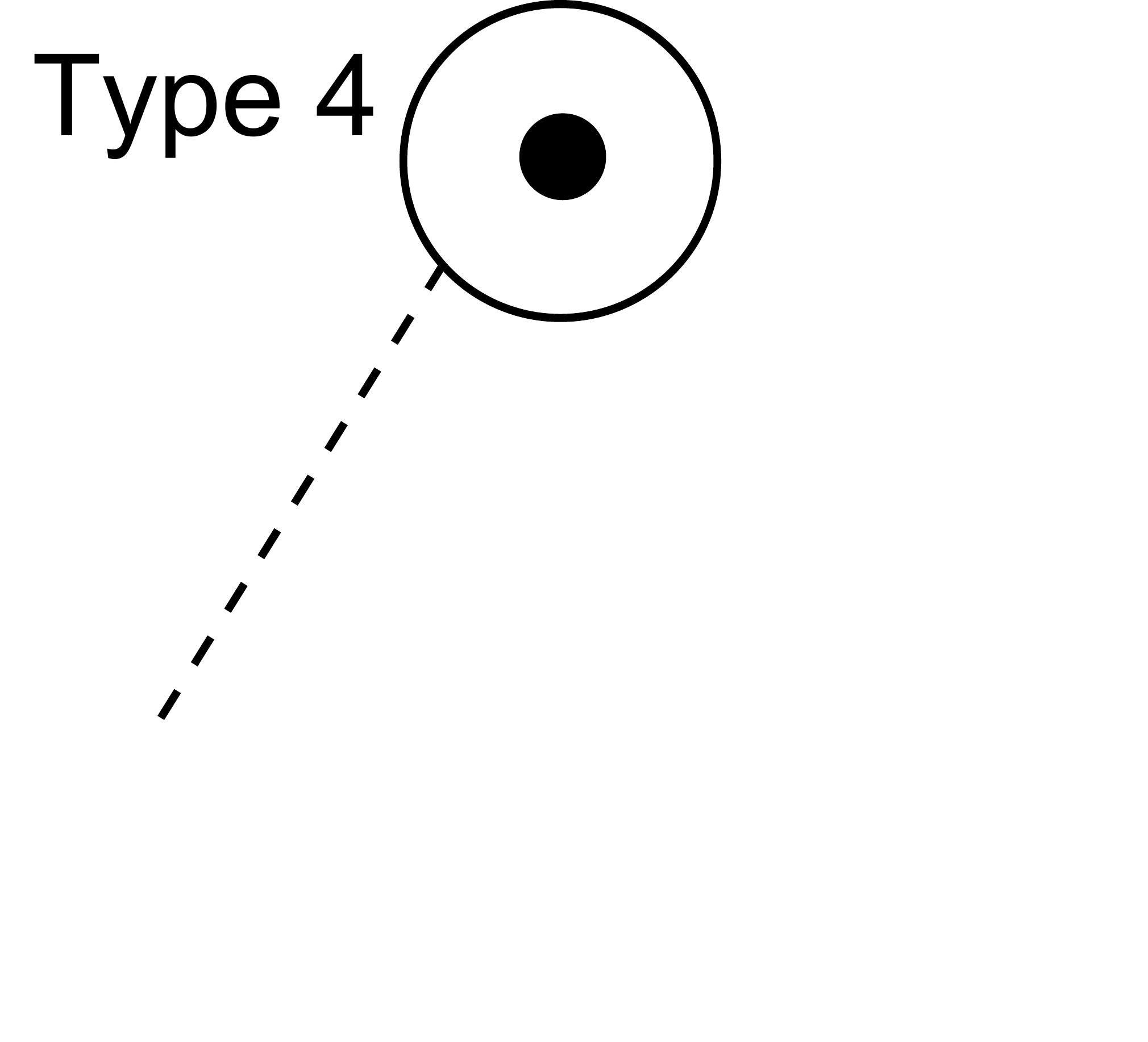}}
%\caption{default}
\label{type4b}
\end{minipage}\hspace{0cm}
\begin{minipage}[t]{0.17\linewidth}
\centering
\raisebox{5pt}{\includegraphics[width=\textwidth]{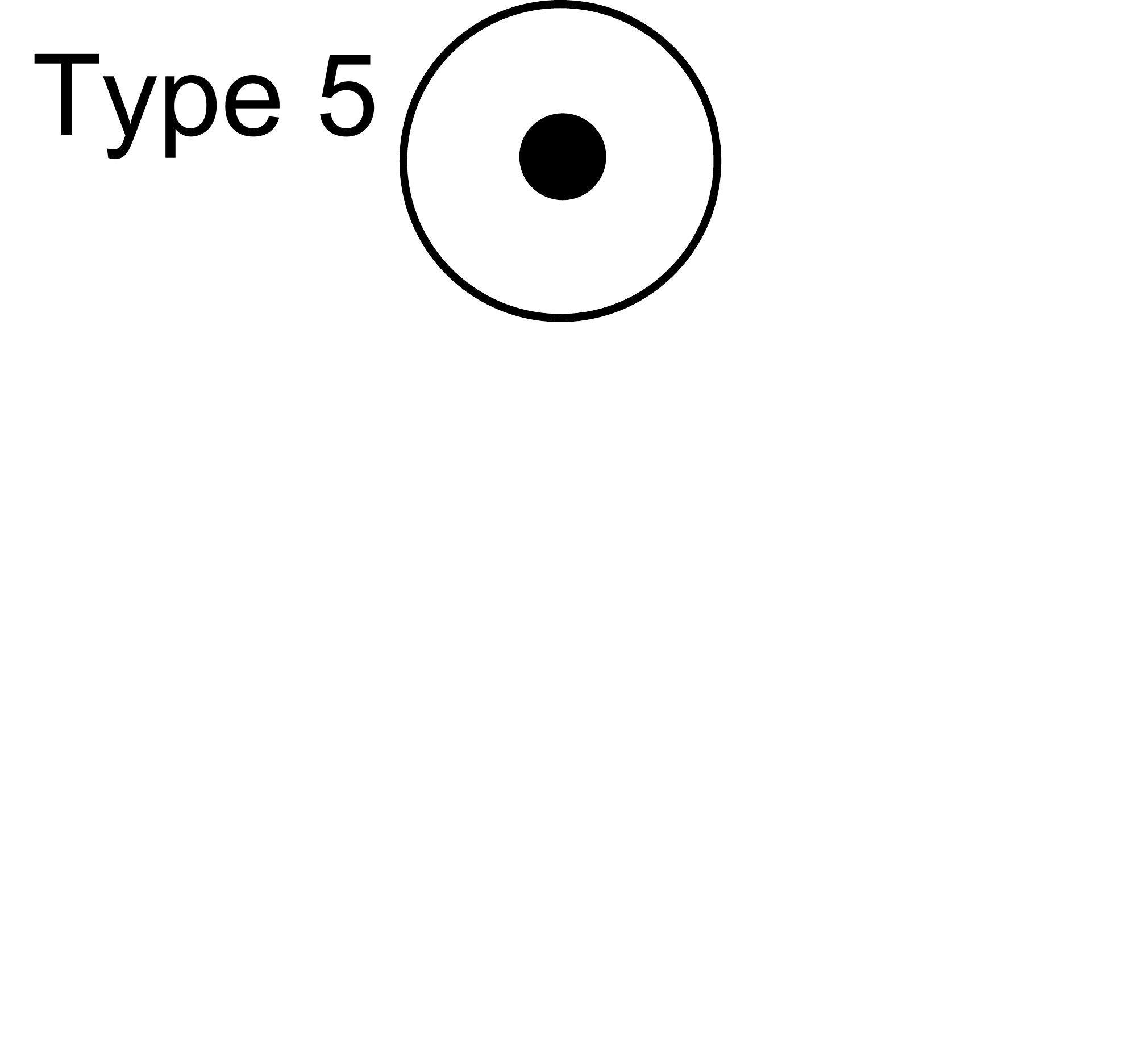}}
\label{type5b}
\end{minipage}
\caption{The different types characterizing protected and unprotected nodes in binary search trees. Type 4 and type 5 are the only ones that include protected nodes.}\label{fig:typesbinary}
\end{figure}

In a binary search tree, each leaf contains one key, so it has two external
children, whereas other internal nodes have either 1 or 0 external children.  
There is one gap at each external node, and no gaps at any internal node.
As explained in Section \ref{mary}, each gap (i.e.\ external node) has
activity 1.  

When a ball is drawn from the urn (i.e., a new key is added to the tree),
as explained in general in \refS{Polya},
a key is either added to an external node that is a
child of the root (we return a ball of another type), or to an external node
that is a child of a leaf (we return two balls).
Figures \ref{add1}--\ref{add4} show the transitions in the \Polya{} urn when
a ball of type $i$ for $ i\in\{1,2,3,4,5\} $ is drawn
(where the types are shown in Figure \ref{fig:typesbinary}), 
so that the drawn
ball is replaced by a new set of balls. (As said above, this set
could depend on which of the nodes in the drawn type the key is added to,
see Figure \ref{add2}.) 
The activities of the different types depend on
their number of gaps; the total activities for the types $ 1,2,3,4,5 $ are $
4,3,2,1,0 $, respectively; thus $ a=(4,3,2,1,0)' $.

\begin{figure}[t]\label{addbinary}
\begin{minipage}[t]{0.4\linewidth}
\centering
\includegraphics[width=\textwidth]{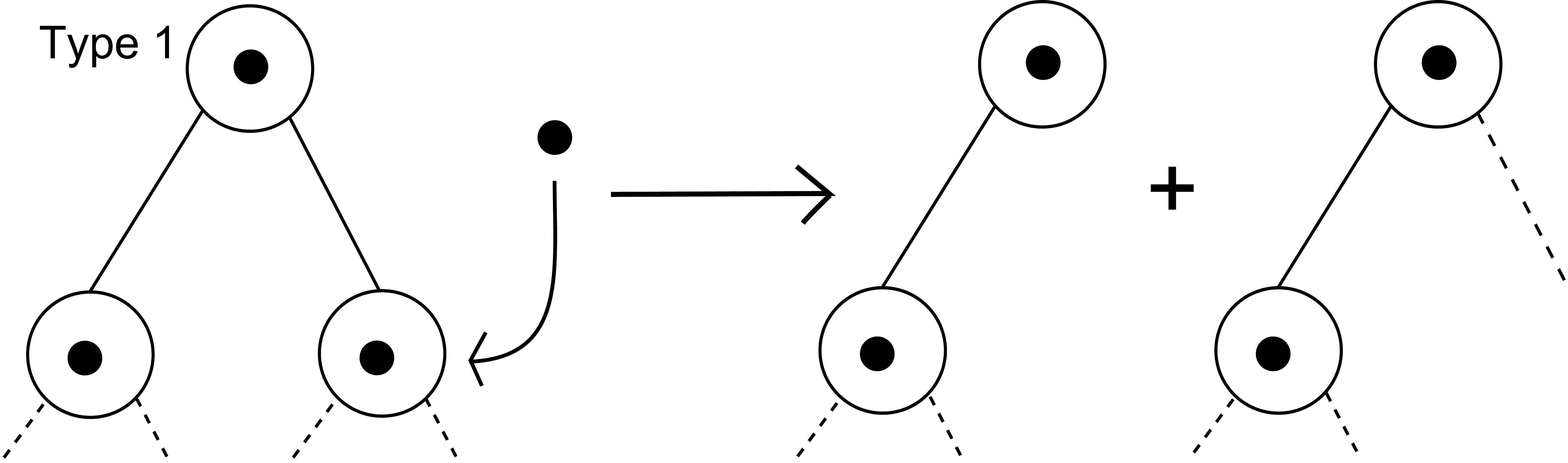}
\caption{Adding a key to type 1.}
\label{add1}
\end{minipage}
\hspace{1cm}
\begin{minipage}[t]{0.65\linewidth}
\centering
\includegraphics[width=\textwidth]{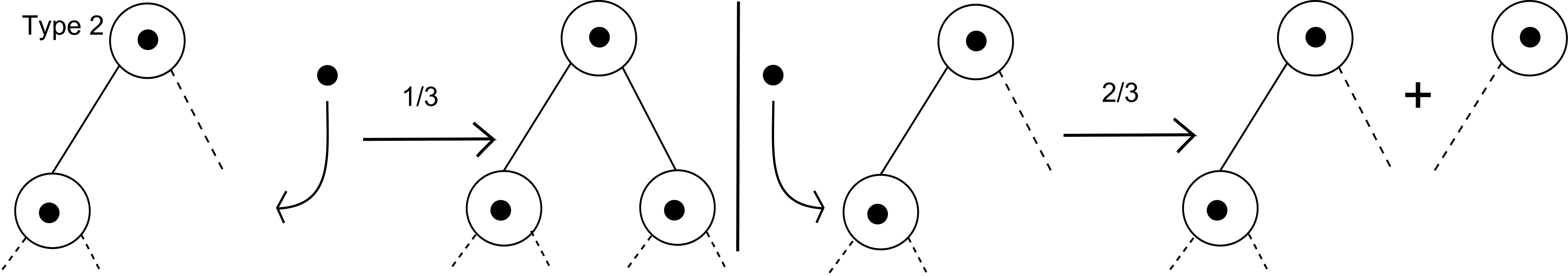}
\caption{Adding a key to type 2.}
\label{add2}\end{minipage}
\begin{minipage}[t]{0.33\linewidth}
\centering
\includegraphics[width=\textwidth]{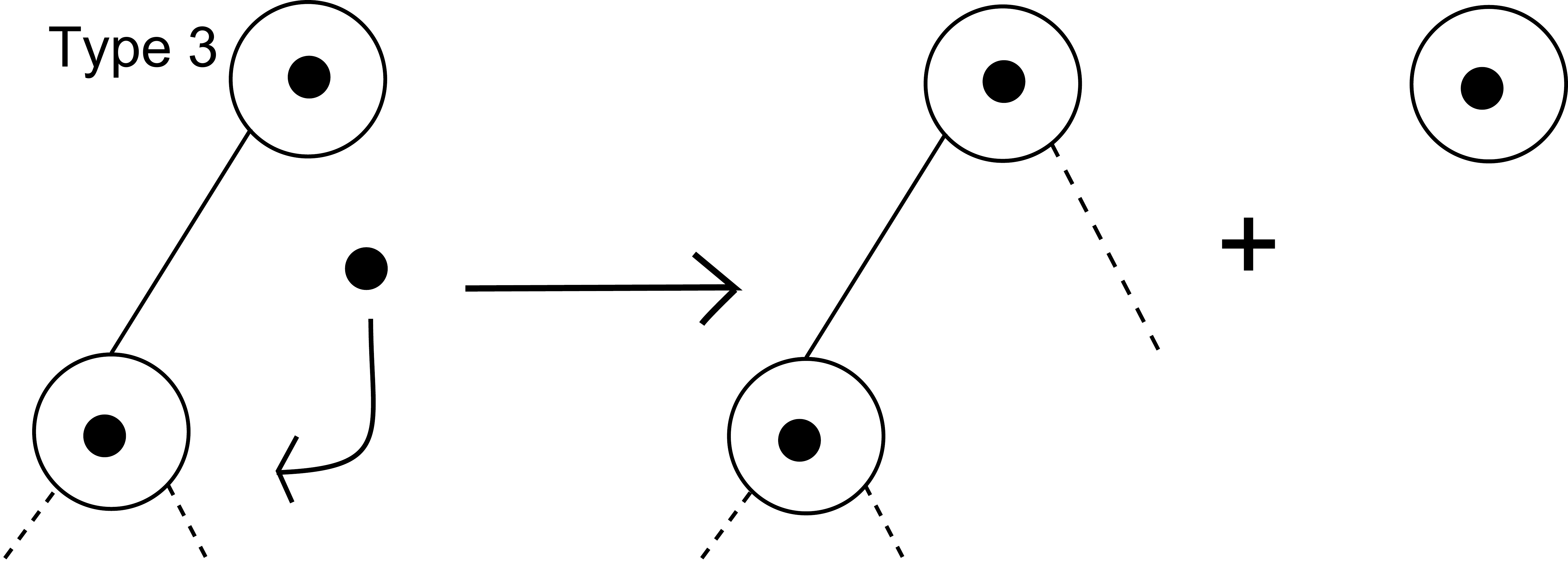}
\caption{Adding a key to type 3}
\label{add3}
\end{minipage}
\hspace{2cm}
\begin{minipage}[t]{0.22\linewidth}
\centering
\includegraphics[width=\textwidth]{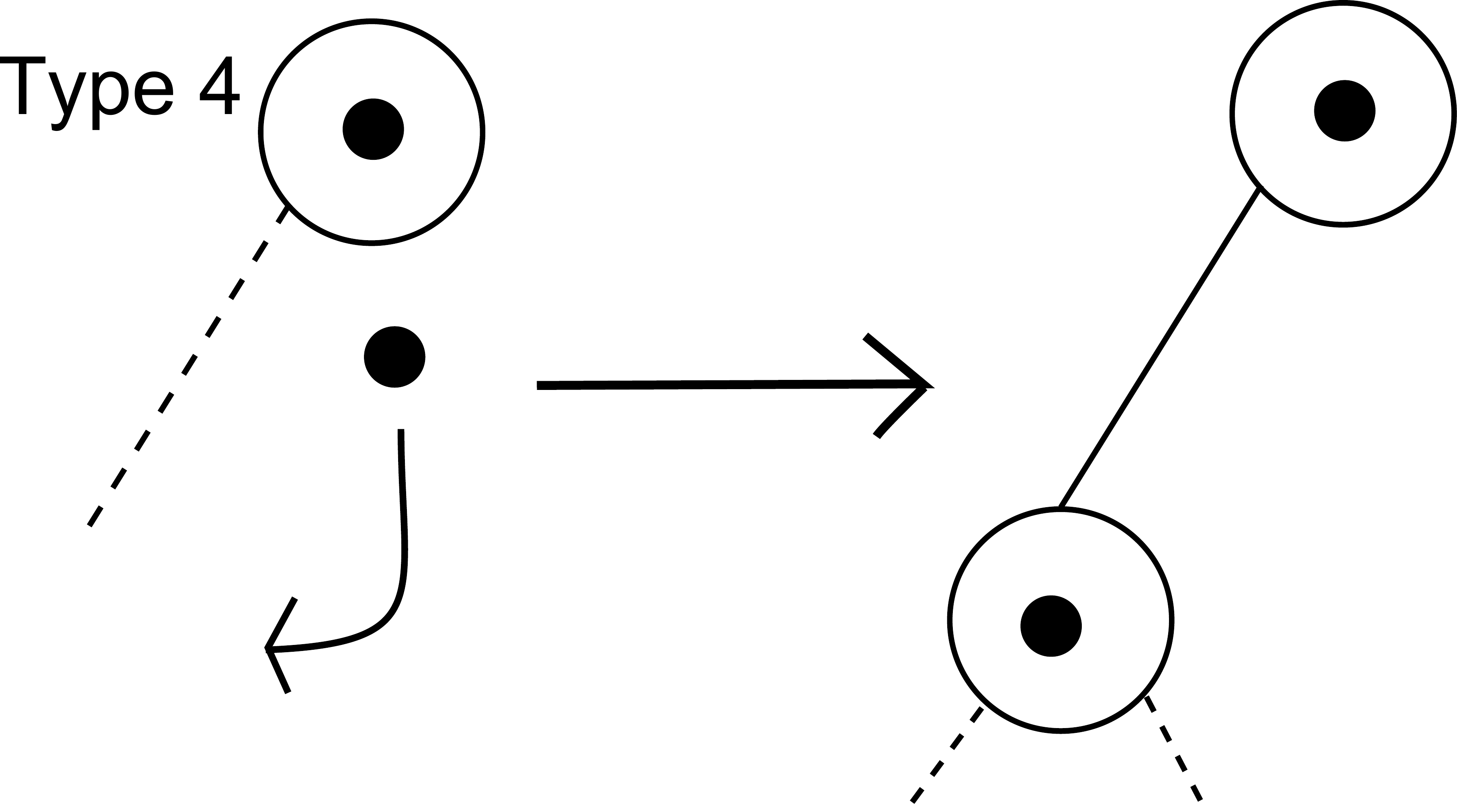}
\caption{Adding a key to type 4}
\label{add4}
\end{minipage}

%\caption{The different types characterizing protected and unprotected nodes in binary search trees.}
\end{figure}

From the transitions that are shown in Figures \ref{add1}--\ref{add4}, we easily obtain  the matrix $A=(a_j\E\xi_{ji})_{i,j=1}^{5}$ in  (\ref{Abin}).

\begin{align}\label{Abin}
A=
\left(
\begin{array}{rrrrr}
 -4 & 1 & 0 & 0 & 0 \\
 4 & -1 & 2 & 0 & 0 \\
 4 & 0 & -2 & 1 & 0 \\
 0 & 2 & 0 & -1 & 0 \\
 0 & 0 & 2 & 0 & 0 \\
\end{array}
\right)
\end{align}
To do the matrix operations in this paper we use Mathematica, but one could
alternatively use e.g.\ Maple.
 
The eigenvalues of $A$ are 
%$ -4,-3,-2,1,0 $. 
$1,0,-2,-3,-4$.
Corresponding
right eigenvectors of $ A $ are:
\begin{align}\label{eigenbin}
\def\+{\phantom{-}}
\frac{1}{30}\begin{pmatrix} %1
  1\\5\\3\\5\\6
\end{pmatrix}
,
\begin{pmatrix} %0
  0\\0\\0\\0\\1
\end{pmatrix}
,
\frac{1}{3}\begin{pmatrix} %-2
  \+1\\ \+2\\ -3\\-4\\\+3
\end{pmatrix}
,
\frac{1}{2}\begin{pmatrix} %-3
  \+1\\ \+1 \\ -3\\-1\\ \+2
\end{pmatrix}
,
\frac{1}{5}\begin{pmatrix} %-4
  \+1\\ \+0 \\-2\\ \+0\\\+1
\end{pmatrix},
\end{align}
and corresponding left eigenvectors of $ A $ are:
\begin{align}\label{eigenbint}
\def\+{\phantom{-}}
\begin{pmatrix} %-3
  4\\ 3 \\ 2\\1\\ 0
\end{pmatrix}
,
\begin{pmatrix} %-4
  -1\\ -1 \\\+0\\ \+0\\\+1
\end{pmatrix},
\begin{pmatrix} %-2
  \+2\\ \+0\\ \+1\\-1\\\+0
\end{pmatrix},
\begin{pmatrix} %0
  -4\\\+1\\-2\\\+1\\\+0
\end{pmatrix},
\begin{pmatrix} %1
  \hskip3pt11\\-3\\\+3\\-1\\\+0
\end{pmatrix}.
\end{align}
Since the eigenvalues for the matrix $ A $ are distinct it follows
automatically that $ u_i\cdot v_j=0 $  for $ i\neq j $. Note that we have
scaled the eigenvectors so that $u_i\cdot v_i=1$ and (\ref{normalized}) 
hold.
Note also that $ u_1 $ is equal to the activity vector $ a$.
This is a consequence of the fact that the total activity always increases by 1
when we draw a ball from the urn, 
and thus $a\cdot \E\xi_i=1$ for each $i$,
see \cite[Lemma 5.4]{Janson}.

It is easy to see that we can apply Theorem \ref{simplepolyathm} for this
generalised \Polya{} urn. Note that it is obvious that the matrix $ A $ is
diagonalisable since all eigenvalues are simple. 
From Theorem \ref{simplepolyathm} we  obtain that $
X_n=(X_{n1},X_{n2},X_{n3},X_{n4},X_{n5}) $, where $ X_{ni} $ is the number
of balls  of type $ i $ (in our case the number of  trees that correspond to
type $ i $ in our forest),
has asymptotically a multivariate normal distribution. 
Let $ Y_n $ be equal to the number of protected nodes in the binary search
tree with $ n $ nodes. Since type 4 and type 5 each contains exactly one
protected node, while the other types contain no protected nodes, 
  $$Y_n=X_{n4}+X_{n5}.$$ Thus, Theorem \ref{simplepolyathm} implies that  
\begin{align}\label{normalY}
n^{-1/2} (Y_n-n\mu_Y)\stackrel{d}\rightarrow \N(0,\sigma^2_Y)
\end{align}
with parameters $ \mu_Y= \mu_4+\mu_5$ and 
\begin{align}\label{protectedbinvariance}\sigma^2_Y=\sigma_{4,4} +\sigma_{4,5}+\sigma_{5,4}+\sigma_{5,5}.
\end{align}
 Since $\lambda_1=1  $, Theorem
\ref{simplepolyathm} implies, using $v_1$ in \eqref{eigenbin},  that 
\begin{align}%\label{expectedbin}
\mu_Y=\mu_4+\mu_5=\frac{5}{30}+\frac{6}{30}=\frac{11}{30}. 
\end{align}
Thus, to show Theorem \ref{binary} it remains to calculate the sum in 
\eqref{protectedbinvariance}.

To calculate the matrix $ B $ in (\ref{B}) we need to calculate  $
B_i=\E(\xi_i\xi_i') $ in (\ref{Bi}). In all cases except for $ B_2 $ these are
deterministic and  equal to $ \xi_i\xi_i' $. We only show how to obtain $
B_2 $ (since the other cases are simpler).  
As shown in Figure \ref{add2} when adding a key to type 2 we can either add
it to the leaf or to the external node. In case we add it to the external
node (which happens with probability 1/3) a node of type 2 is replaced by a
node of type 1; this change corresponds to the column vector 
$(1, -1, 0, 0, 0)'  $. If the key is instead added to the leaf (which happens with probability 2/3) a node of type 2 is replaced by another node of type 2 (the change of type 2 is 0) and a node of type 4; this change corresponds to the column vector $ (0, 0, 0, 1, 0)' $. Hence 
\begin{align}\label{B2}
B_2
&=\tfrac13\cdot (1, -1, 0, 0, 0)' (1, -1, 0, 0, 0)
+\tfrac23\cdot (0, 0, 0, 1, 0)'(0, 0, 0, 1, 0)\nonumber
\\ & = \def\arraystretch{1.3}\left(
\begin{array}{rrrrr}
 \frac{1}{3} & -\frac{1}{3} & 0 & 0 & 0 \\
 -\frac{1}{3} & \frac{1}{3} & 0 & 0 & 0 \\
 0 & 0 & 0 & 0 & 0 \\
 0 & 0 & 0 & \frac{2}{3} & 0 \\
 0 & 0 & 0 & 0 & 0 \\
\end{array}
\right).
\end{align}
By calculating the $ B_i $'s we obtain the matrix $ B $ in (\ref{B}) as 
\begin{align}\label{Bbin}\def\+{\phantom{-}}
B=\def\arraystretch{1.3}\left(
\begin{array}{rrrrr}
 \+\frac{3}{10} & -\frac{3}{10} & -\frac{2}{15} & \+0 & \+0 \\
 -\frac{3}{10} & \+\frac{1}{2} & -\frac{1}{15} & \+0 & \+\frac{1}{5} \\
 -\frac{2}{15} & -\frac{1}{15} & \+\frac{1}{2} & -\frac{1}{6} & -\frac{1}{5} \\
 \+0 & \+0 & -\frac{1}{6} & \+\frac{1}{2} & \+0 \\
 \+0 & \+\frac{1}{5} & -\frac{1}{5} & \+0 & \+\frac{1}{5} \\
\end{array}
\right).
\end{align}

 From (\ref{simpleSigma}) in Theorem \ref{simplepolyathm} it follows 
that the  covariance matrix $ \Sigma $ for the asymptotic multivariate normal distribution of $X_n=(X_{n1},X_{n2},X_{n3},X_{n4},X_{n5}) $, is given by
\begin{align}\label{covbin}\def\+{\phantom{-}}
\Sigma=
\def\arraystretch{1.4}\left(
\begin{array}{rrrrr}
 \+\frac{43}{1575} & -\frac{67}{2520} & -\frac{113}{12600} & -\frac{29}{2520} & \+\frac{1}{1400} \\
 -\frac{67}{2520} & \+\frac{23}{420} & -\frac{1}{42} & -\frac{13}{1260} & \+\frac{71}{2520} \\
 -\frac{113}{12600} & -\frac{1}{42} & \+\frac{443}{6300} & -\frac{1}{30} & -\frac{59}{1800} \\
 -\frac{29}{2520} & -\frac{13}{1260} & -\frac{1}{30} & \+\frac{181}{1260} & -\frac{11}{504} \\
 \+\frac{1}{1400} & \+\frac{71}{2520} & -\frac{59}{1800} & -\frac{11}{504} & \+\frac{13}{450} \\
\end{array}
\right).
\end{align}
 Thus, it follows that 
\begin{align}\label{variancebin}\sigma^2_Y=\sigma_{4,4} +\sigma_{4,5}+\sigma_{5,4}+\sigma_{5,5}
=\frac{181}{1260} + \frac{13}{450} - 2\cdot \frac{11}{504}=\frac{29}{225}.
\end{align}
Thus, the proof of Theorem \ref{binary} is completed.
\qed

\section{Protected nodes in ternary search trees and \Polya{} urns}
\label{S3}

We now proceed by analyzing the number of protected nodes in ternary search trees, by using the \Polya{} urn in Section \ref{Polya} (described for general $ m $-ary search trees ) when $ m=3 $. The 19 different types we get are shown in Figure \ref{ternarytypes} (with a numbering that will be used below). From our characterization of the types in Section \ref{Polya}, for example type 2 corresponds to (0,1,2).
Note that type 17, type 18 and type 19 contain one protected node each, while the other types contain no protected nodes.

\begin{figure}[t]

\includegraphics[scale=0.12]{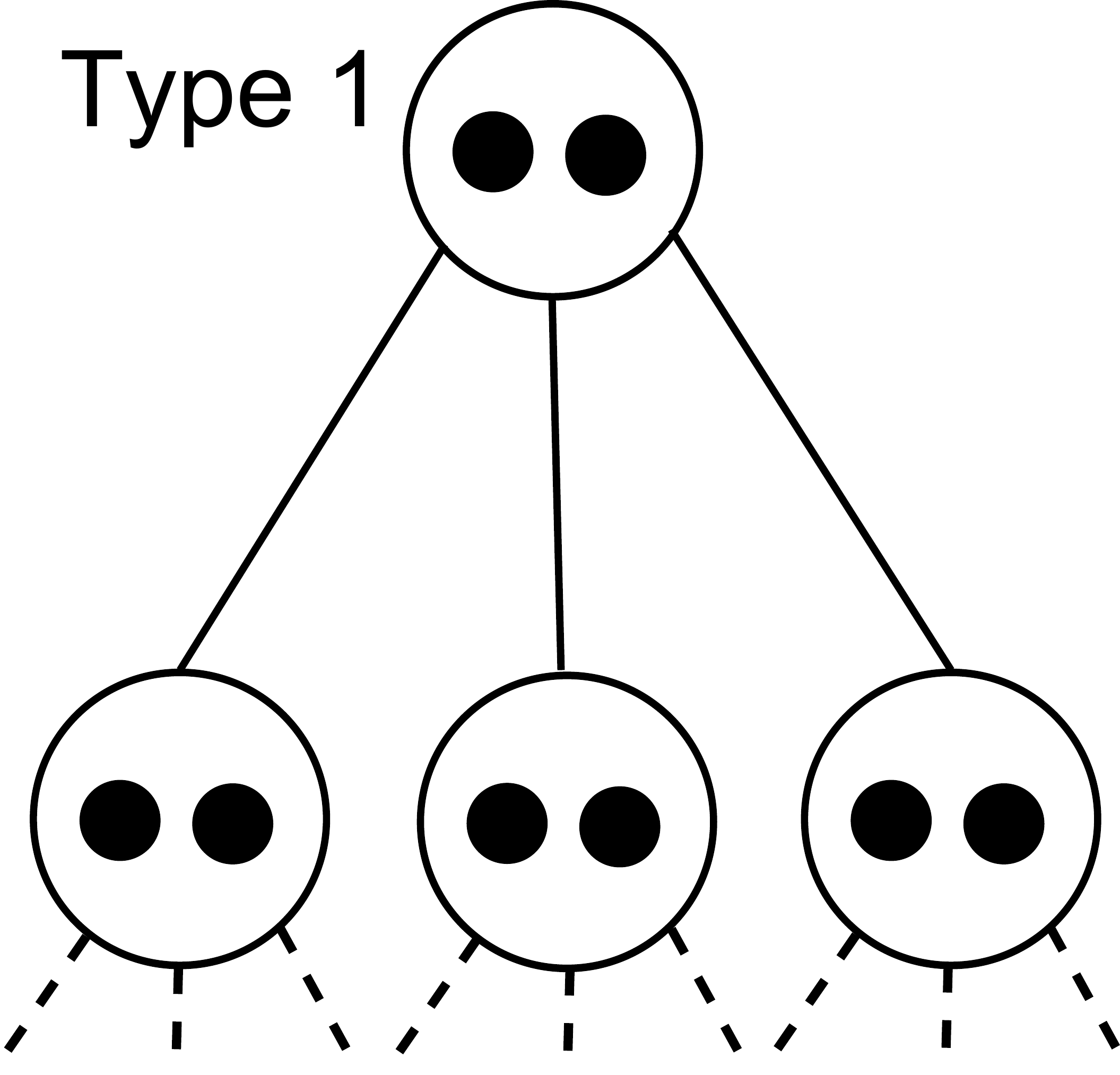}
\includegraphics[scale=0.12]{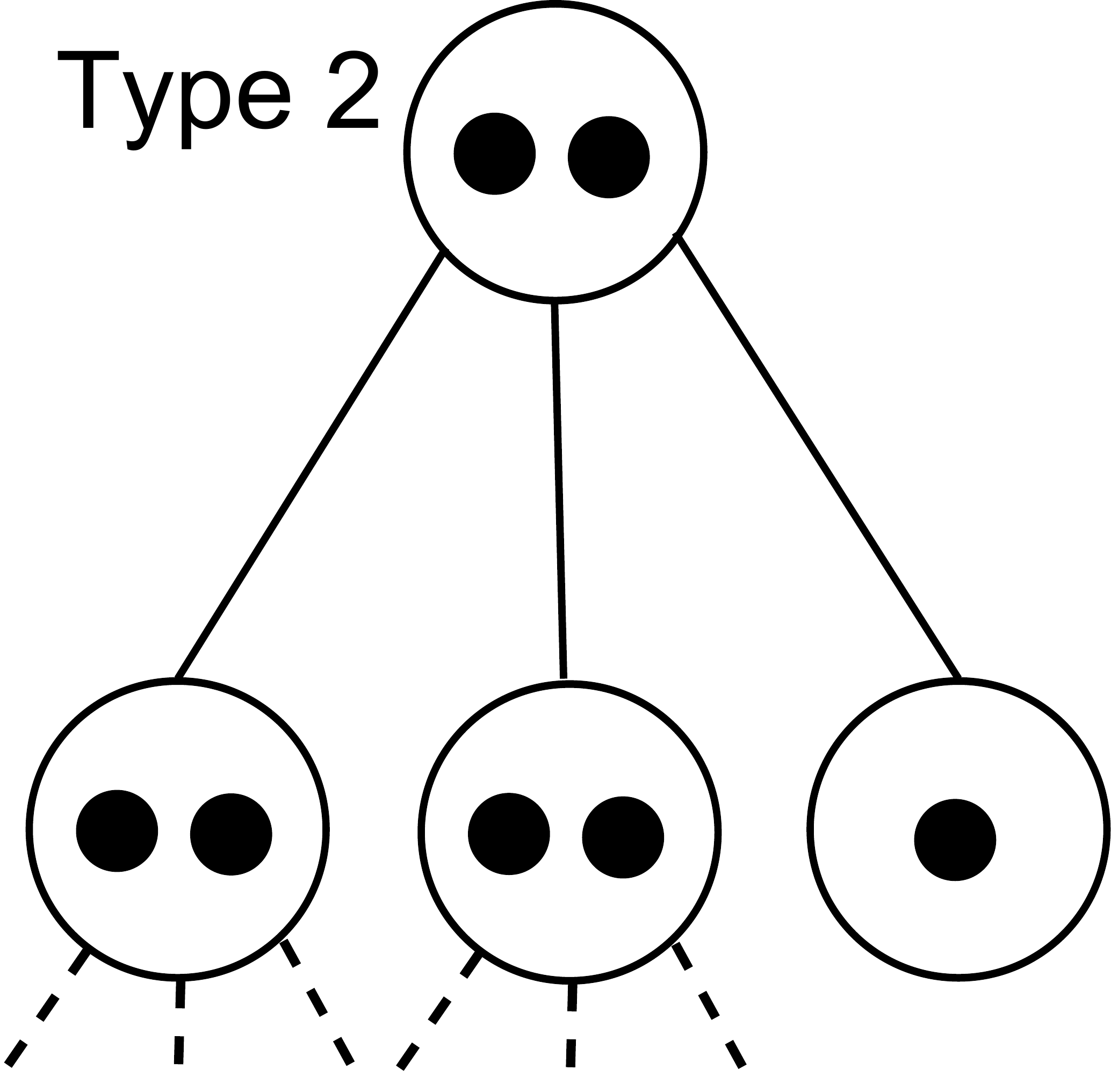}
\includegraphics[scale=0.12]{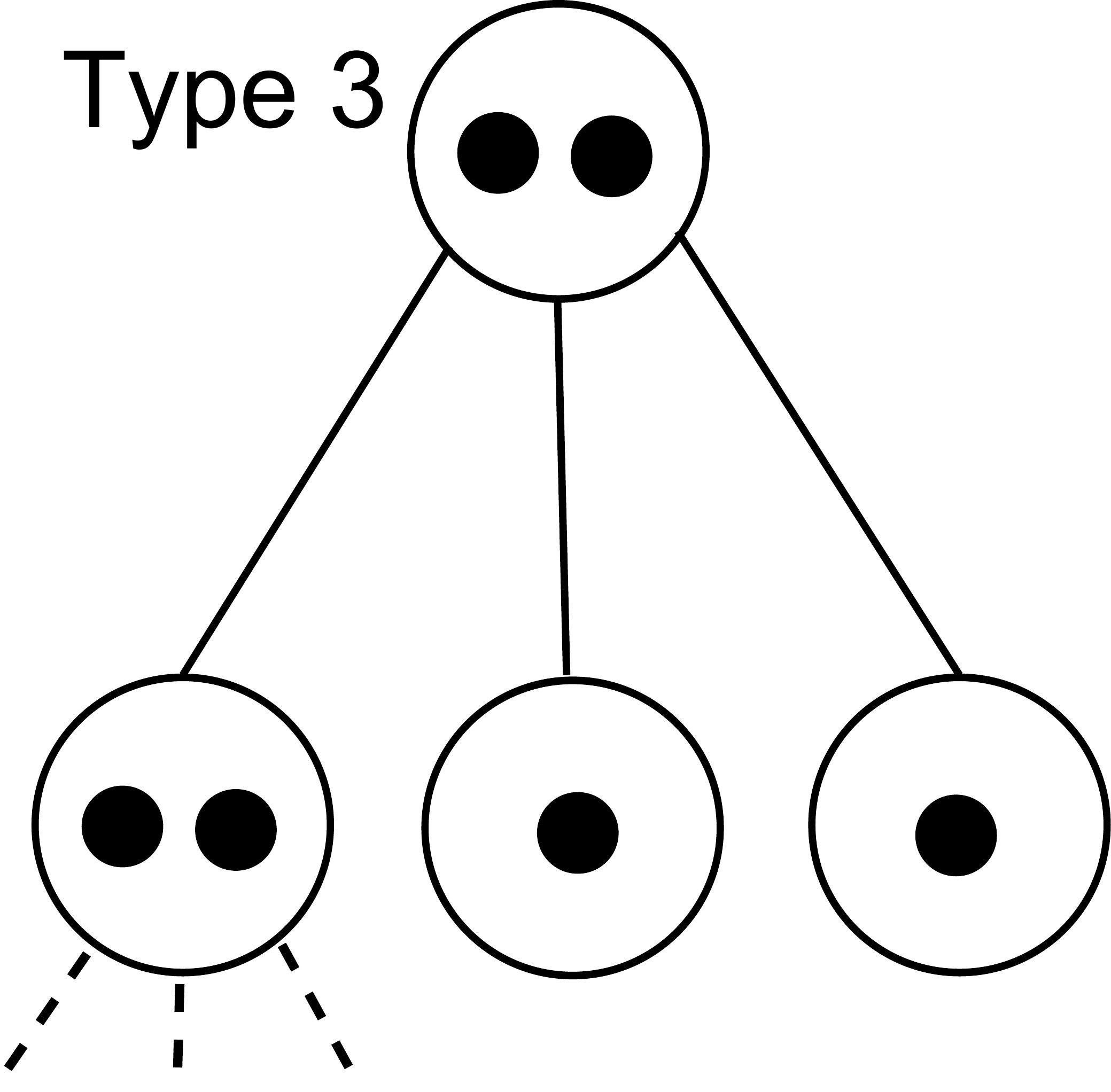}
\includegraphics[scale=0.12]{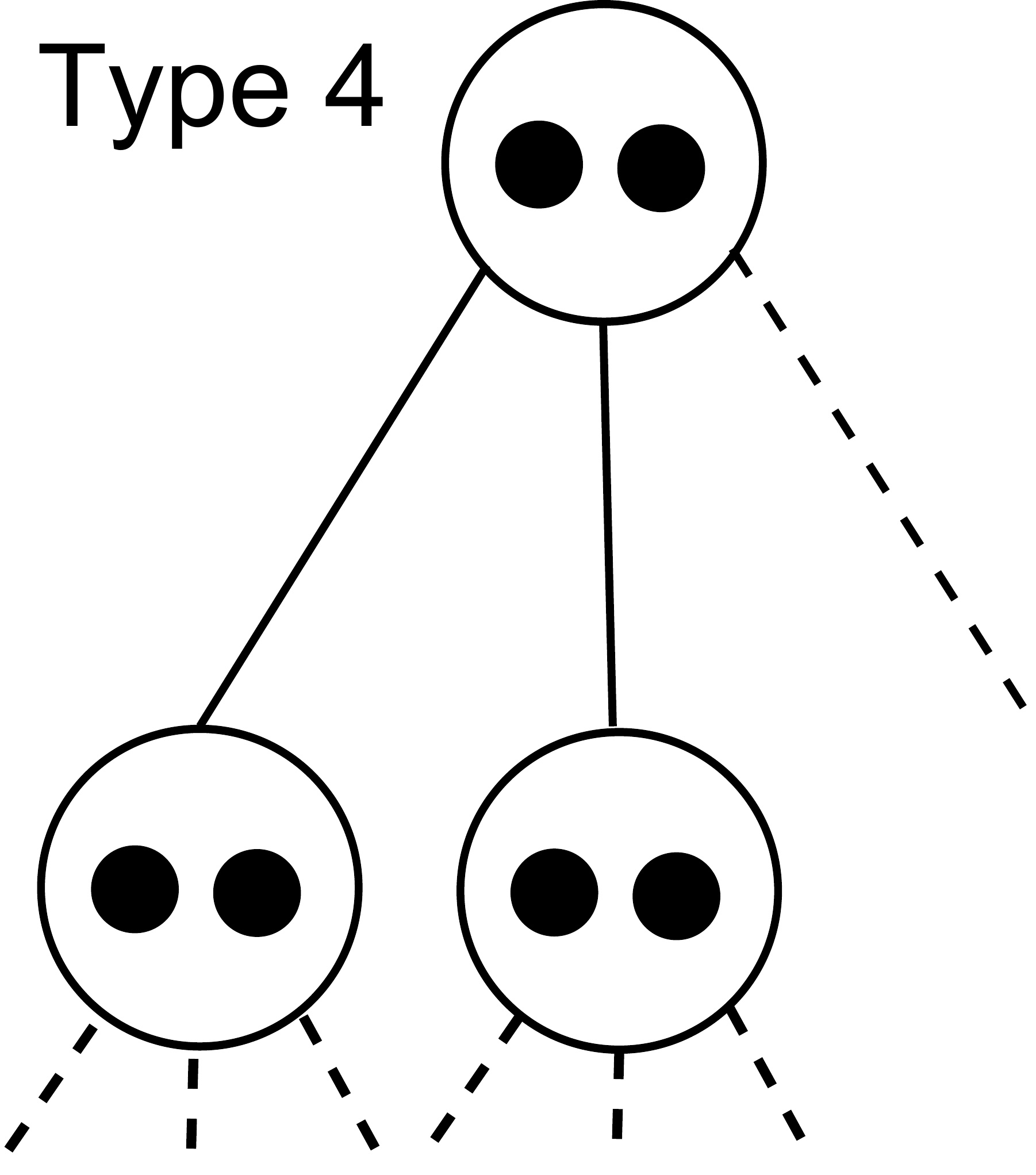}
\raisebox{7pt}{
\includegraphics[scale=0.12]{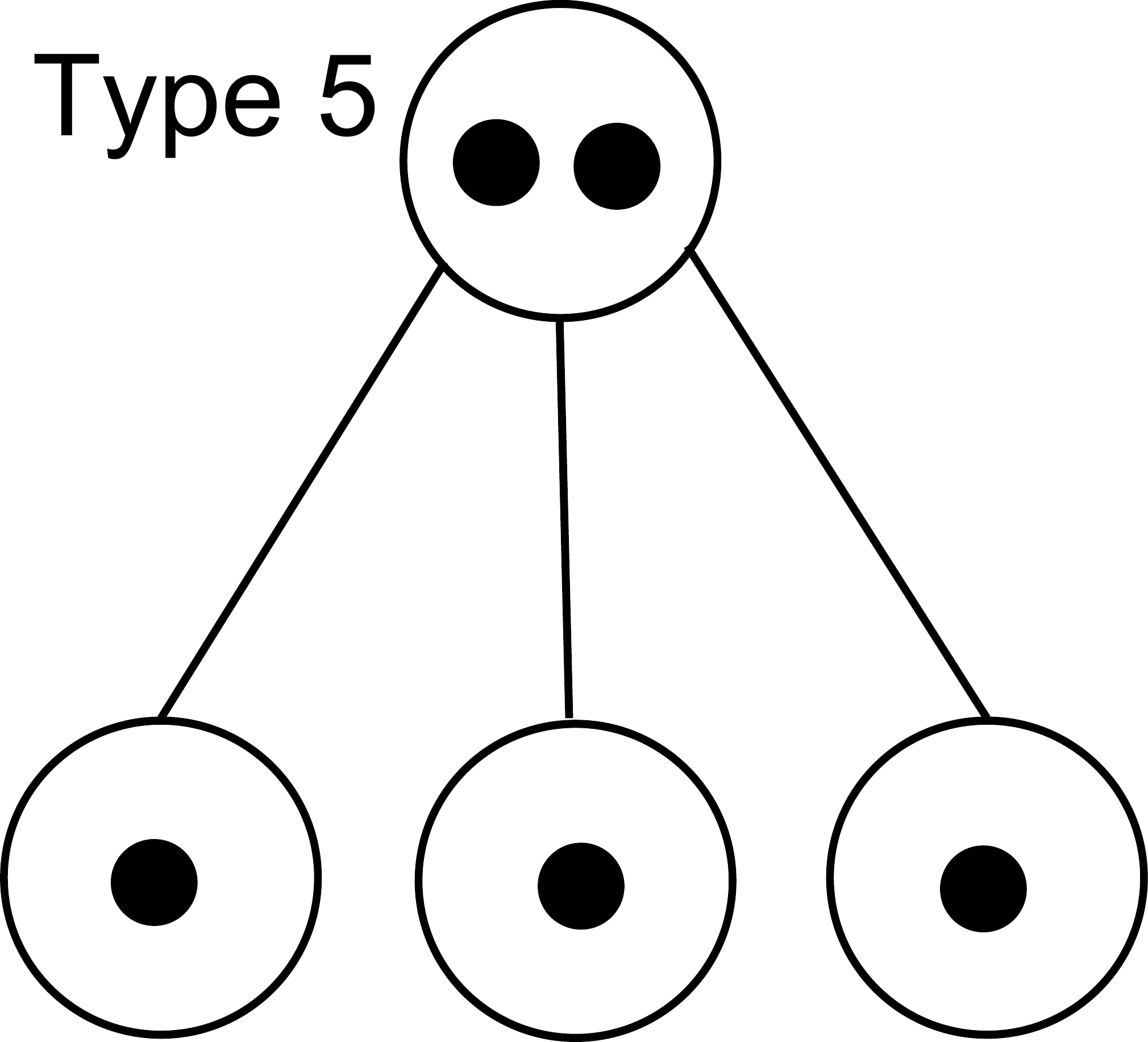}}
\includegraphics[scale=0.12]{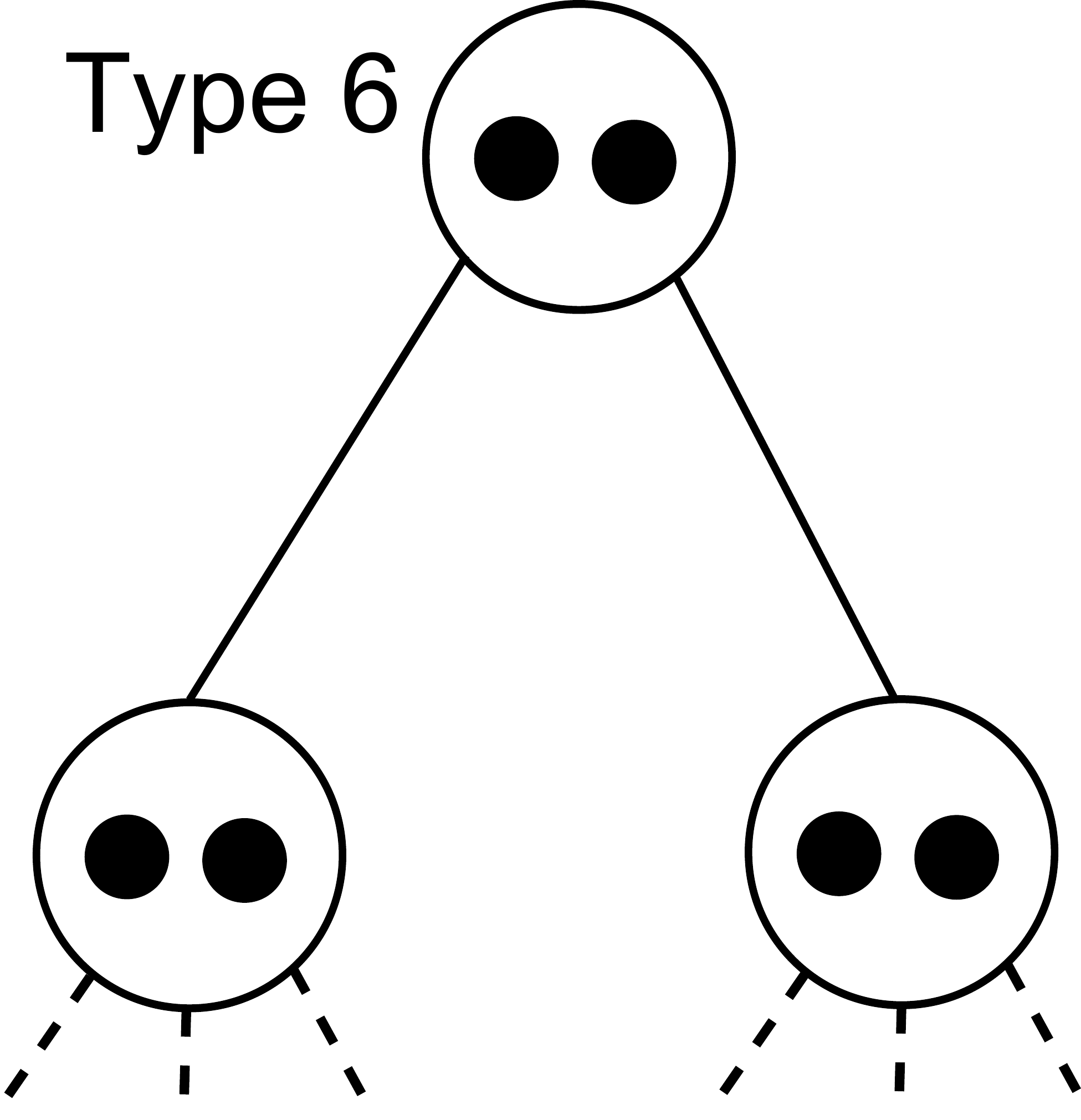}
\includegraphics[scale=0.12]{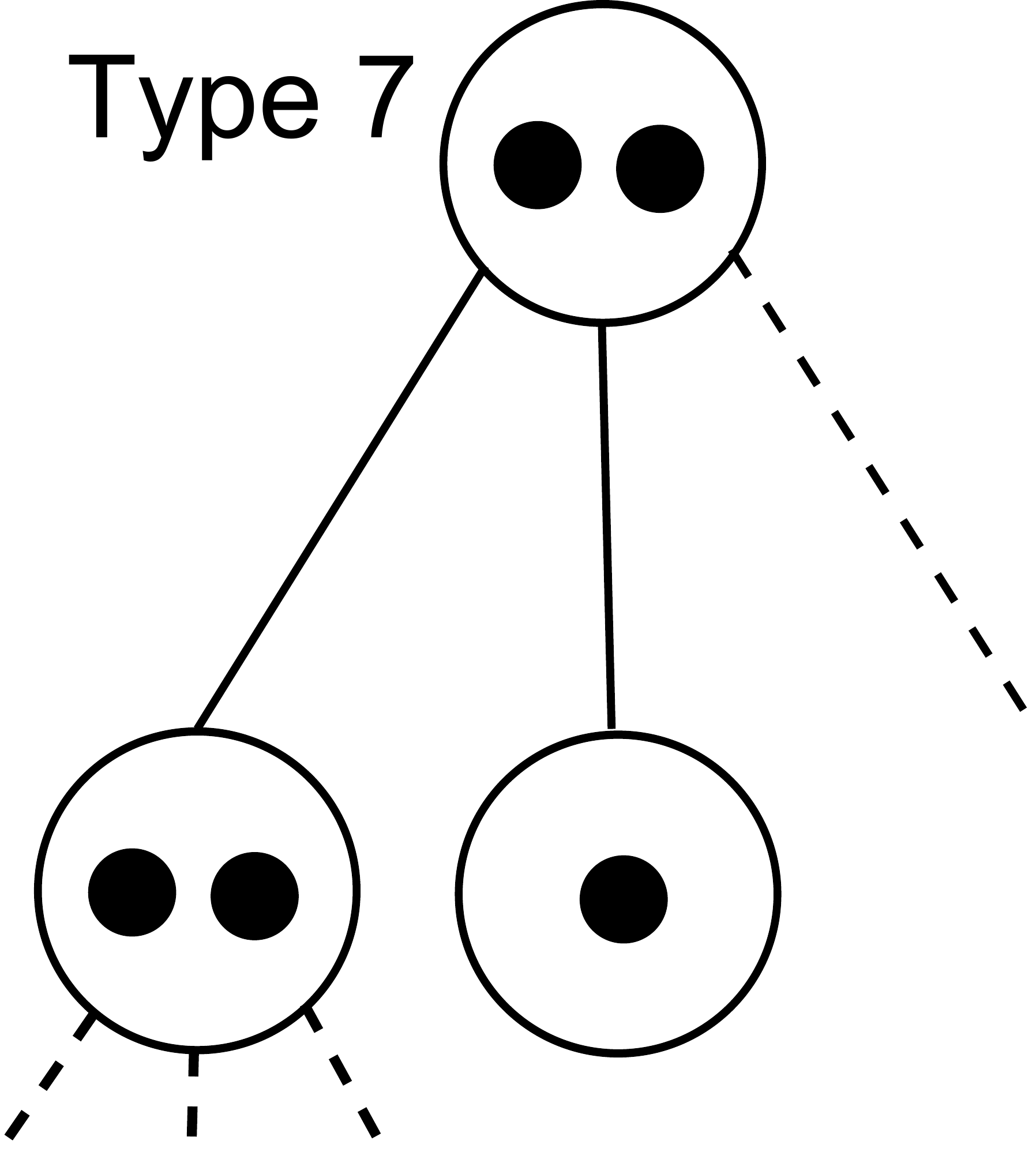}
\includegraphics[scale=0.12]{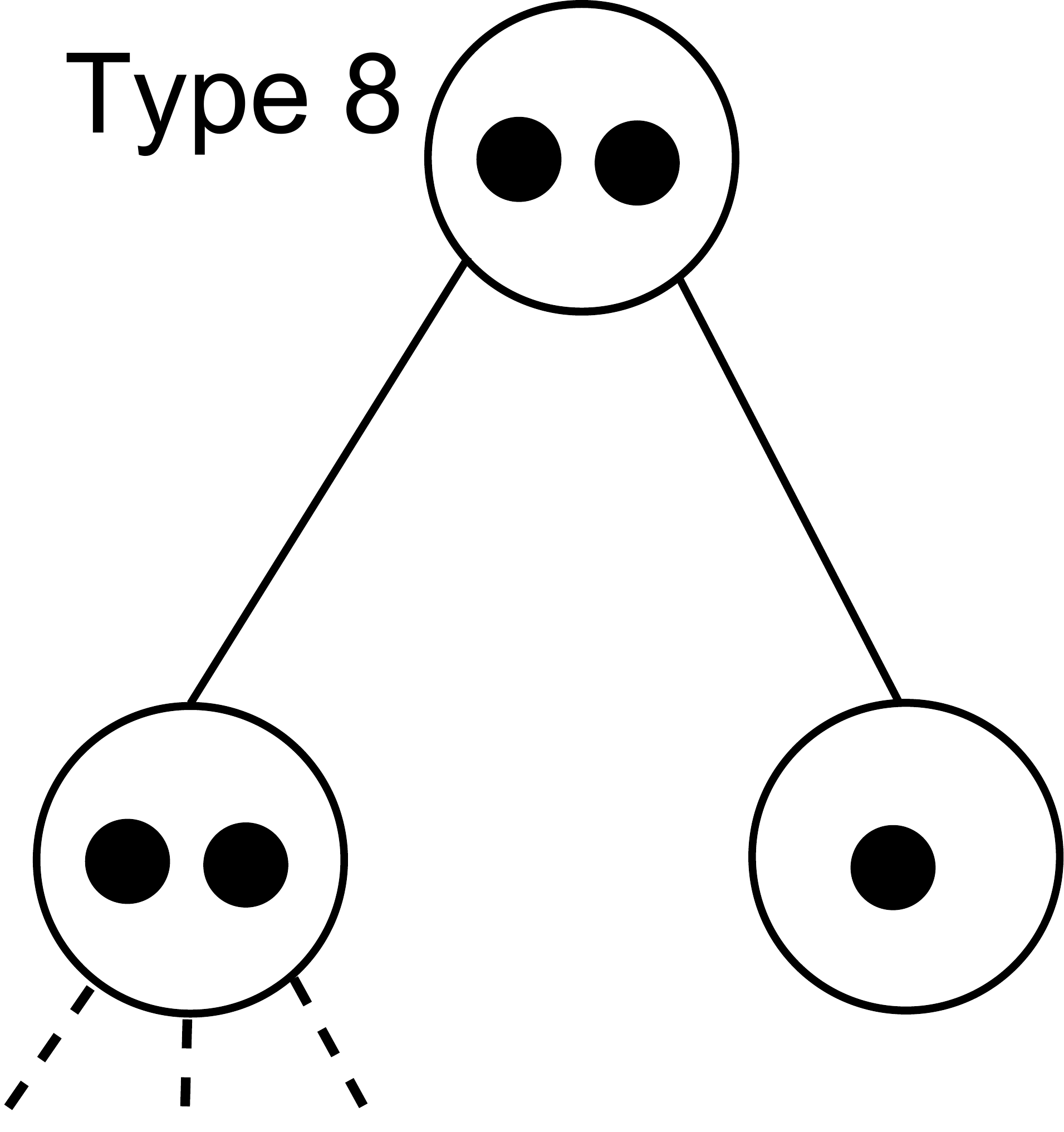}
\raisebox{7pt}{
\includegraphics[scale=0.12]{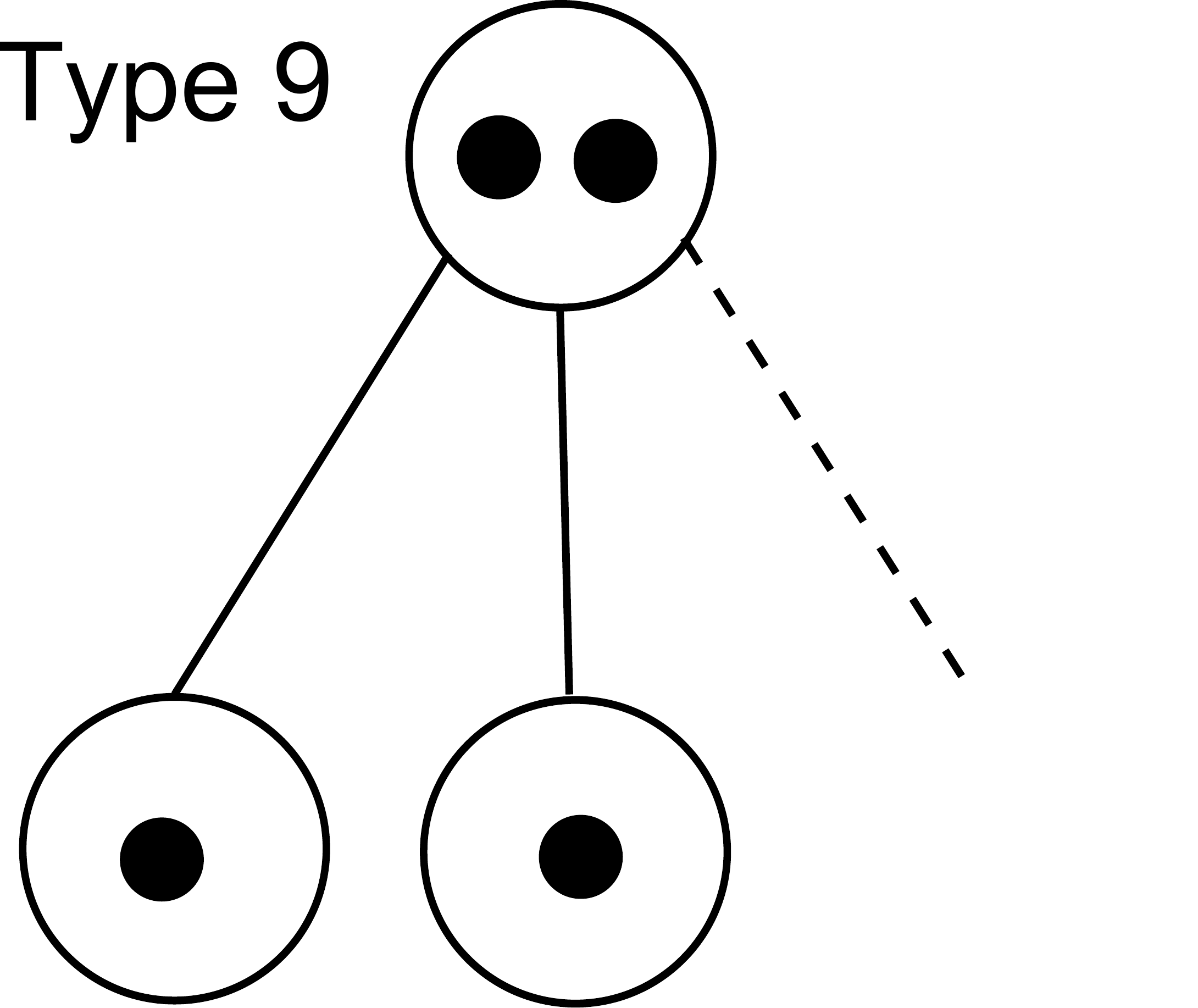}}
\includegraphics[scale=0.12]{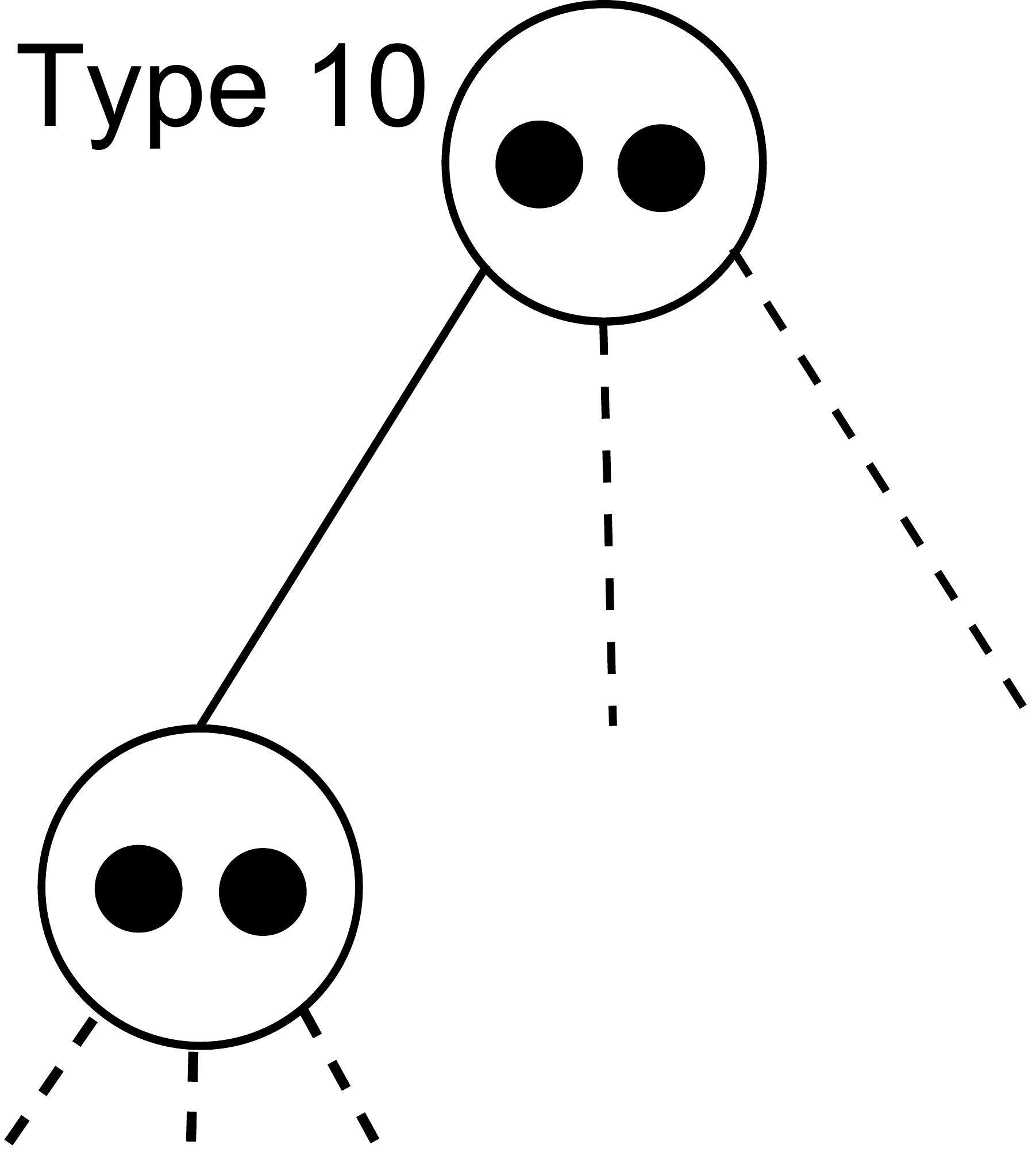}
\raisebox{7pt}{\includegraphics[scale=0.12]{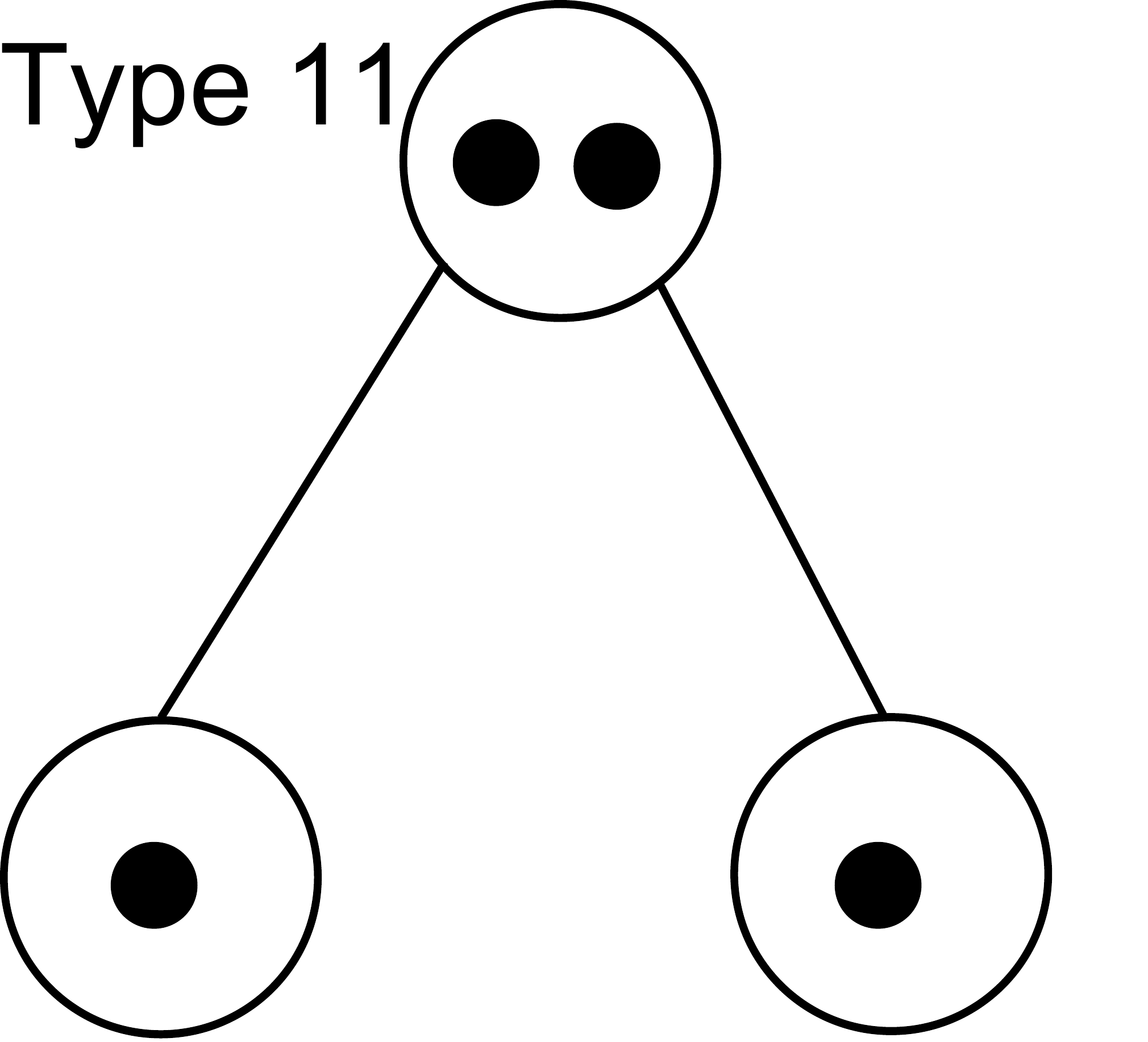}}
\includegraphics[scale=0.12]{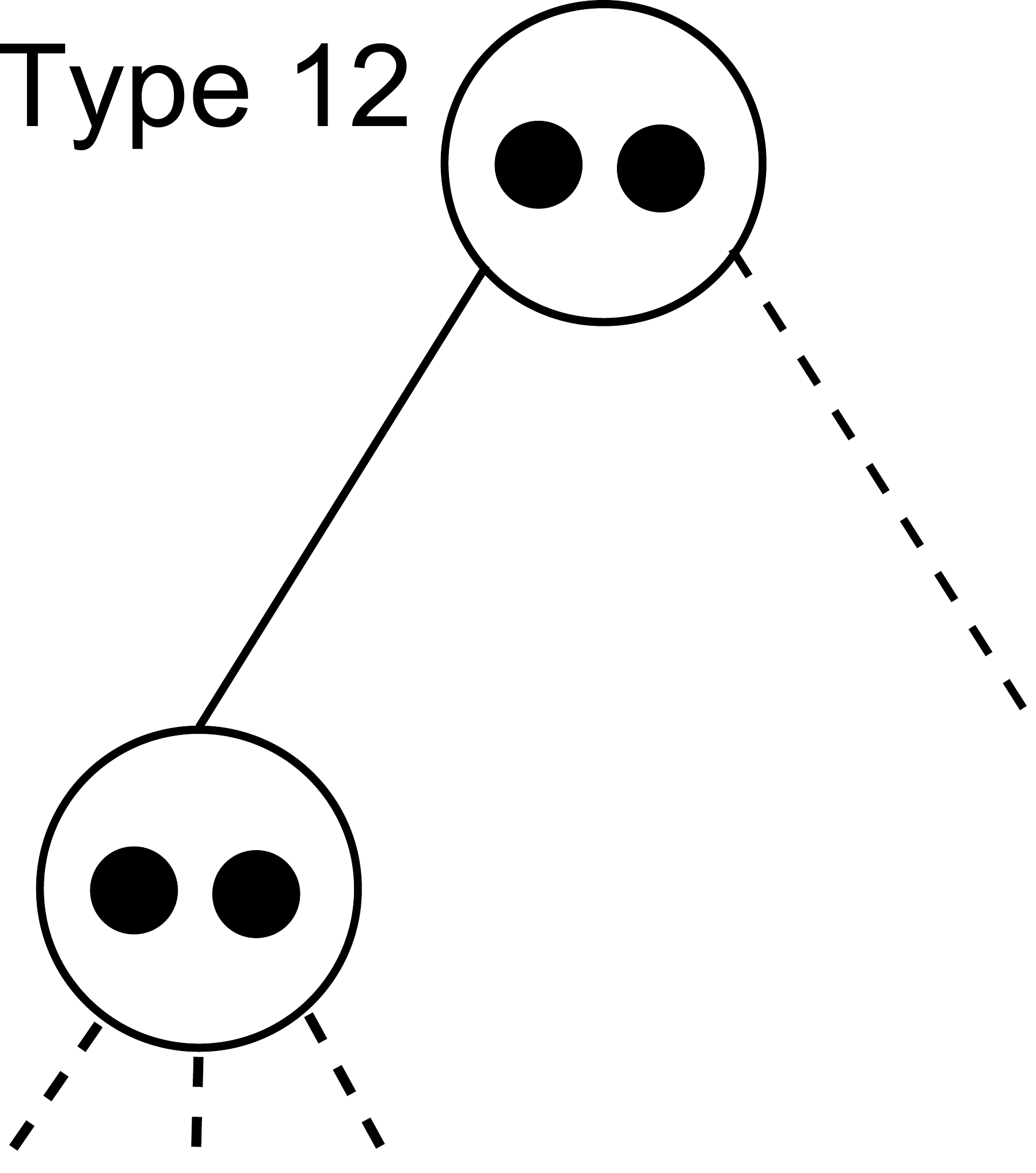}
\raisebox{7pt}{\includegraphics[scale=0.12]{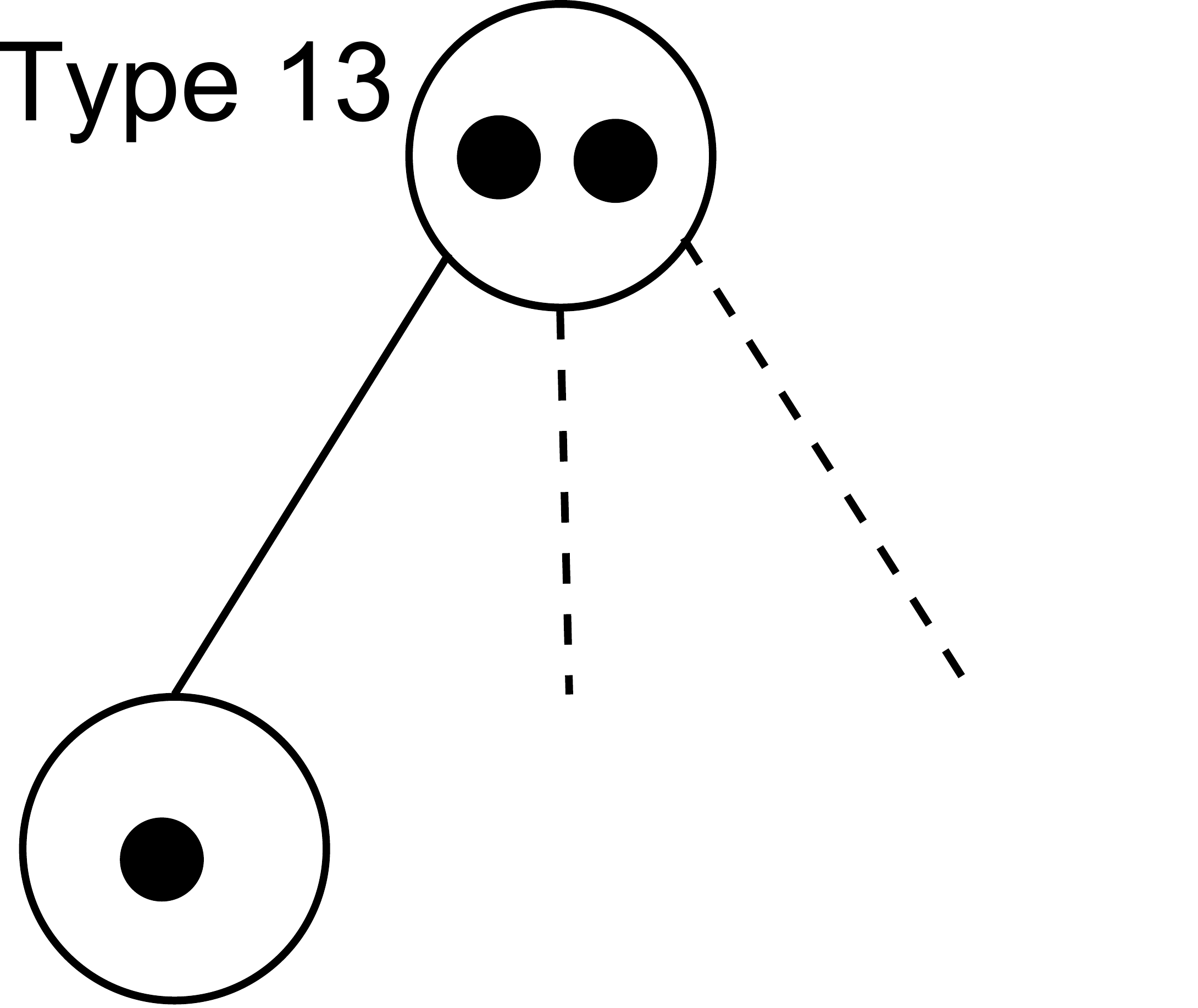}}
\includegraphics[scale=0.12]{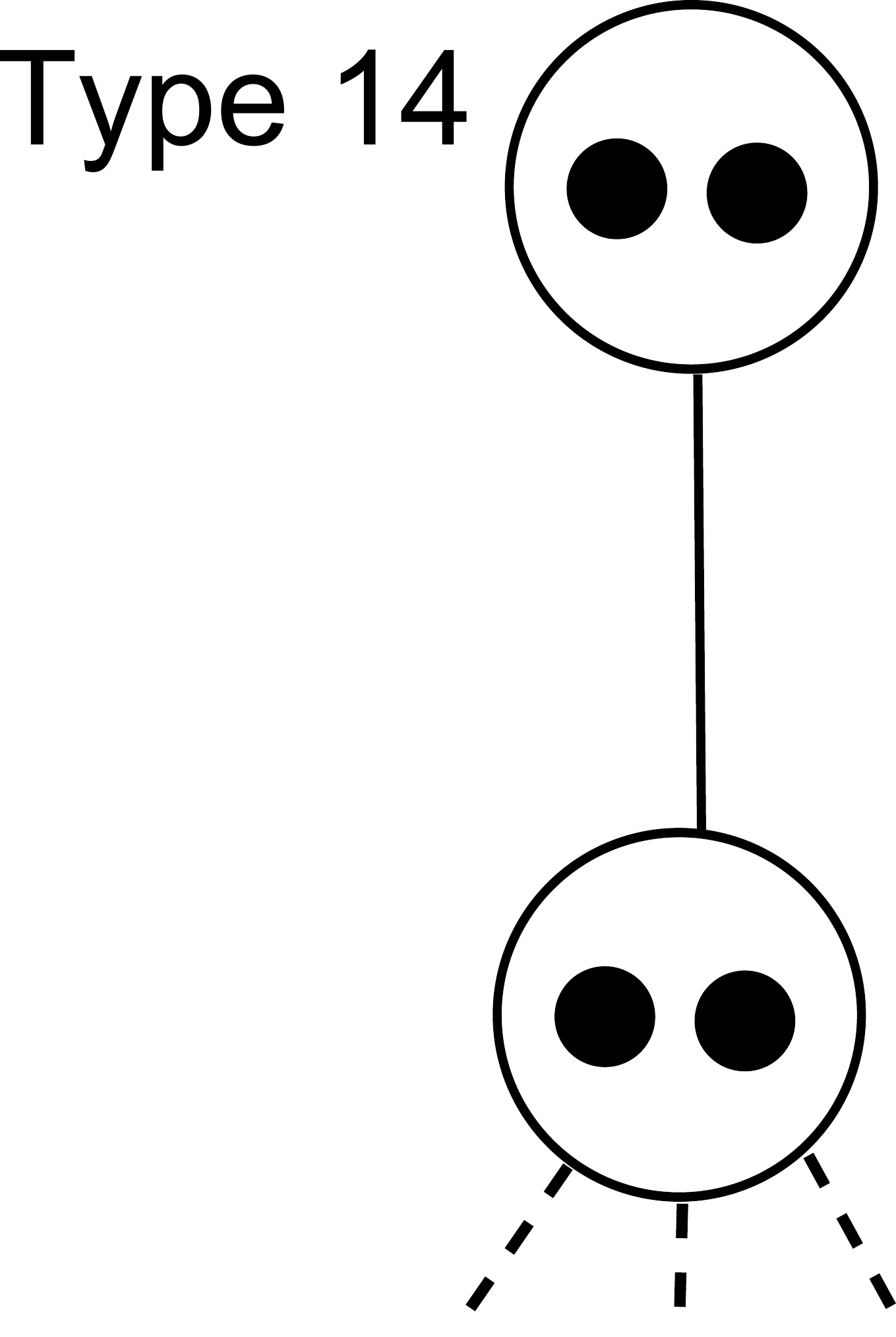}
\raisebox{7pt}{\includegraphics[scale=0.12]{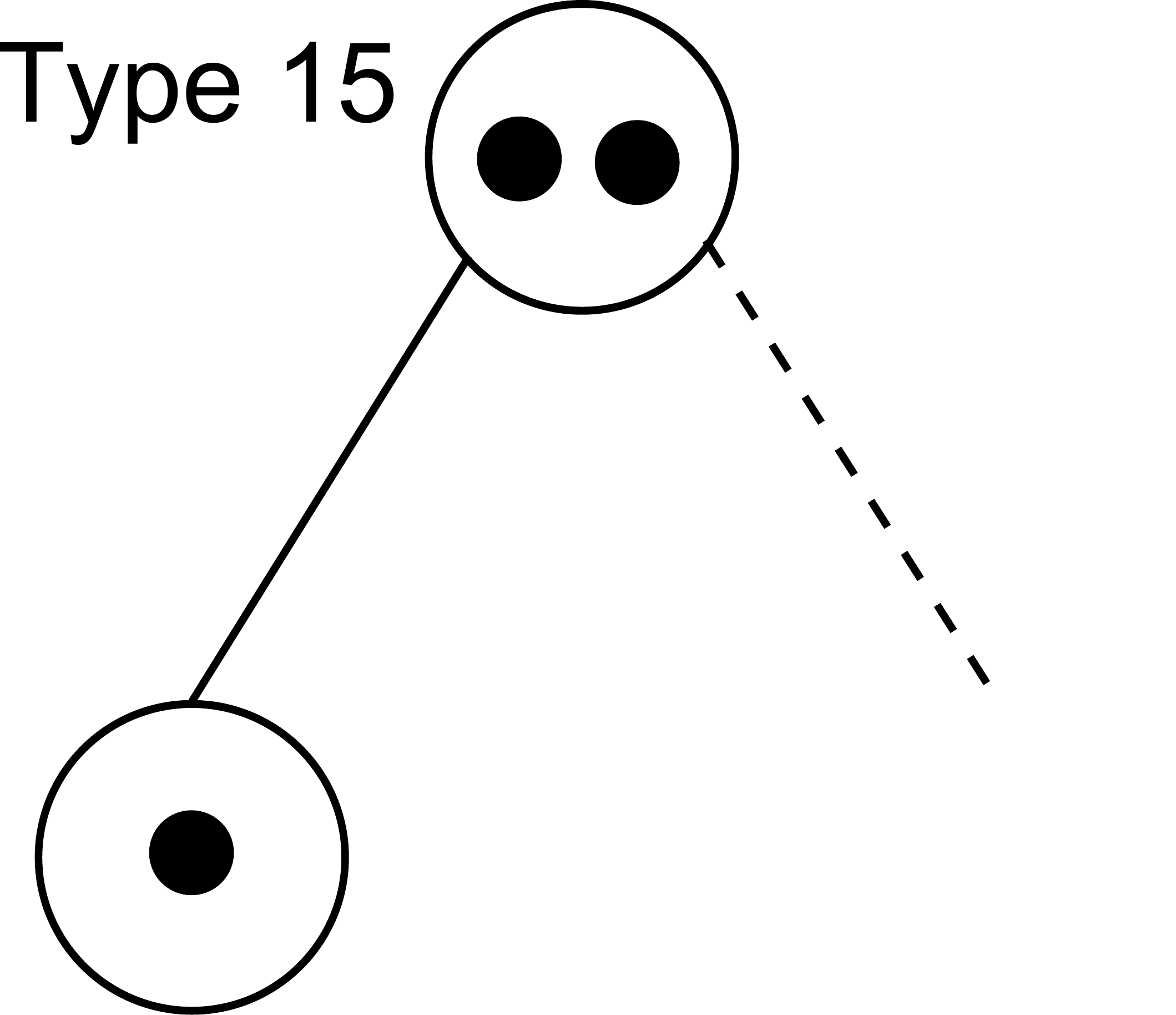}}
\raisebox{7pt}{\includegraphics[scale=0.12]{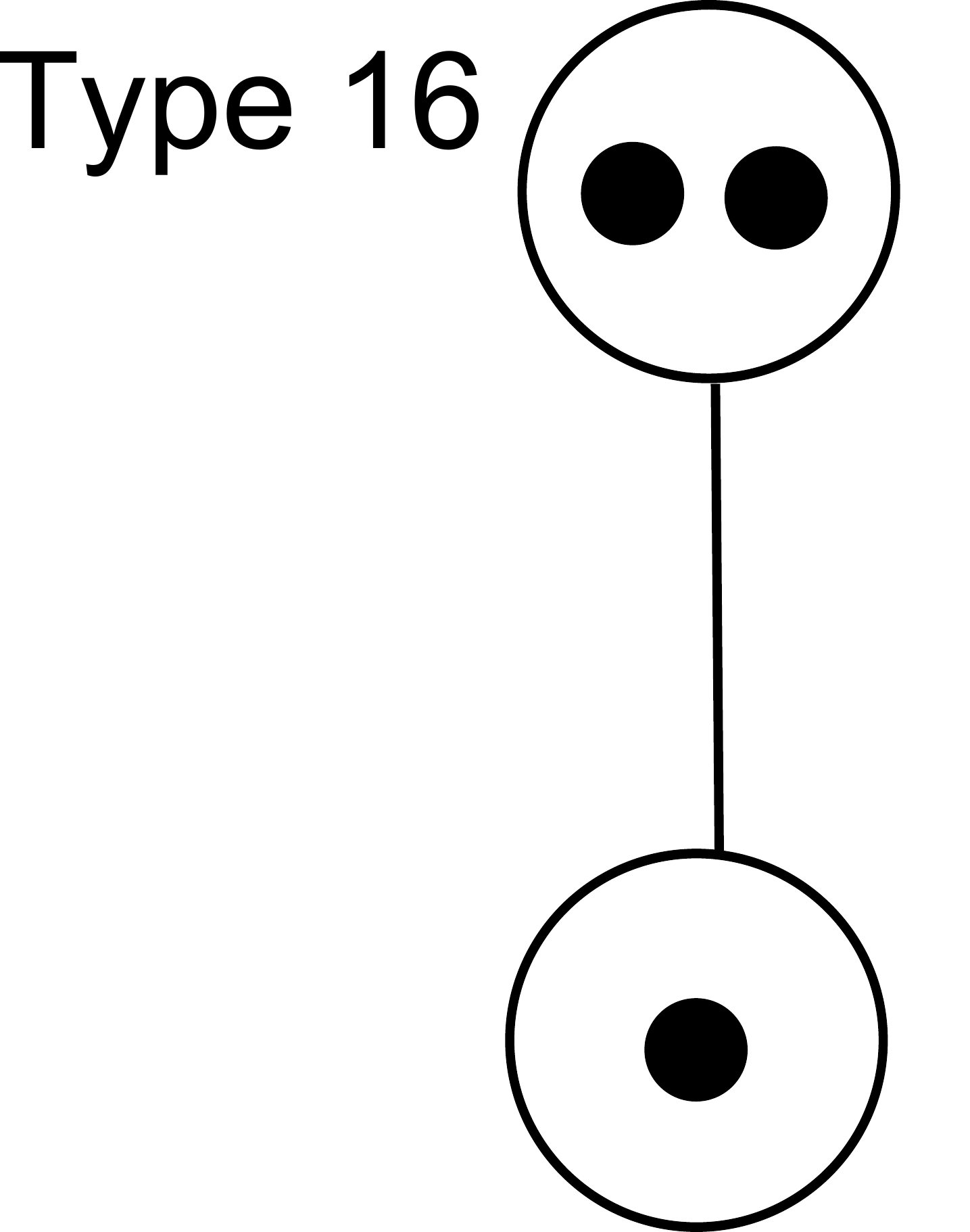}}
\includegraphics[scale=0.12]{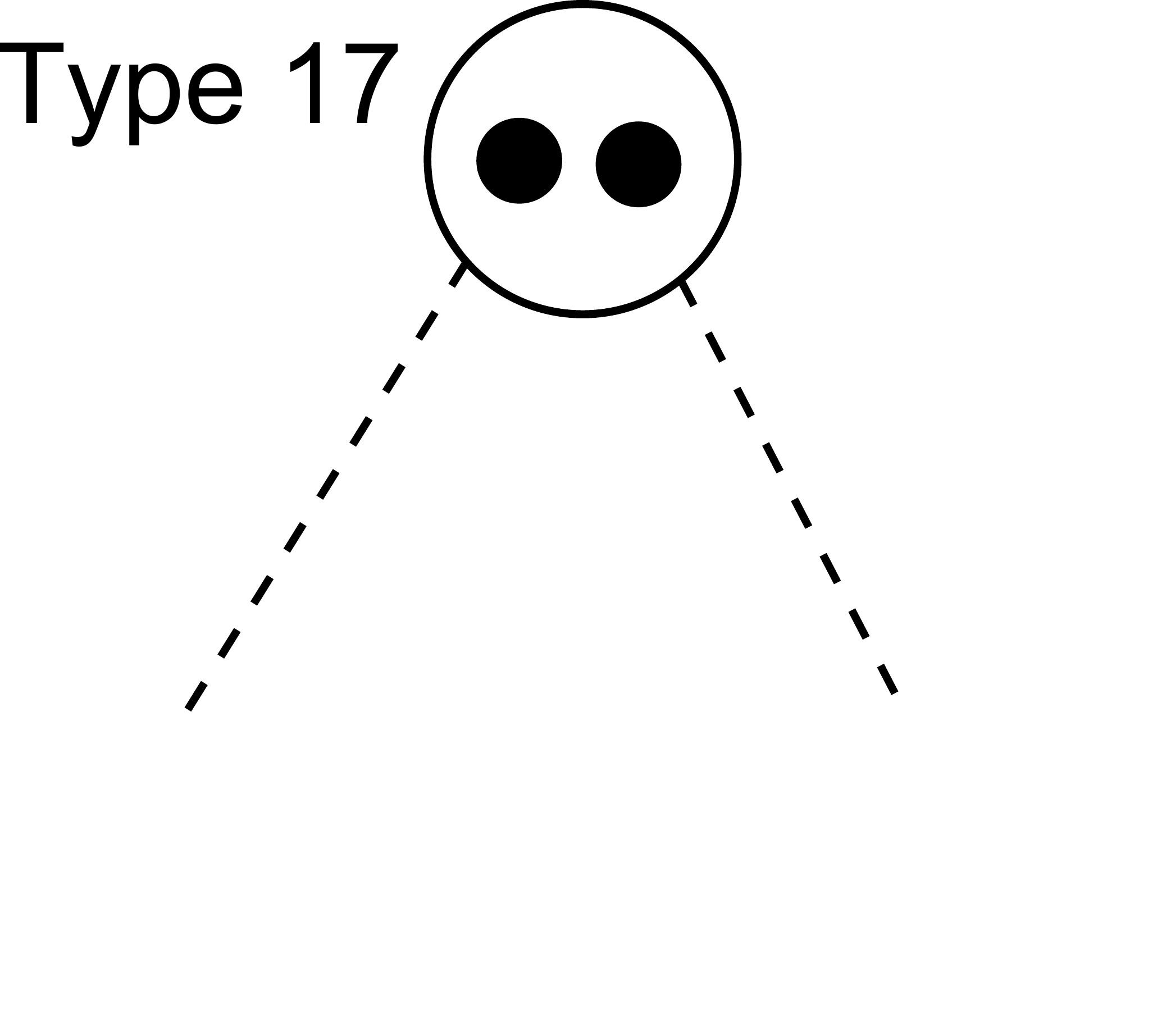}
\includegraphics[scale=0.12]{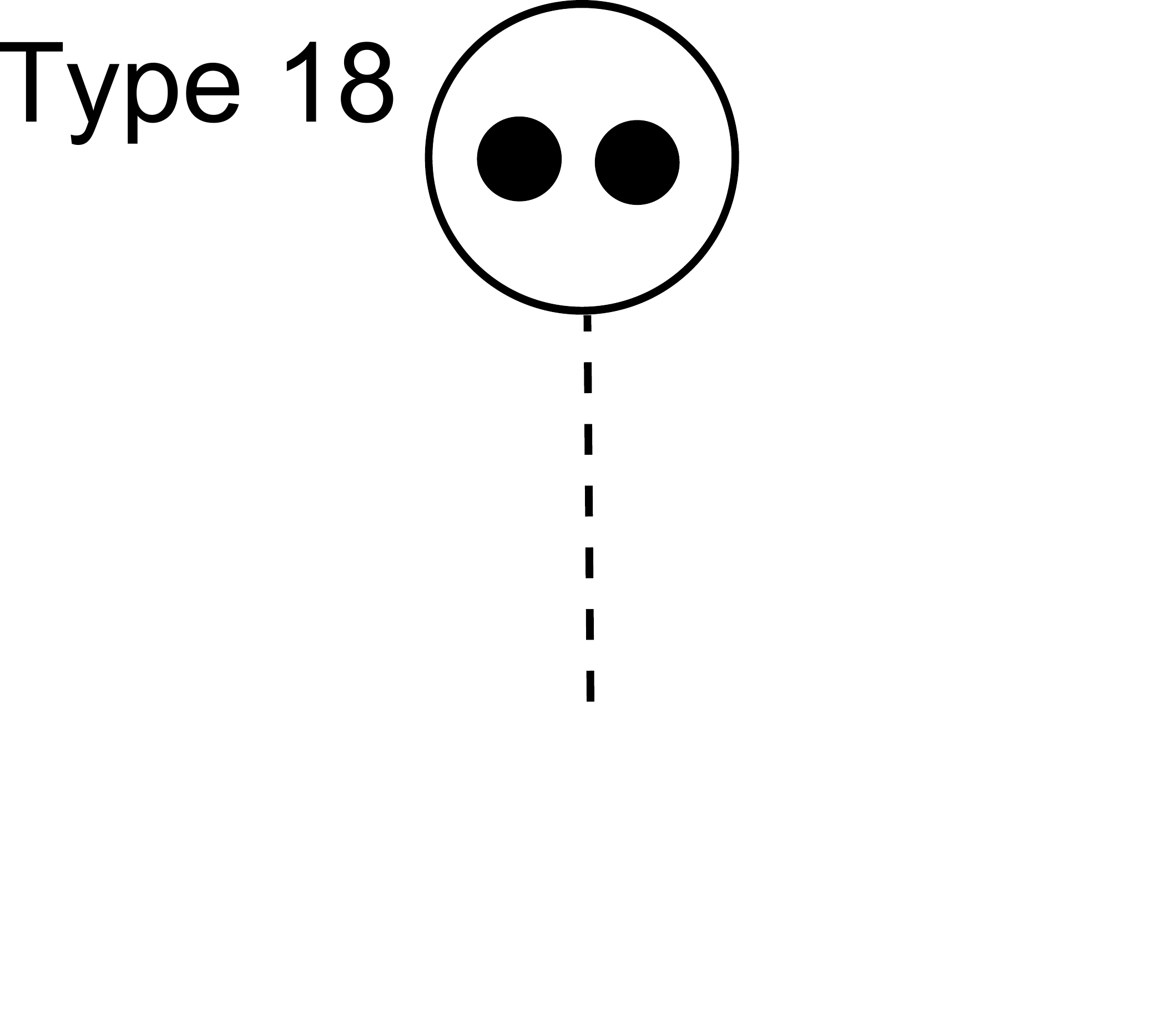}
\includegraphics[scale=0.12]{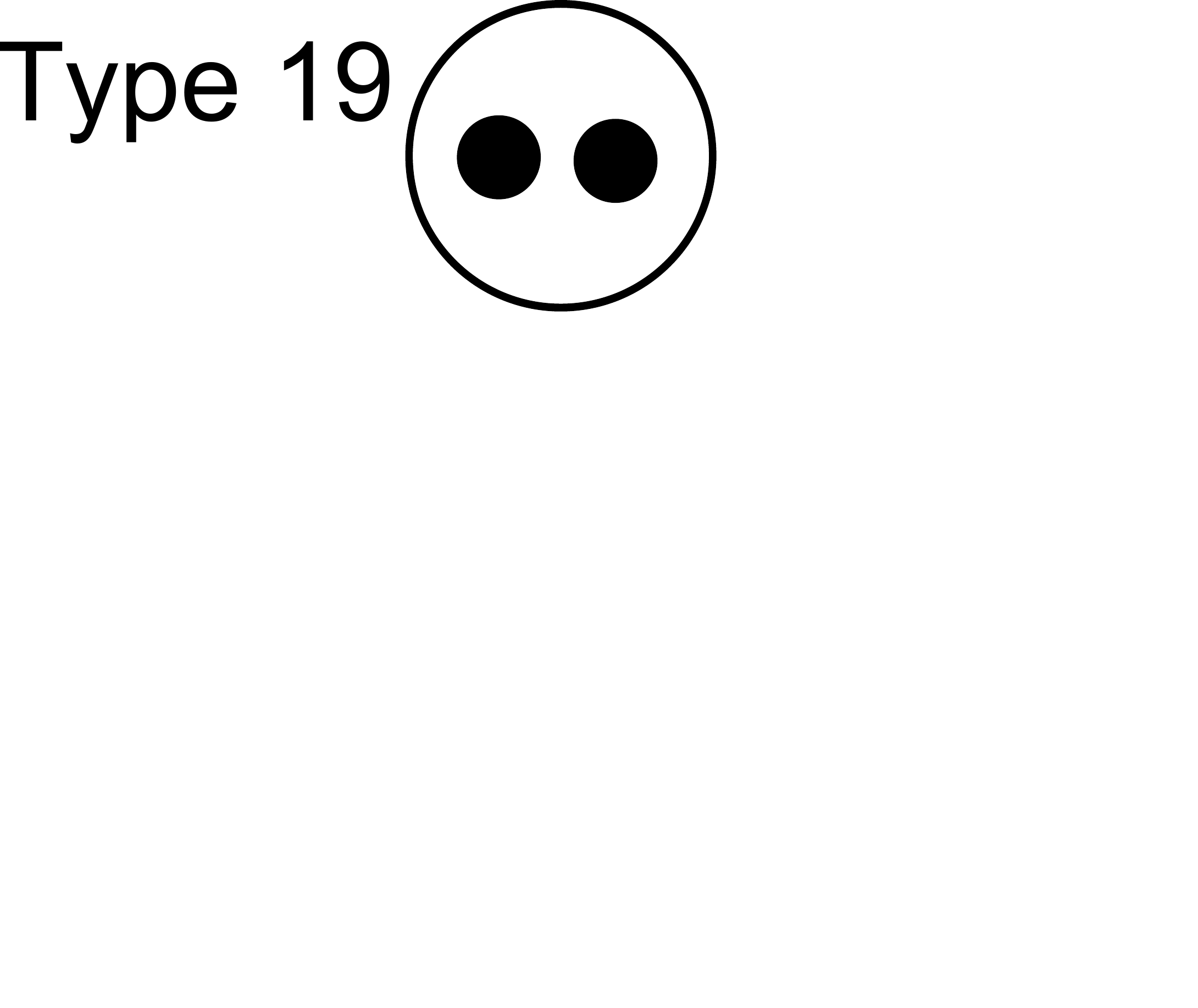}

\caption{The different types characterizing protected and unprotected nodes in ternary search trees. Type 17, type 18 and type 19 are the only ones that include protected nodes.  }\label{ternarytypes}
\end{figure}

To determine the matrix $ A $ we proceed (as for the binary search tree) to find the transitions when a ball (in our case one of the 19 trees in our forest) of type $ i $ is chosen. Figure  \ref{add2ternary} illustrates the different situations for how a new key could be added to a ball (a tree) of type 2. All the other cases are similar, and we leave these cases as an exercise to the reader.
\begin{figure}[h]
\begin{center}
\includegraphics[scale=0.12]{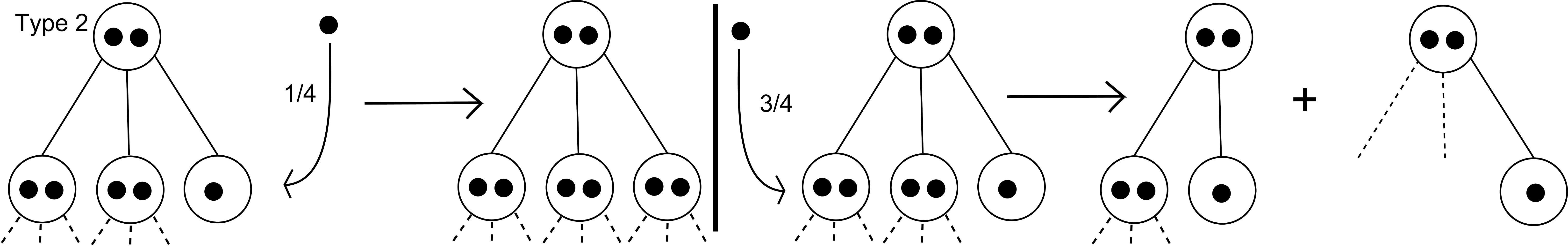}
\end{center}
\caption{The two possibilities for adding a key to a node in a tree of type 2 of a ternary search tree. }\label{add2ternary}
\end{figure}
From the different transitions for changing a node of type $ i $ we get the
matrix $ A $ for ternary search trees in Figure \ref{Atrinar}. The example 
in Figure \ref{add2ternary} gives the second column  of $A$. 
The tree of type 2 has activity 8. If it is drawn, and the
new key is added to the node with only one key which happens with
probability $ \frac{2}{8} $, then a tree of type 2 is replaced with a tree
of type 1. If the new key is instead added to one of the nodes containing
two keys which happens with probability $ \frac{6}{8} $, then the tree of
type 2 is replaced by a tree of type 8 and one tree of type 13. Thus, the
second column of the matrix $ A $ for the ternary search tree is given by  
$$ 8\cdot
(\tfrac{2}{8},-1,0,0,0,0,0,\tfrac{6}{8},0,0,0,0,\tfrac{6}{8},0,0,0,0,0,0)'. 
$$
In this way we obtain $A$ in Figure \ref{Atrinar}.
\begin{figure}
\def\+{\phantom{-}}
$ A=
\left(\arraycolsep=3pt %\def\arraystretch{2.0}
\begin{array}{ccccccccccccccccccc}
 -9 & \+2 & \+0 & \+0 & \+0 & \+0 & \+0 & \+0 & \+0 & \+0 & \+0 & \+0 & \+0 & \+0 & \+0 & \+0 & \+0 & \+0 & \+0 \\
 \+0 & -8 & \+4 & \+1 & \+0 & \+0 & \+0 & \+0 & \+0 & \+0 & \+0 & \+0 & \+0 & \+0 & \+0 & \+0 & \+0 & \+0 & \+0 \\
 \+0 & \+0 & -7 & \+0 & \+6 & \+0 & \+1 & \+0 & \+0 & \+0 & \+0 & \+0 & \+0 & \+0 & \+0 & \+0 & \+0 & \+0 & \+0 \\
 \+0 & \+0 & \+0 & -7 & \+0 & \+0 & \+2 & \+0 & \+0 & \+0 & \+0 & \+0 & \+0 & \+0 & \+0 & \+0 & \+0 & \+0 & \+0 \\
 \+0 & \+0 & \+0 & \+0 & -6 & \+0 & \+0 & \+0 & \+1 & \+0 & \+0 & \+0 & \+0 & \+0 & \+0 & \+0 & \+0 & \+0 & \+0 \\
 \+9 & \+0 & \+0 & \+0 & \+0 & -6 & \+0 & \+2 & \+0 & \+0 & \+0 & \+0 & \+0 & \+0 & \+0 & \+0 & \+0 & \+0 & \+0 \\
 \+0 & \+0 & \+0 & \+0 & \+0 & \+0 & -6 & \+0 & \+4 & \+2 & \+0 & \+0 & \+0 & \+0 & \+0 & \+0 & \+0 & \+0 & \+0 \\
 \+0 & \+6 & \+0 & \+0 & \+0 & \+0 & \+0 & -5 & \+0 & \+0 & \+4 & \+1 & \+0 & \+0 & \+0 & \+0 & \+0 & \+0 & \+0 \\
 \+0 & \+0 & \+0 & \+0 & \+0 & \+0 & \+0 & \+0 & -5 & \+0 & \+0 & \+0 & \+2 & \+0 & \+0 & \+0 & \+0 & \+0 & \+0 \\
 \+0 & \+0 & \+0 & \+0 & \+0 & \+0 & \+0 & \+0 & \+0 & -5 & \+0 & \+0 & \+2 & \+0 & \+0 & \+0 & \+0 & \+0 & \+0 \\
 \+0 & \+0 & \+3 & \+0 & \+0 & \+0 & \+0 & \+0 & \+0 & \+0 & -4 & \+0 & \+0 & \+0 & \+1 & \+0 & \+0 & \+0 & \+0 \\
 \+0 & \+0 & \+0 & \+6 & \+0 & \+0 & \+0 & \+0 & \+0 & \+0 & \+0 & -4 & \+0 & \+0 & \+2 & \+0 & \+0 & \+0 & \+0 \\
 \+9 & \+6 & \+3 & \+6 & \+0 & \+6 & \+3 & \+3 & \+0 & \+3 & \+0 & \+3 & -4 & \+3 & \+0 & \+0 & \+0 & \+0 & \+0 \\
 \+0 & \+0 & \+0 & \+0 & \+0 & \+6 & \+0 & \+0 & \+0 & \+0 & \+0 & \+0 & \+0 & -3 & \+0 & \+2 & \+0 & \+0 & \+0 \\
 \+0 & \+0 & \+0 & \+0 & \+0 & \+0 & \+3 & \+0 & \+0 & \+0 & \+0 & \+0 & \+0 & \+0 & -3 & \+0 & \+2 & \+0 & \+0 \\
 \+0 & \+0 & \+0 & \+0 & \+0 & \+0 & \+0 & \+3 & \+0 & \+0 & \+0 & \+0 & \+0 & \+0 & \+0 & -2 & \+0 & \+1 & \+0 \\
 \+0 & \+0 & \+0 & \+0 & \+0 & \+0 & \+0 & \+0 & \+0 & \+3 & \+0 & \+0 & \+0 & \+0 & \+0 & \+0 & -2 & \+0 & \+0 \\
 \+0 & \+0 & \+0 & \+0 & \+0 & \+0 & \+0 & \+0 & \+0 & \+0 & \+0 & \+3 & \+0 & \+0 & \+0 & \+0 & \+0 & -1 & \+0 \\
 \+0 & \+0 & \+0 & \+0 & \+0 & \+0 & \+0 & \+0 & \+0 & \+0 & \+0 & \+0 & \+0 & \+3 & \+0 & \+0 & \+0 & \+0 & \+0 \\
\end{array}
\right) $
\caption{The transition matrix $ A $ for the \Polya{} urn defined in \refS{Polya} in the case of the ternary search tree.}
\label{Atrinar}
\end{figure}

The activities of the different types are given by the vector 
$$a=(9, 8, 7, 7, 6, 6, 6, 5, 5, 5, 4, 4, 4, 3, 3, 2, 2,
1, 0)'.  $$ 
These correspond to the number of gaps for the different types.
The eigenvalues of the matrix $ A $ %in Figure \ref{Atrinar} 
are 
$$1, 0,-2,-3,-3,-4, -4, -4, -4, -5, -5, -5,-6, -6, -6,-7, -7,-8, -9.$$
The eigenspace belonging to the eigenvalue $ -4 $ (which has algebraic
multiplicity 4) has dimension 3. 
Since the dimension of the eigenspace belonging to the eigenvalue $-4$ is not
equal to the algebraic multiplicity, the matrix $A$ is not
diagonalisable. (However, all other eigenspaces have full dimension.) 
 Hence, we can not apply Theorem \ref{simplepolyathm}. 
However, Theorem \ref{polyathm} can be applied since  $ a\cdot\E(\xi_i)=1 $ for
each $i$ (this
follows since we always add exactly one key when a tree of type $ i $ is
chosen). 

From Theorem \ref{polyathm} we  obtain that the vector $
X_n=(X_{n1},\dots,X_{n19}) $, where $ X_{ni} $ are the number of balls  of
type $ i $ (in our case the number of  trees that correspond to type $ i $
in our forest obtained from the ternary search tree), has asymptotically a
multivariate normal distribution. 
Let $ Z_n $ be  the number of protected nodes in the ternary search tree
with $ n $ nodes. Since type 17, type 18 and type 19 each contains exactly
one protected node, while the other types contain no protected nodes, 
\begin{equation}\label{ZX3}
Z_n=X_{n17}+X_{n18}+X_{n19}.  
\end{equation}
Thus, Theorem \ref{polyathm} implies that 
\begin{align}\label{normalZ}
n^{-1/2} (Z_n-n\mu_Z)\stackrel{d}\rightarrow \N(0,\sigma^2_Z),
\end{align}
with parameters $$\mu_Z= \mu_{17}+\mu_{18}+\mu_{19}
$$ 
and, writing $\Sigma=( \sigma_{i,j})_{i,j=1}^{19}$, 
\begin{align}\label{protectedtrinvariance}
\sigma^2_Z=\sum_{i=17}^{19}\sum_{j=17}^{19}\sigma_{i,j}.
\end{align}
 
Using the normalization in (\ref{normalized}), we see that 
\begin{align}\label{v1trin} v_1=\frac{1}{2100}\cdot 
(1, 5, 9, 9, 6, 7, 36, 20, 42, 42, 15, 30, 126, 28, 48, 35, 42, 45, 84)'
\end{align}
and that
$$u_1=(9, 8, 7, 7, 6, 6, 6, 5, 5, 5, 4, 4, 4, 3, 3, 2, 2,1, 0)'.$$
(As in the binary case, $ u_1=a$ 
since $a\cdot \E\xi_i=1$ for each $i$,
see \cite[Lemma 5.4]{Janson}.) 
Since $\lambda_1=1  $, Theorem \ref{polyathm} and \eqref{v1trin} yield
\begin{align}%\label{expectedtrinar}
\mu_Z=\mu_{17}+\mu_{18}+\mu_{19}=\frac{42}{2100}+\frac{45}{2100}+\frac{84}{2100}=\frac{57}{700}.
\end{align}
Thus, to show Theorem \ref{main} it remains to calculate the sum in (\ref{protectedtrinvariance}).

Since we want to determine the matrix $ \Sigma_I $ in (\ref{Sigma}) we need to determine the matrices $ P_I $ and $ B $.
We have $ P_I=I_{19}-v_1u_1' $, which is a $ 19\times 19 $ matrix that is
shown in (\ref{projtrinar}) in the appendix.
To calculate the matrix $ B $ in (\ref{B}) we need to calculate  
$ B_i=\E(\xi_i\xi_i') $ in (\ref{Bi}). We only describe how to get 
$B_2 $  since the other cases are analogous. From Figure \ref{add2ternary}
(and the explanation of that figure above) it is easy to see that
\begin{gather*}
B_2=\tfrac{1}4\cdot b_1b_1'+\tfrac{3}4\cdot b_2b_2',~~\text{where}, 
\\
\begin{aligned}
   b_1&=(1, -1, 0, 0, 0, 0, 0, 0, 0, 0, 0, 0, 0, 0, 0, 0, 0, 0, 0)'
~~ \text{and}
\\
b_2&=(0, -1, 0, 0, 0, 0, 0, 1, 0, 0, 0, 0, 1, 0, 0, 0, 0, 0, 0)'. 
\end{aligned}
\end{gather*}
Note that $ B_2 $ is a $ 19\times 19 $ matrix.
The matrix $ B $ is shown in (\ref{Bternary}) in the appendix.
Now we can use Mathematica to evaluate the integral in (\ref{Sigma}), which
yields $\Sigma_1$. Finally, $\Sigma=\Sigma_1$ by 
Theorem \ref{polyathm} with $m=1$.
This matrix is given last in the appendix.

By \eqref{ZX3} and \eqref{protectedtrinvariance}, we only need the submatrix
\begin{align}\label{covprotectedternary}&\Sigma_p=
\left(\arraycolsep=2pt
\def\arraystretch{2.0}
\begin{array}{ccc}
 \sigma_{17,17} &  \sigma_{17,18} &
   \sigma_{17,19} \\
  \sigma_{18,17} &  \sigma_{18,18} &
    \sigma_{18,19} \\
  \sigma_{19,17} &  \sigma_{19,18} &
    \sigma_{19,19} \\
\end{array}
\right)=
\left(\arraycolsep=2pt
\def\arraystretch{2.0}
\begin{array}{ccc}
 \phantom{-}\frac{156031}{8085000} & -\frac{826069}{1387386000} &
   \phantom{-}\frac{3453169}{15030015000} \\
 -\frac{826069}{1387386000} & \phantom{-}\frac{2222557}{118918800} &
   -\frac{439517549}{87603516000} \\
 \phantom{-}\frac{3453169}{15030015000} & -\frac{439517549}{87603516000} &
   \phantom{-}\frac{142536826}{12384425625} \\
\end{array}
\right)
.
\end{align}
Summing the $\sigma_{i,j}  $  in \eqref{covprotectedternary}, 
which is equivalent to calculating $ (1,1,1)\Sigma_p(1,1,1)' $, we find
$$ \sigma^2_Z=\sum_{i=17}^{19}\sum_{j=17}^{19}\sigma_{i,j}=  \frac{1692302314867}{43692253605000}, $$
which completes the proof of Theorem \ref{main}.
\qed

\section{Leaves in ternary search trees} \label{Sleaves3}
%We end this section by verifying our results by using our \Polya{} urn
%(with 19 types), as well as a simpler \Polya{} urn, to analyze the number
%of leaves in ternary search trees.  

Recall that a leaf is an internal node without internal children,
i.e., a node that contains at least one key and has no children except 
possibly external ones. 
The proof of Theorem \ref{main} yields also the following theorem.
(The corresponding result for a binary search tree was considered
already by Devroye \cite{Devroye1} using two different methods, one of them
a \Polya{} urn as here.) 

\begin{thm}\label{leaves}
Let $ L_n $ be the number of leaves in a ternary search tree.
Then,
$$
\frac{L_n-\frac{3}{10}n}{\sqrt{n}}\stackrel{d}
\longrightarrow \mathcal{N}\Bigpar{0,\frac{89}{2100}}.  $$
%where $\stackrel{d}\rightarrow  $ denotes convergence in distribution and $ \mathcal{N} $ denotes a normal distribution.
\end{thm}

\begin{proof}[First proof]
Counting the number of
leaves (of the original ternary search tree) in each type in  Figure
\ref{ternarytypes}, we see that the number of 
leaves in a subtree of type $ i $, $i=1,\dots,19$, 
is given by the vector 
%$ l $ in \ref{leaf} (where the
%$ i $th coordinate corresponds to the number of leaves in type $ i $) 
 \begin{align}\label{leaf}
 \ell=( 3, 3, 3, 2, 3, 2, 2, 2, 2, 1, 2, 1, 1, 1, 1, 1, 0, 0, 0)'.
 \end{align}
Hence, $L_n=\ell\cdot X_n$.
By the proof of Theorem \ref{main}, the vector $ X_n $ has asymptotically a
multivariate normal distribution, 
and it follows that 
\begin{align}\label{normalL}
n^{-1/2} (L_n-n\mu_L)\stackrel{d}\rightarrow \N(0,\sigma^2_L)
\end{align}
with, using (\ref{v1trin}) and (\ref{leaf}), 
\begin{align}\label{expectedleaf}
 \mu_L= \ell\cdot v_1=\frac{3}{10},
\end{align}
and, using the covariance matrix  $ \Sigma $ shown in the appendix,
\begin{align}\label{varianceleaf}
 \sigma^{2}_L= \ell'\,\Sigma\, \ell=\frac{89}{2100}. 
\end{align}
\end{proof}

\begin{figure}
\begin{minipage}{0.16\textwidth}
\centering
\raisebox{35pt}{\includegraphics[width=\linewidth]{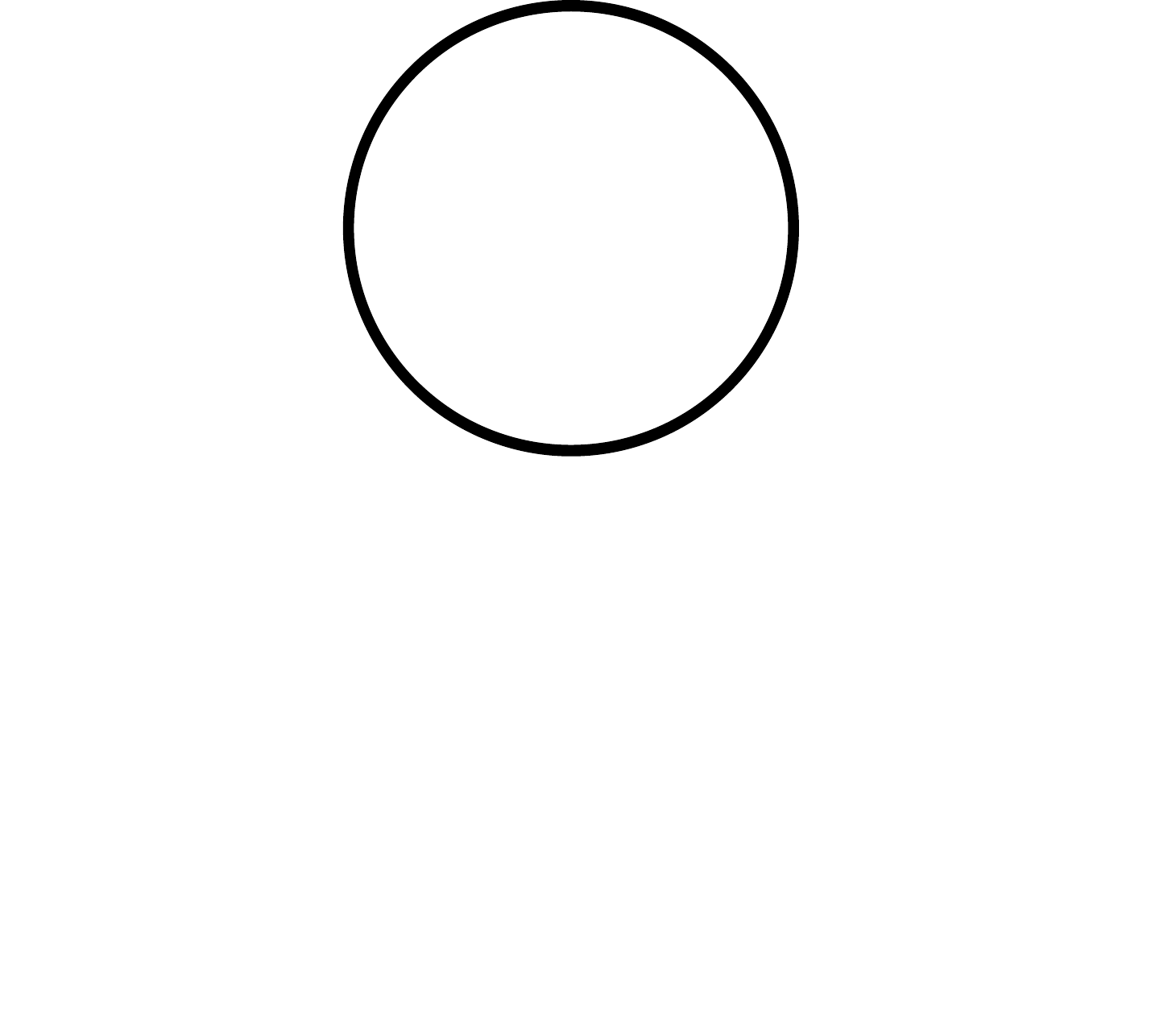}}
\caption{An external node which is not a child of a leaf.}
\label{leaf1}
\end{minipage}\hspace{1cm}
\begin{minipage}{0.16\textwidth}
\centering
\raisebox{30pt}{\includegraphics[width=\linewidth]{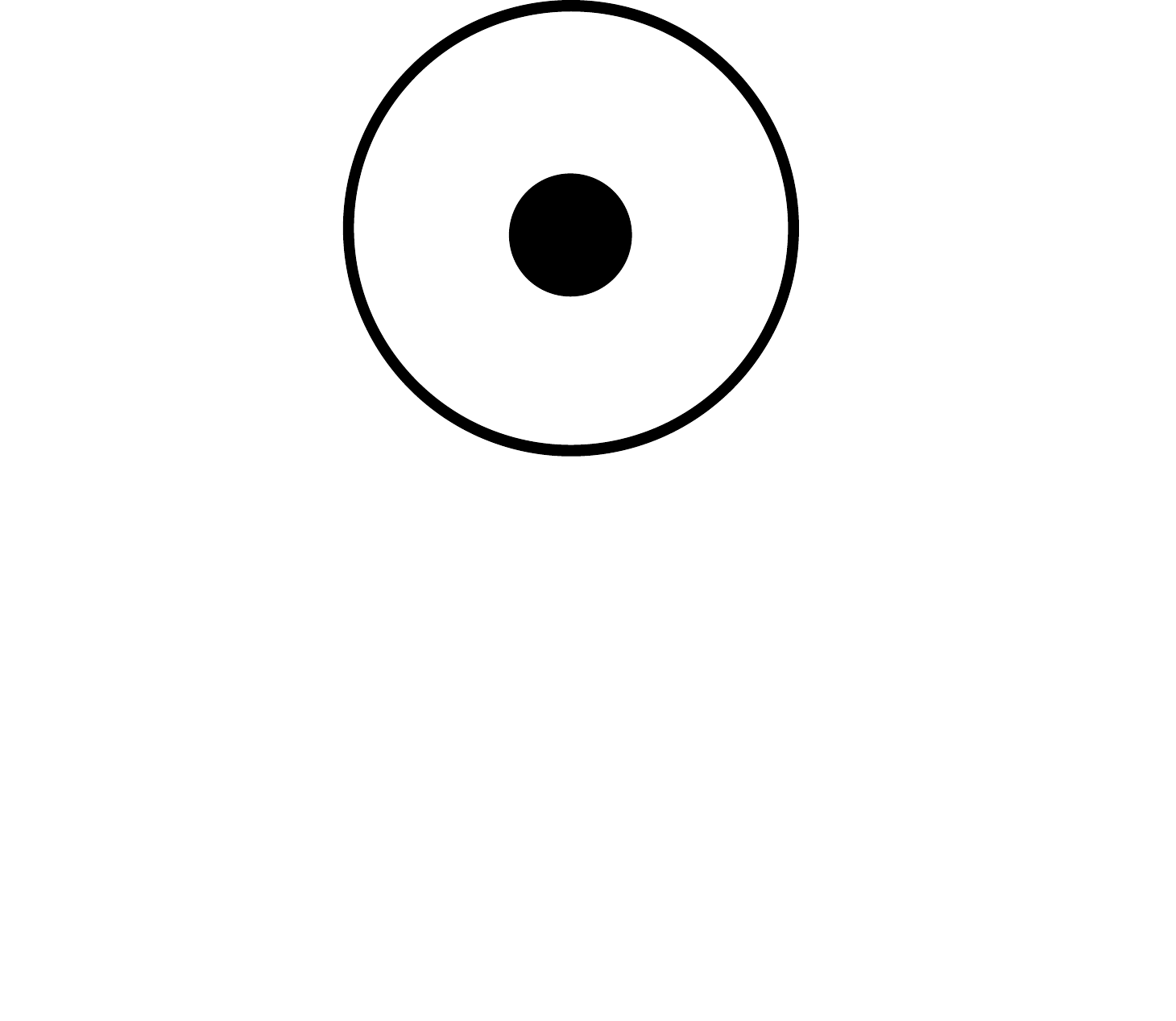}}
\caption{A leaf containing one key.}
\label{leaf2}
\end{minipage}
\hspace{1cm}
\begin{minipage}{0.16\textwidth}
\centering
\includegraphics[width=\linewidth]{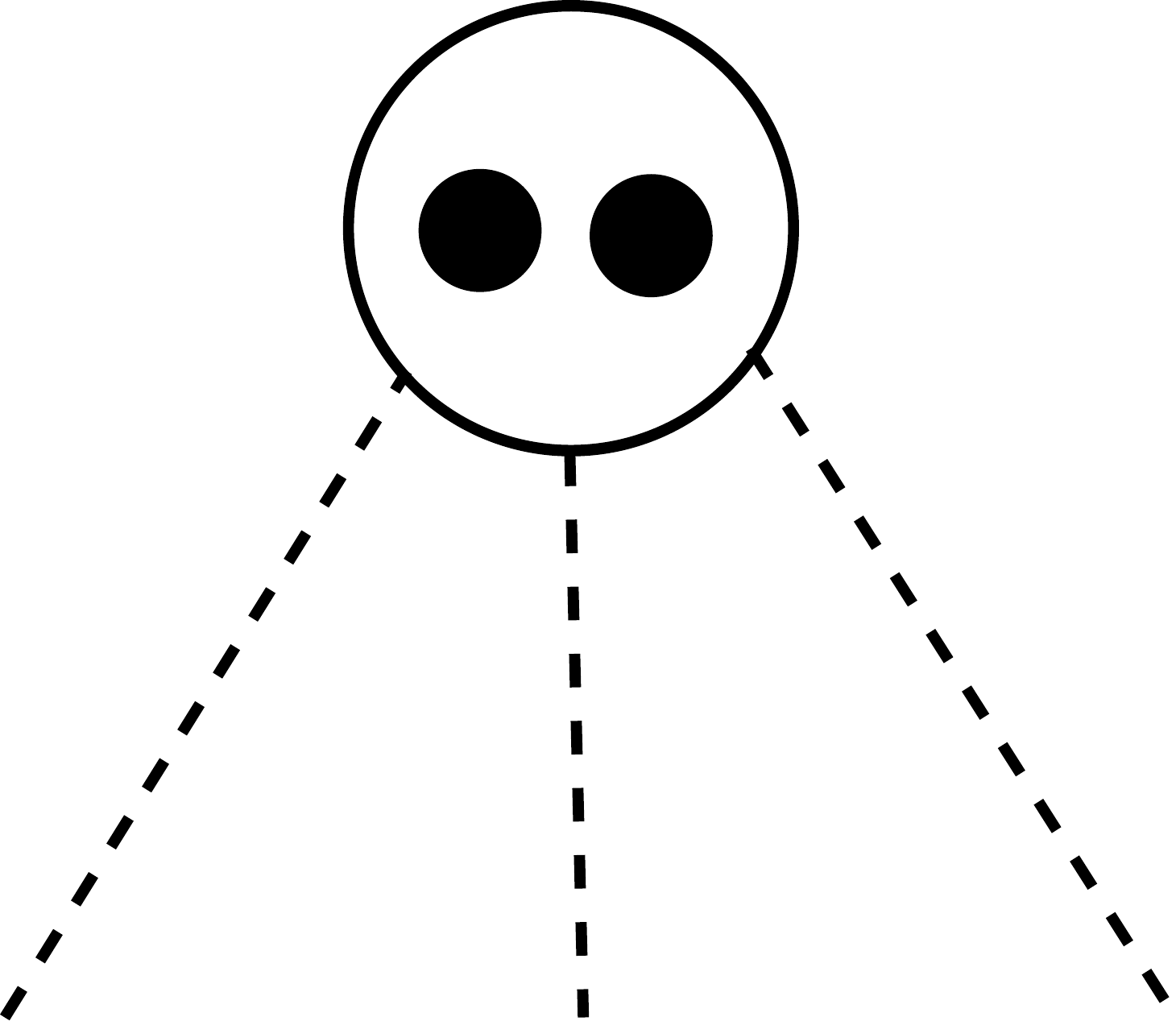}
\caption{A leaf containing two keys and its three external children.}
\label{leaf3}
\end{minipage}\hspace{1cm}
\begin{minipage}{0.16\textwidth}
\centering
\raisebox{12pt}{\includegraphics[width=\linewidth]{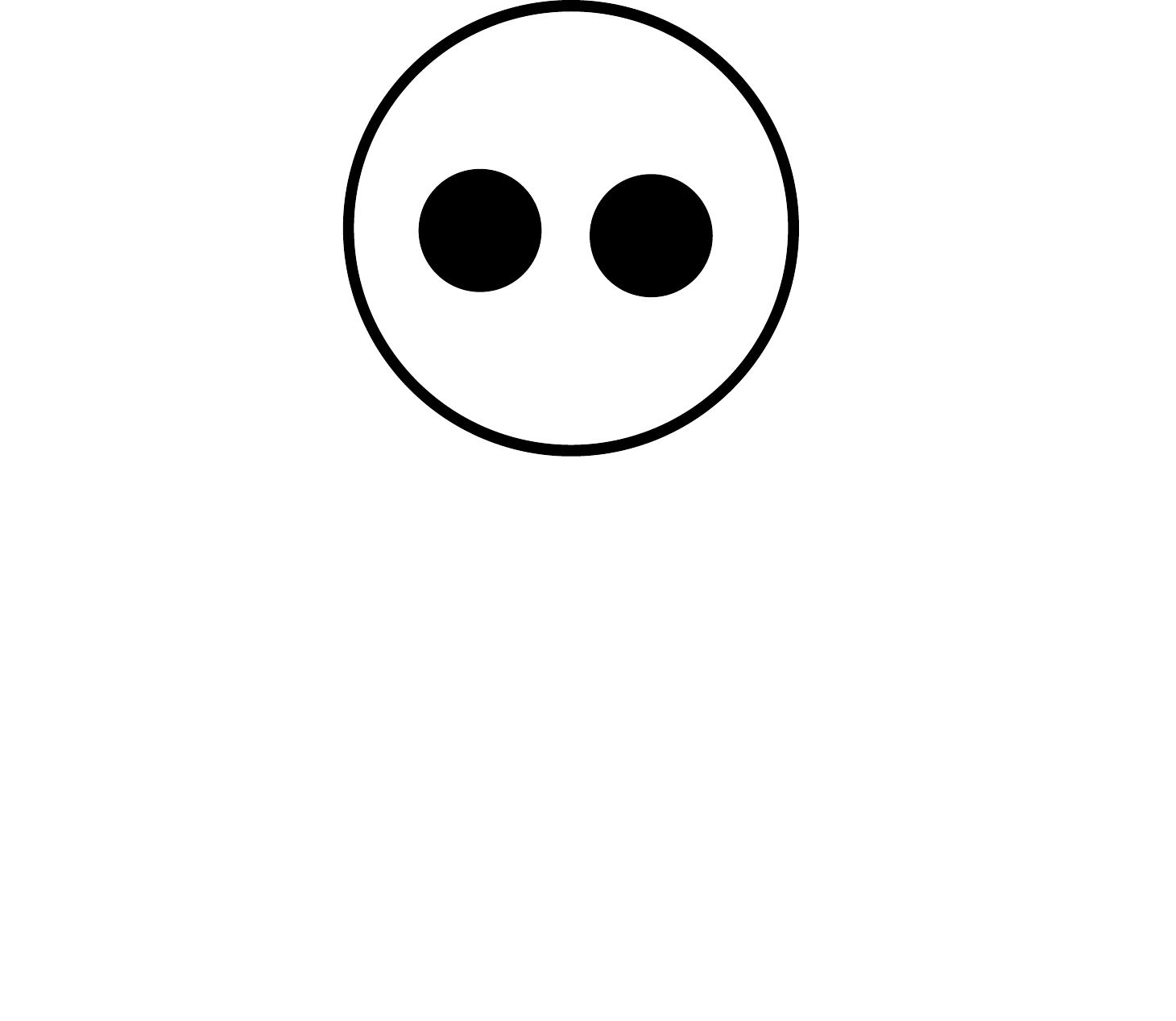}}
\caption{An internal node with two keys which is not a leaf.}
\label{leaf4}
\end{minipage}\hspace{0.8cm}
\caption{The different types characterizing leaves and non-leaves in ternary search trees.}\label{leaftypes}
\end{figure}

However, it is also possible to show Theorem \ref{leaves} using a much
simpler \Polya{} urn process,
where we only need to consider four different
types. We again chop up the ternary search tree into small subtrees, now
using the following types of subtrees.

 Type 1 is an external node which is not a child of a leaf. Type 2 is a node containing one key.  Type 3 is a leaf containing two keys together with its three external children. Type 4 is an internal node
containing two keys which is not a leaf (i.e., it has less than three
external children). The types are shown in Figure \ref{leaftypes}. Note that
all nodes in the ternary search tree belong to exactly one such subtree. 

A ball of type 1 has activity 1; when it is drawn it is replaced by one ball of type 2. A  ball of type 2 has activity 2; when it is drawn it is replaced by one ball of type 3. A ball of type 3 has activity 3; 
when it is drawn it is replaced by one ball of type 2, two balls of type 1 and one ball of type 4. A ball of type 4 has activity 0 and is thus never drawn.
The types that contain leaves are type 2 and type 3.

To simplify we can study another urn using the gaps as balls. 
Type 1 has one gap, type 2 
has two gaps, type 3 has three gaps and type 4 has 0 gaps. We label each gap
with the type it belongs to; thus the gaps have only the three types 1--3.
The gaps evolve as an urn with three types, with all activities 1 and 
the matrix $ A $ in \eqref{A} given by 
\begin{align}\label{gapleaf}
\def\+{\phantom{-}}
\begin{pmatrix}
 -1 & \+0 & \+2 \\
 \+2 & -2 & \+2 \\
 \+0 & \+3 & -3 \\
\end{pmatrix}.
\end{align} 
Since we consider the gaps (with activity 1) it is obvious that all columns
add to 1 (since we always add one ball to the urn). 
The eigenvalues of $ A $ are $ 1,-3,-4 $.
Theorem \ref{simplepolyathm} shows that $ (X_{n1},X_{n2},X_{n3}) $ has
asymptotically a multivariate normal distribution, where $ X_{ni} $ is the
number of balls of type $ i $ in the \Polya{} urn, i.e., the number of gaps
of type $i$. Note that the number of subtrees of Types 1--3 thus is
$(X_{n1}, X_{n2}/2, X_{n3}/3)$, which thus also is asymptotically
multivariate normal.

Since the number of
leaves $ L_n= X_{n2}/2+X_{n3}/3$, 
it follows that
$ L_n $ has asymptotically a normal distribution \eqref{normalL}.

To find the parameters $\mu_L$ and $\gss_L$,
we note that 
right eigenvectors of $ A $ 
corresponding
to the eigenvalues $ 1,-3,-4 $ are: 
\begin{align}\label{eigenleaftrin}
\def\+{\phantom{-}}
\frac{1}{10}\begin{pmatrix} %1
  3\\4\\3
\end{pmatrix}
,\;
\frac{1}{2}\begin{pmatrix} %0
 -1\\\+0\\\+1
\end{pmatrix}
,\;
\frac{1}{5}\begin{pmatrix} %-2
-2 \\ -1\\ \+3
\end{pmatrix},
\end{align}
and corresponding left eigenvectors of $ A $ are:
\begin{align}\label{eigenleaftrinleft}
\def\+{\phantom{-}}
\begin{pmatrix} %1
  1\\1\\1
\end{pmatrix}
,\;
\begin{pmatrix} %0
-3\\\+3\\-1
\end{pmatrix}
,\;
\begin{pmatrix} %-2
  \+2\\-3\\\+2
\end{pmatrix}.
\end{align}
Note that we have scaled the eigenvectors so that $u_i\cdot v_j=\gd_{ij}$
and (\ref{normalized}) holds.
We have $ a=(1,1,1)' $.
Since type 2 has two gaps and one leaf and type 3 has three gaps and one leaf, it follows that
$$ \mu_L=\mu_{2}+\mu_{3}=\frac1{10}(3, 4, 3)\cdot\Bigpar{0,\frac12,\frac13}
=\frac{3}{10}, $$ corresponding to 
(\ref{expectedleaf}).
By calculating $ B $, we get from Theorem \ref{simplepolyathm}, that the
covariance matrix $ \Sigma $  
%with elements $ \sigma_{i,j} $ for $i,j\in\{1,\dots,3\} $ 
is given by
\begin{align}\label{covleafternary}
%\Sigma=
%\arraycolsep=2pt
\def\arraystretch{1.4}\left(
\begin{array}{rrr}
 \frac{479}{2100} & -\frac{7}{150} & -\frac{127}{700} \\[2pt]
 -\frac{7}{150} & \frac{32}{75} & -\frac{19}{50} \\
 -\frac{127}{700} & -\frac{19}{50} & \frac{393}{700} \\
\end{array}
\right).
\end{align}
We thus obtain
\begin{align}\label{varternaryleaf}
\def\arraystretch{1.4}
\sigma^2_L=\bigpar{0,\tfrac12,\tfrac13}
\left(
\begin{array}{rrr}
 \frac{479}{2100} & -\frac{7}{150} & -\frac{127}{700} \\
 -\frac{7}{150} & \frac{32}{75} & -\frac{19}{50} \\
 -\frac{127}{700} & -\frac{19}{50} & \frac{393}{700} \\
\end{array}
\right) 
\begin{pmatrix}
0\\ \frac12 \\ \frac13  
\end{pmatrix}
=\frac{89}{2100}
\end{align}
(corresponding to 
(\ref{varianceleaf})),
which completes the proof of Theorem \ref{leaves} with the simpler \Polya{}
urn model. 

The fact that we have obtained the asymptotic variance in \refT{leaves} in
two different ways, where one uses $\Sigma$ in \refS{S3} and the appendix,  
is also a partial verification of the computer calculations yielding
$\Sigma$.

\section{Higher $m$}\label{Sm}

\subsection{%Eigenvalues of the matrix $ A $ for 
The \Polya{} urn defined in \refS{Polya}}\label{AW}

The \Polya{} urn defined in \refS{Polya} can be used for any given $m$,
although the size of the matrices used in the calculations grow rapidly with
$m$. (For $m=4$ we have 69 types; for $m=10$ we would have 184755.)
However, the central condition $\Re\gl<\gl_1/2$ is not satisfied for large $m$.
We do not know any general formula for the eigenvalues of the matrix $A$,
but some of them are given as follows.

\begin{Lemma}\label{Lroot}
Let $m\ge2$. Then every root of the polynomial
\begin{equation}\label{phi}
  \phi_m(\gl):=\prod_{i=1}^{m-1}(\gl+i)-m!
\end{equation}
  is an eigenvalue of the matrix $A$ for the \Polya{} urn in \refS{Polya}.
\end{Lemma}

\begin{proof}
Let 
$V_{in}$ be the number
of nodes containing exactly $i$ keys (thus $V_{0n}$ is the number of
external nodes), and consider the vector
$W_n=(W_{1,n},\dots,W_{m-1,n})$ where $W_{i,n}=iV_{i-1,n}$; thus $W_{i,n}$ is
the total number of gaps at nodes with $i$ gaps.
The random vector $W_n$ can also be described by a \Polya{} urn, see e.g.,
\cite[Example 7.8]{Janson} and
\cite[Section 8.1.3]{Mahmoud:Polya}; 
we denote the activity vector and the matrix
\eqref{A} for this urn by $a_W=(1,\dots,1)'$ and $A_W$. This means that the
expected 
change of the two vectors when a new key is added are given by
\begin{align}
  \E(X_{n+1}-X_n\mid X_n)
&=\frac{A X_n}{a\cdot X_n}
=\frac{A X_n}{n+1}, \label{ax}
\\
  \E(W_{n+1}-W_n\mid W_n)&=\frac{A_W W_n}{a_W\cdot W_n}
=\frac{A_W W_n}{n+1}.\label{ay}
\end{align}
Furthermore, the vector $X_n$ determines the number of nodes with different
numbers of keys, so there is a linear map $W_n = T X_n$. Consequently, by
\eqref{ax}--\eqref{ay}, for any $X_n$,
\begin{equation*}
  \begin{split}
TAX_n &= (n+1)T\E(X_{n+1}-X_n\mid X_n)
=(n+1)\E(W_{n+1}-W_n\mid X_n) = A_W W_n 
\\&= A_WTX_n,
  \end{split}
\end{equation*}
and thus $TA = A_W T$. 

The $ (m-1)\times (m-1 )$ matrix $ A_W $  is constructed as follows. 
Let $ a_{i,i}=-i$ for $ i\in \{1,\dots,m-1\} $, 
 $ a_{i,i-1}=i$ for $ i\in \{2,\dots,m\} $, 
$ a_{1,m-1}=m $ and all other elements $ a_{i,j}=0 $. 
I.e.,
 \begin{align}\label{AWmatrix}
A_W=\left(
\begin{array}{cccccc}
 -1 & 0 & 0 & \dots & 0 & m  \\
2 & -2 & 0 & \dots & 0 & 0 \\
0 & 3 &-3 & \dots & 0 & 0 \\
0 & 0 & 4 & \dots & 0 & 0 \\
\vdots & \vdots & \vdots & \ddots & \vdots & \vdots \\
0 & 0 & 0 &\dots & m-1 & -(m-1) 
\end{array}
\right).
\end{align}
As is well-known, 
the matrix $A_W$ has characteristic polynomial
$  \phi_m(\gl)$,
see e.g., \cite[Example 7.8]{Janson} or
\cite[Section 8.1.3]{Mahmoud:Polya}.
In particular, $0$ is not an
eigenvalue so $A_W$ is non-singular. The column vectors of $A_W$ are in the
range of $T$, and thus $T$ is onto.

Suppose that $\gl$ is a root of $\phi_m(\gl)=0$. Then $\gl$ is an eigenvalue
of $A_W$ and thus there exists a left eigenvector $u$ with 
$u' A_W = \gl u'$. Consequently,
\begin{equation}
  u'TA = u'A_W T = \gl u'T,
\end{equation}
so $(u'T)'=T'u$ is a left eigenvector of $A$. Since $T$ is onto, $T'$ is
injective and thus $T'u\neq0$. This shows that $\gl$ is an eigenvalue of $A$
too.   
\end{proof}

Recall that $\gl_1=1$ for the matrix $A$, since the total activity increases
by 1 at each step.
Let $\gl_1,\gl_2,\dots,\gl_{m-1}$ 
be the roots of \eqref{phi} in order of decreasing real parts.
It is well-known that $\gl_1=1$ and, moreover, that $\Re \gl_2\le1/2$ if and
only if $m\le26$, see \cite{MahmoudPittel} and \cite{FillKapur}.
Consequently, if $m\ge 27$, then \refL{Lroot} shows that
$A$ has an eigenvalue $\gl=\gl_2\neq\gl_1$ with $\Re \gl_2>1/2$,
and then $X_n$ is \emph{not} asymptotically normal. (See \cite{Janson} for
general results suggesting this, and \cite{ChernHwang} for a rigorous proof
in the present case, showing that the total number of internal nodes is
\emph{not} asymptotically normal.) Furthermore, if $\ga:=\Re\gl_2>1/2$, 
then $(X_n-\E X_n)/n^{\ga}$ is stochastically bounded, but has no limit in
distribution (the distribution oscillates), see
\cite{ChernHwang,ChauvinPouyanne,Janson}.

Some exceptional linear combinations of the variables $X_{ni}$ are
asymptotically normal also in such cases \cite{Janson}, but we conjecture
that for any $m\ge27$,
the number of protected nodes is not one of these exceptional cases
and that it has the same non-normal behaviour as just described for the
number of internal nodes.

On the other hand, if $m\le26$, although $A$ has a much larger dimension
that $A_W$, and thus presumably many more eigenvalues, we conjecture that 
all additional eigenvalues also have $\Re\gl<1/2$, so that \refT{polyathm}
applies showing that the number of protected vertices is asymptotically normal,
with asymptotic variance linear in $n$, just as for $m=2$ and 3 in Theorems
\ref{binary} and \ref{main}.
(This conjecture has been verified for $m\le6$ by Heimb\"urger
\cite{Heimburger}.) 

\subsection{One-protected nodes and leaves in $ m $-ary search trees.}\label{oneprotected}
As mentioned in Section \ref{intro}, the number of one-protected nodes and
the number of leaves (the complement of the one-protected nodes) are easier
to analyze than the two-protected 
 nodes, and
%for the number of one-protected nodes and the number of leaves, 
we prove normal limit laws for all $ m $-ary search trees where $ m\leq
26$. In these cases we can use a \Polya{} urn that is similar to the \Polya{}
urn that has earlier been used to study the total number of internal nodes
in an $ m $-ary search tree, see e.g.\ Mahmoud \cite{Mahmoud2} and
\cite[Section 8.1.3]{Mahmoud:Polya} or \cite[Example 7.8]{Janson}.
 
We can generalise the study of the number of leaves in ternary search tree
in \refS{Sleaves3}
to arbitrary $ m\ge2 $. (For $m=2$, there are minor modifications in the
formulas below; we leave these to the reader. 
As mentioned above, the case $m=2$ was considered
by Devroye \cite{Devroye1}.)
We have in general $ m+1 $ types, defined in analogy with 
Figure~\ref{leaftypes}: 
Type 1 is as before, Type $i$ with $2\le i\le m-1$ is a leaf
with $i-1$ keys, Type $m$ is a leaf with $m-1$ keys together with its $m$
external children, and Type $m+1$ is an internal non-leaf.

Let $ V_{i,n}'=V_{i,n} $ be the number of nodes containing exactly $ i $
keys for $ i\in \{1,\dots, m-2\} $;
let $ V_{0,n} '$ be the number of nodes
containing 0 keys (external nodes) that are not children of leaves; 
let $V_{m,n}' $ be the number of nodes containing $ m-1 $ keys that are leaves
(i.e., they have only external children);
finally, let $ V_{m+1,n }' $ be the
number of internal nodes that are not leaves (all containing $ m-1 $ keys). 
We consider again another, slightly simpler, urn with the balls representing
the gaps, giving them types $1,\dots,m$, and consider
the vector  $W_n'=(W_{1,n}',\dots, W_{m,n}') $ where $ W_{i,n}'=iV_{i-1,n}'$
is the total number of gaps of type $i$.
The random vector $ W_n' $ can be described by a \Polya{} urn, with all
activities 1.
We denote the $ m\times m $ matrix \eqref{A} for this urn by $ A_L $. 
It is a minor modification of the matrix $ A_W $ described in \refS{AW}, 
see \eqref{AWmatrix}; the entries of $ A_L $ are given by 
$ a_{i,i}=-i $ for $ i\in \{1,\dots,m\} $, 
 $ a_{i,i-1}=i$  for $ i\in \{2,\dots,m\} $, 
$ a_{1,m} =m-1$, $ a_{2,m} =2$, and all other entries $ a_{i,j}=0 $. 
I.e.,
 \begin{align}\label{Aallleaves}
A_L=\left(
\begin{array}{ccccccc}
 -1 & 0 & 0 & \dots & 0 & 0 & m-1 \\
2 & -2 & 0 & \dots & 0 & 0 & 2\\
0 & 3 &-3 & \dots & 0 & 0 & 0\\
0 & 0 & 4 & \dots & 0 & 0 & 0\\
\vdots & \vdots & \vdots & \vdots & \ddots & \vdots & \vdots\\
0 & 0 & 0 &\dots & m-1 & -(m-1) & 0 \\
0 & 0 & 0 & \dots &0 & m & -m
\end{array}
\right).
\end{align}
We can easily calculate the characteristic polynomial of $A_L$ and find that
it is 
 \begin{align}\label{characteristicleaves}\phi^L_m(\lambda)=(m+\lambda)\phi_m(\lambda),
\end{align} where $ \phi_m(\lambda) $ is the characteristic polynomial of $A_W $  in \eqref{phi}. 
Thus, $A_L$ has the same eigenvalues as $A$, plus the additional eigenvalue
$ \lambda=-m $.
Since $\phi_m$ has only simple roots \cite[Section 3.3]{Mahmoud:Evolution},
and $-m$ is not one of them, also $\phi^L_m$ has only simple roots.
Hence, $A_L$ has $m$ distinct eigenvalues, and is thus diagonalisable.
 
The largest eigenvalue of $A_L$ is $ \lambda_1=1 $ (as for $A$) 
and this eigenvalue corresponds to the right and left eigenvectors
 \begin{align}\label{eigenallleaves}
v_1=\frac{1}{H_{m}-1}\begin{pmatrix} %1
  \frac{m-1}{2(m+1)}\\\frac{1}{3}\\\frac{1}{4}\\\vdots\\\frac{1}{m-1} \\\frac{1}{m}\\\frac{1}{m+1}
\end{pmatrix},
\qquad
u_1=\begin{pmatrix} %0
 1\\1\\1\\ \vdots \\ 1\\1 \\1
\end{pmatrix},
\end{align}
 where we have normalized so that \eqref{normalized} holds ($ H_m $ denotes
 the $ m $th harmonic number).

Let $ L_{m,n} $ be the number of leaves in an
 $ m $-ary search tree with $n$ keys. 
Then
\begin{equation}\label{lmn}
  L_{m,n} = \sum_{i=1}^{m-1} V'_{i,n} = \sum_{k=2}^m \frac1k W'_{k,n}.
\end{equation}

 \begin{thm}\label{leavesmary}
Suppose that $3\le m\le 26$.
Let $ L_n $ be the number of leaves in an  $ m $-ary search tree.
Then,
\begin{equation}\label{Llim}
\frac{L_n-\mu_{L,{m}} n}{\sqrt{n}}\stackrel{d}
\longrightarrow \mathcal{N}\bigpar{0,\gss_{L,{m}}},  
\end{equation}
where  
\begin{equation}
\mu_{L,{m}}
=\frac{1}{H_{m}-1}\cdot \sum_{k=2} ^{m} \frac{1}{k(k+1)}
=\frac{1}{H_{m}-1}\cdot \frac{m-1}{2(m+1)},
\end{equation}
and $ \gss_{L,{m}} $ can be evaluated as
\begin{equation}
  \gss_{L,m} = \sum_{i,j=2}^m \frac{\gs_{ij}}{ij}
\end{equation}
where $(\gs_{ij})_{i,j=1}^m$ is given by
\eqref{simpleSigma}. 
\end{thm}

\begin{proof}
As said above, for $m\le26$,
$\Re \gl<\gl_1/2=1/2$ for all eigenvalues $\gl\neq\gl_1$ of $A$,	
and thus also of $A_L$. Furthermore, $A$ is diagonalisable.
Hence, \refT{simplepolyathm} applies and shows asymptotic normality 
of $W'_n$. The result follows by \eqref{lmn}, using $v_1$ in
\eqref{eigenallleaves}. 
\end{proof}

\begin{rem} \label{expectedvalue2} Theorem \ref{leavesmary}  implies that 
$ \frac{E(L_n)}{n}\rightarrow \mu_{L,{m}}$, by the same argument as
  in Remark \ref{expectedvalues}. 
\end{rem}

For $m\ge27$, we expect the same non-normal asymptotic behaviour as for the
number of internal nodes \cite{ChernHwang,ChauvinPouyanne}, see \refS{AW}.

For the one-protected nodes we can use the first \Polya{} urn described
above for the leaves, with $ m+1 $ types.
For the leaves we could simplify by considering the
gaps and use a \Polya{} urn with $ m $ types, with all activities 1. 
However, now we also need to consider type $m+1$, which has 0 gaps. So in
the analysis of the one-protected nodes we use the urn with $ m+1 $
different types (as explained in the beginning of this subsection) where
types $ i\in \{1,\dots,m\} $  have activities $ 1,2,\dots, m $ and type $
m+1 $ has activity 0. In this \Polya{} urn, the one-protected nodes
correspond to type $ m+1 $. All other types correspond to  leaves or
external nodes.   
\refT{simplepolyathm} implies the following result (the proof is analogous to the proof of Theorem \ref{leavesmary}).

\begin{thm}\label{oneprotectednodes}
Suppose that $3\le m\le 26$.
Let $ Q_n $ be the number of one-protected nodes in an  $ m $-ary search tree.
Then,
\begin{equation}\label{Qlim}
\frac{Q_n-\mu_{Q,{m}} n}{\sqrt{n}}\stackrel{d}
\longrightarrow \mathcal{N}\bigpar{0,\gss_{Q,{m}}},   
\end{equation}
where  
\begin{equation}
\mu_{Q,{m}}
=\frac{1}{H_{m}-1}\cdot \frac{1}{(m+1)},  
\end{equation}
and $ \gss_{Q,{m}} $ can be evaluated as
\begin{equation}
  \gss_{Q,m} = \gs_{m+1,m+1}
\end{equation}
where $(\gs_{ij})_{i,j=1}^{m+1}$ is given by
\eqref{simpleSigma}. 
\end{thm}

This urn can also be used to study the number of leaves, giving another
proof of \refT{leavesmary}. (Note that $\gs_{ij}$ refers to different urns
and thus has different meanings in Theorems \ref{leavesmary} and
\ref{oneprotectednodes}.) 
Moreover, we can study $L_n$ and $Q_n$ together
and obtain joint asymptotic normality for $m\le26$; 
the covariance $\gs_{LQ,m}$ of the limit variables in \eqref{Llim} and
\eqref{Qlim} equals $\sum_{i=1}^m\gs_{i,m+1}$ with 
$(\gs_{ij})_{i,j=1}^{m+1}$ 
as in \refT{oneprotectednodes}.
In particular, this
implies the well-known asymptotic normality of the total number of internal
nodes $I_n=L_n+Q_n$, see e.g.\
 \cite{MahmoudPittel,Mahmoud:Evolution,LewMahmoud,ChernHwang,Mahmoud2,FillKapur,
Mahmoud:Polya}.  

\begin{example}
For a binary search tree ($m=2$), 
a straightforward calculation of 
the covariance matrix $\Sigma=(\gs_{ij})_{i,j=1}^3$
in \refT{oneprotectednodes} yields
\begin{align}
\def\arraystretch{1.4}
\Sigma=
\left(
\begin{array}{rrr}
 \frac{8}{45} & -\frac{4}{45} & \frac{4}{45} \\
 -\frac{4}{45} & \frac{2}{45} & -\frac{2}{45} \\
 \frac{4}{45} & -\frac{2}{45} & \frac{2}{45} \\
\end{array}
\right).
\end{align}
Hence
\begin{align}\label{varternaryleafallm2}
\sigma^2_{L,2}=(0,1,0)\,\Sigma\, (0,1,0)'=\gs_{22}=\frac{2}{45},
\end{align}
as shown by Devroye \cite{Devroye1}.
Similarly,
%\begin{align}\label{varternaryleafallm2}
$\sigma^2_{Q,2}
%=(0,0,1)\Sigma (0,0,1)'
=\gs_{33}
=\frac{2}{45}$
%\end{align}
and 
$\gs_{LQ,2}=\gs_{23}=-\frac{2}{45}$.
(We have $\gss_{L,2}=\gss_{Q,2}=-\gs_{LQ,2}$
since the total number of internal
nodes $L_n+Q_n=I_n=n$ is deterministic when $m=2$.)
\end{example}

\begin{example}
For a ternary search tree ($m=3$),    
similarly
(cf.~\eqref{covleafternary} for the corresponding urn using the gaps as in 
\refT{leavesmary})
\begin{align}
\def\arraystretch{1.4}
\Sigma=
\left(
\begin{array}{rrrr}
 \frac{479}{2100} & -\frac{7}{300} & -\frac{127}{2100} & \frac{101}{1400} \\
 -\frac{7}{300} & \frac{8}{75} & -\frac{19}{300} & \frac{1}{100} \\
 -\frac{127}{2100} & -\frac{19}{300} & \frac{131}{2100} & -\frac{43}{1400} \\
 \frac{101}{1400} & \frac{1}{100} & -\frac{43}{1400} & \frac{9}{350} \\
\end{array}
\right).
\end{align}
Hence,
cf.~\eqref{varianceleaf} and (\ref{varternaryleaf}), 
\begin{align}\label{varternaryleafallm3}
\sigma^2_{L,3}&=(0,1,1,0)\,\Sigma\,(0,1,1,0)'=\frac{89}{2100},
\\
\sigma^2_{Q,3}&=(0,0,0,1)\,\Sigma\,(0,0,0,1)'=\frac{9}{350},
\\
\sigma_{LQ,3}&=(0,1,1,0)\,\Sigma\,(0,0,0,1)'=-\frac{29}{1400}.
\end{align}
We also obtain the corresponding asymptotic variance
$(0,1,1,1)\,\Sigma\,(0,1,1,1)'=\gss_{L,3}+\gss_{Q,3}+2\gs_{LQ,3}= \frac2{75}$
for the number of internal nodes $L_n+Q_n$, as found by 
Mahmoud and Pittel \cite{MahmoudPittel}.
 \end{example}

\textbf{Acknowledgements:} We would like to thank Hosam M. Mahmoud and
Mark D. Ward for valuable discussions.

\appendix 
\section{Appendix}

\begin{align}\label{projtrinar}
P_I=\tiny{\left(\arraycolsep=2pt\def\arraystretch{2.0}
\begin{array}{rrrrrrrrrrrrrrrrrrr}
 \frac{697}{700} & -\frac{2}{525} & -\frac{1}{300} & -\frac{1}{300} &
   -\frac{1}{350} & -\frac{1}{350} & -\frac{1}{350} & -\frac{1}{420} &
   -\frac{1}{420} & -\frac{1}{420} & -\frac{1}{525} & -\frac{1}{525} &
   -\frac{1}{525} & -\frac{1}{700} & -\frac{1}{700} & -\frac{1}{1050} &
   -\frac{1}{1050} & -\frac{1}{2100} & 0 \\
 -\frac{3}{140} & \frac{103}{105} & -\frac{1}{60} & -\frac{1}{60} &
   -\frac{1}{70} & -\frac{1}{70} & -\frac{1}{70} & -\frac{1}{84} & -\frac{1}{84}
   & -\frac{1}{84} & -\frac{1}{105} & -\frac{1}{105} & -\frac{1}{105} &
   -\frac{1}{140} & -\frac{1}{140} & -\frac{1}{210} & -\frac{1}{210} &
   -\frac{1}{420} & 0 \\
 -\frac{27}{700} & -\frac{6}{175} & \frac{97}{100} & -\frac{3}{100} &
   -\frac{9}{350} & -\frac{9}{350} & -\frac{9}{350} & -\frac{3}{140} &
   -\frac{3}{140} & -\frac{3}{140} & -\frac{3}{175} & -\frac{3}{175} &
   -\frac{3}{175} & -\frac{9}{700} & -\frac{9}{700} & -\frac{3}{350} &
   -\frac{3}{350} & -\frac{3}{700} & 0 \\
 -\frac{27}{700} & -\frac{6}{175} & -\frac{3}{100} & \frac{97}{100} &
   -\frac{9}{350} & -\frac{9}{350} & -\frac{9}{350} & -\frac{3}{140} &
   -\frac{3}{140} & -\frac{3}{140} & -\frac{3}{175} & -\frac{3}{175} &
   -\frac{3}{175} & -\frac{9}{700} & -\frac{9}{700} & -\frac{3}{350} &
   -\frac{3}{350} & -\frac{3}{700} & 0 \\
 -\frac{9}{350} & -\frac{4}{175} & -\frac{1}{50} & -\frac{1}{50} &
   \frac{172}{175} & -\frac{3}{175} & -\frac{3}{175} & -\frac{1}{70} &
   -\frac{1}{70} & -\frac{1}{70} & -\frac{2}{175} & -\frac{2}{175} &
   -\frac{2}{175} & -\frac{3}{350} & -\frac{3}{350} & -\frac{1}{175} &
   -\frac{1}{175} & -\frac{1}{350} & 0 \\
 -\frac{3}{100} & -\frac{2}{75} & -\frac{7}{300} & -\frac{7}{300} &
   -\frac{1}{50} & \frac{49}{50} & -\frac{1}{50} & -\frac{1}{60} & -\frac{1}{60}
   & -\frac{1}{60} & -\frac{1}{75} & -\frac{1}{75} & -\frac{1}{75} &
   -\frac{1}{100} & -\frac{1}{100} & -\frac{1}{150} & -\frac{1}{150} &
   -\frac{1}{300} & 0 \\
 -\frac{27}{175} & -\frac{24}{175} & -\frac{3}{25} & -\frac{3}{25} &
   -\frac{18}{175} & -\frac{18}{175} & \frac{157}{175} & -\frac{3}{35} &
   -\frac{3}{35} & -\frac{3}{35} & -\frac{12}{175} & -\frac{12}{175} &
   -\frac{12}{175} & -\frac{9}{175} & -\frac{9}{175} & -\frac{6}{175} &
   -\frac{6}{175} & -\frac{3}{175} & 0 \\
 -\frac{3}{35} & -\frac{8}{105} & -\frac{1}{15} & -\frac{1}{15} & -\frac{2}{35}
   & -\frac{2}{35} & -\frac{2}{35} & \frac{20}{21} & -\frac{1}{21} &
   -\frac{1}{21} & -\frac{4}{105} & -\frac{4}{105} & -\frac{4}{105} &
   -\frac{1}{35} & -\frac{1}{35} & -\frac{2}{105} & -\frac{2}{105} &
   -\frac{1}{105} & 0 \\
 -\frac{9}{50} & -\frac{4}{25} & -\frac{7}{50} & -\frac{7}{50} & -\frac{3}{25} &
   -\frac{3}{25} & -\frac{3}{25} & -\frac{1}{10} & \frac{9}{10} & -\frac{1}{10}
   & -\frac{2}{25} & -\frac{2}{25} & -\frac{2}{25} & -\frac{3}{50} &
   -\frac{3}{50} & -\frac{1}{25} & -\frac{1}{25} & -\frac{1}{50} & 0 \\
 -\frac{9}{50} & -\frac{4}{25} & -\frac{7}{50} & -\frac{7}{50} & -\frac{3}{25} &
   -\frac{3}{25} & -\frac{3}{25} & -\frac{1}{10} & -\frac{1}{10} & \frac{9}{10}
   & -\frac{2}{25} & -\frac{2}{25} & -\frac{2}{25} & -\frac{3}{50} &
   -\frac{3}{50} & -\frac{1}{25} & -\frac{1}{25} & -\frac{1}{50} & 0 \\
 -\frac{9}{140} & -\frac{2}{35} & -\frac{1}{20} & -\frac{1}{20} & -\frac{3}{70}
   & -\frac{3}{70} & -\frac{3}{70} & -\frac{1}{28} & -\frac{1}{28} &
   -\frac{1}{28} & \frac{34}{35} & -\frac{1}{35} & -\frac{1}{35} &
   -\frac{3}{140} & -\frac{3}{140} & -\frac{1}{70} & -\frac{1}{70} &
   -\frac{1}{140} & 0 \\
 -\frac{9}{70} & -\frac{4}{35} & -\frac{1}{10} & -\frac{1}{10} & -\frac{3}{35} &
   -\frac{3}{35} & -\frac{3}{35} & -\frac{1}{14} & -\frac{1}{14} & -\frac{1}{14}
   & -\frac{2}{35} & \frac{33}{35} & -\frac{2}{35} & -\frac{3}{70} &
   -\frac{3}{70} & -\frac{1}{35} & -\frac{1}{35} & -\frac{1}{70} & 0 \\
 -\frac{27}{50} & -\frac{12}{25} & -\frac{21}{50} & -\frac{21}{50} &
   -\frac{9}{25} & -\frac{9}{25} & -\frac{9}{25} & -\frac{3}{10} & -\frac{3}{10}
   & -\frac{3}{10} & -\frac{6}{25} & -\frac{6}{25} & \frac{19}{25} &
   -\frac{9}{50} & -\frac{9}{50} & -\frac{3}{25} & -\frac{3}{25} & -\frac{3}{50}
   & 0 \\
 -\frac{3}{25} & -\frac{8}{75} & -\frac{7}{75} & -\frac{7}{75} & -\frac{2}{25} &
   -\frac{2}{25} & -\frac{2}{25} & -\frac{1}{15} & -\frac{1}{15} & -\frac{1}{15}
   & -\frac{4}{75} & -\frac{4}{75} & -\frac{4}{75} & \frac{24}{25} &
   -\frac{1}{25} & -\frac{2}{75} & -\frac{2}{75} & -\frac{1}{75} & 0 \\
 -\frac{36}{175} & -\frac{32}{175} & -\frac{4}{25} & -\frac{4}{25} &
   -\frac{24}{175} & -\frac{24}{175} & -\frac{24}{175} & -\frac{4}{35} &
   -\frac{4}{35} & -\frac{4}{35} & -\frac{16}{175} & -\frac{16}{175} &
   -\frac{16}{175} & -\frac{12}{175} & \frac{163}{175} & -\frac{8}{175} &
   -\frac{8}{175} & -\frac{4}{175} & 0 \\
 -\frac{3}{20} & -\frac{2}{15} & -\frac{7}{60} & -\frac{7}{60} & -\frac{1}{10} &
   -\frac{1}{10} & -\frac{1}{10} & -\frac{1}{12} & -\frac{1}{12} & -\frac{1}{12}
   & -\frac{1}{15} & -\frac{1}{15} & -\frac{1}{15} & -\frac{1}{20} &
   -\frac{1}{20} & \frac{29}{30} & -\frac{1}{30} & -\frac{1}{60} & 0 \\
 -\frac{9}{50} & -\frac{4}{25} & -\frac{7}{50} & -\frac{7}{50} & -\frac{3}{25} &
   -\frac{3}{25} & -\frac{3}{25} & -\frac{1}{10} & -\frac{1}{10} & -\frac{1}{10}
   & -\frac{2}{25} & -\frac{2}{25} & -\frac{2}{25} & -\frac{3}{50} &
   -\frac{3}{50} & -\frac{1}{25} & \frac{24}{25} & -\frac{1}{50} & 0 \\
 -\frac{27}{140} & -\frac{6}{35} & -\frac{3}{20} & -\frac{3}{20} & -\frac{9}{70}
   & -\frac{9}{70} & -\frac{9}{70} & -\frac{3}{28} & -\frac{3}{28} &
   -\frac{3}{28} & -\frac{3}{35} & -\frac{3}{35} & -\frac{3}{35} &
   -\frac{9}{140} & -\frac{9}{140} & -\frac{3}{70} & -\frac{3}{70} &
   \frac{137}{140} & 0 \\
 -\frac{9}{25} & -\frac{8}{25} & -\frac{7}{25} & -\frac{7}{25} & -\frac{6}{25} &
   -\frac{6}{25} & -\frac{6}{25} & -\frac{1}{5} & -\frac{1}{5} & -\frac{1}{5} &
   -\frac{4}{25} & -\frac{4}{25} & -\frac{4}{25} & -\frac{3}{25} & -\frac{3}{25}
   & -\frac{2}{25} & -\frac{2}{25} & -\frac{1}{25} & 1 \\
\end{array}
\right)
 }\end{align}

\begin{align}\label{Bternary}
B=\tiny{\left(\arraycolsep=2pt\def\arraystretch{2.0}
\begin{array}{rrrrrrrrrrrrrrrrrrr}
 \frac{19}{2100} & -\frac{1}{210} & 0 & 0 & 0 & -\frac{3}{700} & 0 & 0 & 0 & 0 &
   0 & 0 & -\frac{3}{700} & 0 & 0 & 0 & 0 & 0 & 0 \\
 -\frac{1}{210} & \frac{17}{420} & -\frac{3}{175} & -\frac{3}{700} & 0 & 0 & 0 &
   -\frac{1}{70} & 0 & 0 & 0 & 0 & -\frac{1}{70} & 0 & 0 & 0 & 0 & 0 & 0 \\
 0 & -\frac{3}{175} & \frac{9}{140} & 0 & -\frac{3}{175} & 0 & -\frac{3}{175} &
   0 & 0 & 0 & -\frac{9}{700} & 0 & -\frac{9}{700} & 0 & 0 & 0 & 0 & 0 & 0 \\
 0 & -\frac{3}{700} & 0 & \frac{9}{140} & 0 & 0 & -\frac{6}{175} & 0 & 0 & 0 & 0
   & -\frac{9}{350} & -\frac{9}{350} & 0 & 0 & 0 & 0 & 0 & 0 \\
 0 & 0 & -\frac{3}{175} & 0 & \frac{13}{350} & 0 & 0 & 0 & -\frac{1}{50} & 0 & 0
   & 0 & 0 & 0 & 0 & 0 & 0 & 0 & 0 \\
 -\frac{3}{700} & 0 & 0 & 0 & 0 & \frac{13}{300} & 0 & -\frac{2}{105} & 0 & 0 &
   0 & 0 & -\frac{11}{700} & -\frac{1}{50} & 0 & 0 & 0 & 0 & 0 \\
 0 & 0 & -\frac{3}{175} & -\frac{6}{175} & 0 & 0 & \frac{39}{175} & 0 &
   -\frac{2}{25} & -\frac{1}{25} & 0 & 0 & -\frac{9}{175} & 0 & -\frac{9}{175} &
   0 & 0 & 0 & 0 \\
 0 & -\frac{1}{70} & 0 & 0 & 0 & -\frac{2}{105} & 0 & \frac{11}{105} & 0 & 0 &
   -\frac{1}{35} & -\frac{1}{70} & -\frac{1}{70} & 0 & 0 & -\frac{1}{35} & 0 & 0
   & 0 \\
 0 & 0 & 0 & 0 & -\frac{1}{50} & 0 & -\frac{2}{25} & 0 & \frac{11}{50} & 0 & 0 &
   0 & -\frac{3}{25} & 0 & 0 & 0 & 0 & 0 & 0 \\
 0 & 0 & 0 & 0 & 0 & 0 & -\frac{1}{25} & 0 & 0 & \frac{11}{50} & 0 & 0 &
   -\frac{9}{50} & 0 & 0 & 0 & -\frac{3}{50} & 0 & 0 \\
 0 & 0 & -\frac{9}{700} & 0 & 0 & 0 & 0 & -\frac{1}{35} & 0 & 0 & \frac{9}{140}
   & 0 & \frac{9}{700} & 0 & -\frac{4}{175} & 0 & 0 & 0 & 0 \\
 0 & 0 & 0 & -\frac{9}{350} & 0 & 0 & 0 & -\frac{1}{70} & 0 & 0 & 0 &
   \frac{9}{70} & -\frac{3}{175} & 0 & -\frac{8}{175} & 0 & 0 & -\frac{3}{70} &
   0 \\
 -\frac{3}{700} & -\frac{1}{70} & -\frac{9}{700} & -\frac{9}{350} & 0 &
   -\frac{11}{700} & -\frac{9}{175} & -\frac{1}{70} & -\frac{3}{25} &
   -\frac{9}{50} & \frac{9}{700} & -\frac{3}{175} & \frac{27}{50} &
   -\frac{1}{50} & \frac{9}{175} & \frac{1}{35} & \frac{3}{50} & \frac{3}{70} &
   \frac{1}{25} \\
 0 & 0 & 0 & 0 & 0 & -\frac{1}{50} & 0 & 0 & 0 & 0 & 0 & 0 & -\frac{1}{50} &
   \frac{7}{75} & 0 & -\frac{1}{30} & 0 & 0 & -\frac{1}{25} \\
 0 & 0 & 0 & 0 & 0 & 0 & -\frac{9}{175} & 0 & 0 & 0 & -\frac{4}{175} &
   -\frac{8}{175} & \frac{9}{175} & 0 & \frac{4}{25} & 0 & -\frac{1}{25} & 0 & 0
   \\
 0 & 0 & 0 & 0 & 0 & 0 & 0 & -\frac{1}{35} & 0 & 0 & 0 & 0 & \frac{1}{35} &
   -\frac{1}{30} & 0 & \frac{1}{12} & 0 & -\frac{3}{140} & 0 \\
 0 & 0 & 0 & 0 & 0 & 0 & 0 & 0 & 0 & -\frac{3}{50} & 0 & 0 & \frac{3}{50} & 0 &
   -\frac{1}{25} & 0 & \frac{1}{10} & 0 & 0 \\
 0 & 0 & 0 & 0 & 0 & 0 & 0 & 0 & 0 & 0 & 0 & -\frac{3}{70} & \frac{3}{70} & 0 &
   0 & -\frac{3}{140} & 0 & \frac{9}{140} & 0 \\
 0 & 0 & 0 & 0 & 0 & 0 & 0 & 0 & 0 & 0 & 0 & 0 & \frac{1}{25} & -\frac{1}{25} &
   0 & 0 & 0 & 0 & \frac{1}{25} \\
\end{array}
\right)}\end{align}

\newpage
\thispagestyle{empty}
\begin{changemargin}{-1cm}{-1cm}

\rotatebox{90}{$\Sigma=$\\
$\frac{1}{2100}\cdot\tiny{\left(
\arraycolsep=-0.7pt\def\arraystretch{2.9}\begin{array}{rrrrrrrrrrrrrrrrrrr}
 \frac{41376542411}{41611670100} & -\frac{23408479}{876035160} &
   -\frac{24963443}{547521975} & -\frac{24963443}{547521975} &
   -\frac{338777}{11711700} & -\frac{288002937}{9247037800} &
   -\frac{338777}{1951950} & -\frac{61814939}{730029300} &
   -\frac{2497361}{12882870} & -\frac{2497361}{12882870} &
   -\frac{6678998}{109504395} & -\frac{13357996}{109504395} &
   -\frac{72324583}{128828700} & -\frac{720772978}{10402917525} &
   -\frac{291188}{1533675} & -\frac{343662071}{4380175800} &
   -\frac{7052429}{42942900} & -\frac{10190867}{109504395} &
   -\frac{68043227}{3963016200} \\
 -\frac{23408479}{876035160} & \frac{106759904}{21900879} &
   -\frac{10916617}{51531480} & -\frac{10916617}{51531480} &
   -\frac{244879}{1840410} & -\frac{1337363}{9733724} & -\frac{244879}{306735} &
   -\frac{53606911}{146005860} & -\frac{3248263}{3680820} &
   -\frac{3248263}{3680820} & -\frac{2680231}{10306296} &
   -\frac{2680231}{5153148} & -\frac{465730}{184041} &
   -\frac{59525563}{219008790} & -\frac{3427337}{4294290} &
   -\frac{126531917}{438017580} & -\frac{1469917}{2147145} &
   -\frac{3279013}{10306296} & \frac{3706671}{97337240} \\
 -\frac{24963443}{547521975} & -\frac{10916617}{51531480} &
   \frac{12239749}{1415700} & -\frac{501551}{1415700} & -\frac{2178359}{9909900}
   & -\frac{312293407}{1460058600} & -\frac{2178359}{1651650} &
   -\frac{3411301}{6134700} & -\frac{20368}{14157} & -\frac{20368}{14157} &
   -\frac{757667}{1981980} & -\frac{757667}{990990} & -\frac{443321}{108900} &
   -\frac{196315124}{547521975} & -\frac{941257}{825825} &
   -\frac{86297573}{257657400} & -\frac{3159973}{3303300} &
   -\frac{85739}{283140} & \frac{56654959}{208579800} \\
 -\frac{24963443}{547521975} & -\frac{10916617}{51531480} &
   -\frac{501551}{1415700} & \frac{12239749}{1415700} & -\frac{2178359}{9909900}
   & -\frac{312293407}{1460058600} & -\frac{2178359}{1651650} &
   -\frac{3411301}{6134700} & -\frac{20368}{14157} & -\frac{20368}{14157} &
   -\frac{757667}{1981980} & -\frac{757667}{990990} & -\frac{443321}{108900} &
   -\frac{196315124}{547521975} & -\frac{941257}{825825} &
   -\frac{86297573}{257657400} & -\frac{3159973}{3303300} &
   -\frac{85739}{283140} & \frac{56654959}{208579800} \\
 -\frac{338777}{11711700} & -\frac{244879}{1840410} & -\frac{2178359}{9909900} &
   -\frac{2178359}{9909900} & \frac{1162547}{198198} & -\frac{1290076}{10735725}
   & -\frac{26641}{33033} & -\frac{3201412}{10735725} & -\frac{94213}{108900} &
   -\frac{94213}{108900} & -\frac{76651}{396396} & -\frac{76651}{198198} &
   -\frac{65447}{27225} & -\frac{456479}{2927925} & -\frac{889403}{1651650} &
   -\frac{3296674}{32207175} & -\frac{26839}{63525} & -\frac{39443}{1981980} &
   \frac{6456887}{21471450} \\
 -\frac{288002937}{9247037800} & -\frac{1337363}{9733724} &
   -\frac{312293407}{1460058600} & -\frac{312293407}{1460058600} &
   -\frac{1290076}{10735725} & \frac{1341943367}{198150810} &
   -\frac{2580152}{3578575} & -\frac{433585133}{730029300} &
   -\frac{29434157}{42942900} & -\frac{29434157}{42942900} &
   -\frac{117464953}{292011720} & -\frac{117464953}{146005860} &
   -\frac{34515419}{21471450} & -\frac{366941303}{495377025} &
   -\frac{8194161}{7157150} & -\frac{676938}{790075} & -\frac{1247859}{1431430}
   & -\frac{59455463}{58402344} & -\frac{3567263197}{3963016200} \\
 -\frac{338777}{1951950} & -\frac{244879}{306735} & -\frac{2178359}{1651650} &
   -\frac{2178359}{1651650} & -\frac{26641}{33033} & -\frac{2580152}{3578575} &
   \frac{343114}{11011} & -\frac{6402824}{3578575} & -\frac{94213}{18150} &
   -\frac{94213}{18150} & -\frac{76651}{66066} & -\frac{76651}{33033} &
   -\frac{130894}{9075} & -\frac{912958}{975975} & -\frac{889403}{275275} &
   -\frac{6593348}{10735725} & -\frac{53678}{21175} & -\frac{39443}{330330} &
   \frac{6456887}{3578575} \\
 -\frac{61814939}{730029300} & -\frac{53606911}{146005860} &
   -\frac{3411301}{6134700} & -\frac{3411301}{6134700} &
   -\frac{3201412}{10735725} & -\frac{433585133}{730029300} &
   -\frac{6402824}{3578575} & \frac{96310504}{5214495} &
   -\frac{33553843}{21471450} & -\frac{33553843}{21471450} &
   -\frac{8757649}{8588580} & -\frac{8757649}{4294290} &
   -\frac{33159457}{10735725} & -\frac{109833091}{60835775} &
   -\frac{10151767}{3578575} & -\frac{298480453}{146005860} &
   -\frac{7463983}{3578575} & -\frac{2909761}{1226940} &
   -\frac{95014589}{52144950} \\
 -\frac{2497361}{12882870} & -\frac{3248263}{3680820} & -\frac{20368}{14157} &
   -\frac{20368}{14157} & -\frac{94213}{108900} & -\frac{29434157}{42942900} &
   -\frac{94213}{18150} & -\frac{33553843}{21471450} & \frac{90452}{2475} &
   -\frac{13498}{2475} & -\frac{439141}{495495} & -\frac{878282}{495495} &
   -\frac{13309}{900} & -\frac{712856}{1288287} & -\frac{25198}{12705} &
   \frac{12017233}{128828700} & -\frac{25111}{23100} & \frac{559703}{495495} &
   \frac{12078779}{4294290} \\
 -\frac{2497361}{12882870} & -\frac{3248263}{3680820} & -\frac{20368}{14157} &
   -\frac{20368}{14157} & -\frac{94213}{108900} & -\frac{29434157}{42942900} &
   -\frac{94213}{18150} & -\frac{33553843}{21471450} & -\frac{13498}{2475} &
   \frac{90452}{2475} & -\frac{439141}{495495} & -\frac{878282}{495495} &
   -\frac{13309}{900} & -\frac{712856}{1288287} & -\frac{25198}{12705} &
   \frac{12017233}{128828700} & -\frac{25111}{23100} & \frac{559703}{495495} &
   \frac{12078779}{4294290} \\
 -\frac{6678998}{109504395} & -\frac{2680231}{10306296} &
   -\frac{757667}{1981980} & -\frac{757667}{1981980} & -\frac{76651}{396396} &
   -\frac{117464953}{292011720} & -\frac{76651}{66066} &
   -\frac{8757649}{8588580} & -\frac{439141}{495495} & -\frac{439141}{495495} &
   \frac{5682427}{396396} & -\frac{263513}{198198} & -\frac{172523}{152460} &
   -\frac{122261819}{109504395} & -\frac{295822}{165165} &
   -\frac{519971}{425880} & -\frac{824749}{660660} & -\frac{538127}{396396} &
   -\frac{240667397}{292011720} \\
 -\frac{13357996}{109504395} & -\frac{2680231}{5153148} & -\frac{757667}{990990}
   & -\frac{757667}{990990} & -\frac{76651}{198198} &
   -\frac{117464953}{146005860} & -\frac{76651}{33033} &
   -\frac{8757649}{4294290} & -\frac{878282}{495495} & -\frac{878282}{495495} &
   -\frac{263513}{198198} & \frac{2709457}{99099} & -\frac{172523}{76230} &
   -\frac{244523638}{109504395} & -\frac{591644}{165165} &
   -\frac{519971}{212940} & -\frac{824749}{330330} & -\frac{538127}{198198} &
   -\frac{240667397}{146005860} \\
 -\frac{72324583}{128828700} & -\frac{465730}{184041} & -\frac{443321}{108900} &
   -\frac{443321}{108900} & -\frac{65447}{27225} & -\frac{34515419}{21471450} &
   -\frac{130894}{9075} & -\frac{33159457}{10735725} & -\frac{13309}{900} &
   -\frac{13309}{900} & -\frac{172523}{152460} & -\frac{172523}{76230} &
   \frac{39227}{450} & -\frac{7955746}{32207175} & \frac{54421}{127050} &
   \frac{161191087}{64414350} & \frac{893}{210} & \frac{215345}{30492} &
   \frac{217039103}{21471450} \\
 -\frac{720772978}{10402917525} & -\frac{59525563}{219008790} &
   -\frac{196315124}{547521975} & -\frac{196315124}{547521975} &
   -\frac{456479}{2927925} & -\frac{366941303}{495377025} &
   -\frac{912958}{975975} & -\frac{109833091}{60835775} &
   -\frac{712856}{1288287} & -\frac{712856}{1288287} &
   -\frac{122261819}{109504395} & -\frac{244523638}{109504395} &
   -\frac{7955746}{32207175} & \frac{36088849118}{1486131075} &
   -\frac{30069938}{10735725} & -\frac{655132571}{156434850} &
   -\frac{19005188}{10735725} & -\frac{522952526}{109504395} &
   -\frac{4662351652}{495377025} \\
 -\frac{291188}{1533675} & -\frac{3427337}{4294290} & -\frac{941257}{825825} &
   -\frac{941257}{825825} & -\frac{889403}{1651650} & -\frac{8194161}{7157150} &
   -\frac{889403}{275275} & -\frac{10151767}{3578575} & -\frac{25198}{12705} &
   -\frac{25198}{12705} & -\frac{295822}{165165} & -\frac{591644}{165165} &
   \frac{54421}{127050} & -\frac{30069938}{10735725} & \frac{11949972}{275275} &
   -\frac{61733263}{21471450} & -\frac{123237}{42350} & -\frac{483403}{165165} &
   -\frac{3529243}{3578575} \\
 -\frac{343662071}{4380175800} & -\frac{126531917}{438017580} &
   -\frac{86297573}{257657400} & -\frac{86297573}{257657400} &
   -\frac{3296674}{32207175} & -\frac{676938}{790075} &
   -\frac{6593348}{10735725} & -\frac{298480453}{146005860} &
   \frac{12017233}{128828700} & \frac{12017233}{128828700} &
   -\frac{519971}{425880} & -\frac{519971}{212940} & \frac{161191087}{64414350}
   & -\frac{655132571}{156434850} & -\frac{61733263}{21471450} &
   \frac{949120823}{31286970} & -\frac{34044067}{21471450} &
   -\frac{269791793}{51531480} & -\frac{694478313}{69526600} \\
 -\frac{7052429}{42942900} & -\frac{1469917}{2147145} & -\frac{3159973}{3303300}
   & -\frac{3159973}{3303300} & -\frac{26839}{63525} & -\frac{1247859}{1431430}
   & -\frac{53678}{21175} & -\frac{7463983}{3578575} & -\frac{25111}{23100} &
   -\frac{25111}{23100} & -\frac{824749}{660660} & -\frac{824749}{330330} &
   \frac{893}{210} & -\frac{19005188}{10735725} & -\frac{123237}{42350} &
   -\frac{34044067}{21471450} & \frac{156031}{3850} & -\frac{826069}{660660} &
   \frac{3453169}{7157150} \\
 -\frac{10190867}{109504395} & -\frac{3279013}{10306296} & -\frac{85739}{283140}
   & -\frac{85739}{283140} & -\frac{39443}{1981980} & -\frac{59455463}{58402344}
   & -\frac{39443}{330330} & -\frac{2909761}{1226940} & \frac{559703}{495495} &
   \frac{559703}{495495} & -\frac{538127}{396396} & -\frac{538127}{198198} &
   \frac{215345}{30492} & -\frac{522952526}{109504395} & -\frac{483403}{165165}
   & -\frac{269791793}{51531480} & -\frac{826069}{660660} &
   \frac{2222557}{56628} & -\frac{439517549}{41715960} \\
 -\frac{68043227}{3963016200} & \frac{3706671}{97337240} &
   \frac{56654959}{208579800} & \frac{56654959}{208579800} &
   \frac{6456887}{21471450} & -\frac{3567263197}{3963016200} &
   \frac{6456887}{3578575} & -\frac{95014589}{52144950} &
   \frac{12078779}{4294290} & \frac{12078779}{4294290} &
   -\frac{240667397}{292011720} & -\frac{240667397}{146005860} &
   \frac{217039103}{21471450} & -\frac{4662351652}{495377025} &
   -\frac{3529243}{3578575} & -\frac{694478313}{69526600} &
   \frac{3453169}{7157150} & -\frac{439517549}{41715960} &
   \frac{3991031128}{165125675} \\
\end{array}
\right)}$}

\end{changemargin}

%\end{changemargin}

\end{document}